\pdfoutput=1

\documentclass[%
fontsize=12pt,
paper=A4,  
parskip=half,  
DIV=calc,            
headinclude=true,    
footinclude=false,   
open=right,          
appendixprefix=true, 
bibliography=totoc,  
draft=false,      
BCOR=0mm,        
oneside
]{scrbook}  

\usepackage[utf8]{inputenc} 
\usepackage[ngerman,american]{babel}  

\usepackage{scrpage2} 

\usepackage[backend=biber, 
style=alphabetic, 
backref=true, 
natbib=true, 
hyperref=true, 
]{biblatex}  
\usepackage{pifont}

\usepackage{ifthen}

\newboolean{myaddcolophon}
\newboolean{myaddlistoftodos}
\newboolean{english_affidavit}

\usepackage{xspace}

\usepackage[usenames,dvipsnames]{xcolor}
\definecolor{DispositionColor}{RGB}{30,103,182} 

\usepackage[normalem]{ulem}

\usepackage{framed}

\usepackage{eso-pic}

\usepackage{enumitem}

\usepackage{units}

\newcommand{\mylinespread}{1.0}

\setboolean{myaddcolophon}{false}  

\setboolean{english_affidavit}{false}  

\newcommand{\myauthor}{Clemens Schiffer}  
\newcommand{\mytitle}{TV-regularized CT Reconstruction and Metal Artifact Reduction Using Inequality Constraints with Preconditioning}  
\newcommand{\mysubject}{SUBJECT}  
\newcommand{\mykeywords}{KEYWORDS}  

\newcommand{\myworktitle}{Master's Thesis}  
\newcommand{\mygrade}{Master of Science} 
\newcommand{\mystudy}{Mathematik} 
\newcommand{\myuniversity}{Karl-Franzens-Universität Graz} 
\newcommand{\myinstitute}{Institut für Mathematik und Wissenschaftliches Rechnen} 
\newcommand{\mysupervisor}{Dr.~Kristian Bredies} 
\newcommand{\myhometown}{Graz} 
\newcommand{\mysubmissionmonth}{August} 
\newcommand{\mysubmissionyear}{2018} 
\newcommand{\mysubmissiontown}{\myhometown} 

\usepackage{amsmath}
\newcounter{myclonecnt}

\newcommand{\abs}[2][]{\left|{#2}\right|_{#1}}
\newcommand{\norm}[2][]{\|{#2}\|_{#1}}

\newcommand{\scp}[3][]{\langle{#2},{#3}\rangle_{#1}}
\newcommand{\set}[2]{\{{#1} \ | \ {#2}\}}
\newcommand{\sett}[1]{\{{#1}\}}
\newcommand{\bigset}[2]{\bigl\{{#1} \ \bigl| \ {#2}\bigr\}}
\newcommand{\Bigset}[2]{\Bigl\{{#1} \ \Bigl| \ {#2}\Bigr\}}

\newcommand{\code}[1]{\texttt{#1}}
\newcommand{\ndel}[1]{\left(#1\right)}

\newcommand{\ie}{i.\,e.\ }
\newcommand{\eg}{e.\,g.\ }
\newcommand{\lsc}{lower semi-continuous}

\newcommand{\iu}{\mathrm{i}\mkern1mu}
\newcommand{\extR}{\mathbb{R}_\infty}

\newcommand{\nabh}{\nabla_h}
\newcommand{\divh}{\dive_h}
\newcommand{\laph}{\Delta_h}
\newcommand{\dd}{\; \mathrm{d}}
\newcommand{\diff}{\mathrm{d}}

\newcommand{\unull}{u^0}
\newcommand{\Om}{D}        
\newcommand{\Dm}{\Omega}   

\newcommand{\R}{R}
\newcommand{\Rh}{R_h}
\DeclareMathOperator{\dive}{div}
\DeclareMathOperator{\indicator}{\mathcal{I}}
\DeclareMathOperator{\id}{id}
\DeclareMathOperator{\range}{range}
\DeclareMathOperator{\dom}{dom}

\DeclareMathOperator{\four}{\mathcal{F}}        
\DeclareMathOperator{\fourld}{\mathcal{F}_1}    
\DeclareMathOperator{\fzdh}{\mathcal{F}_h}      
\DeclareMathOperator{\fldh}{\mathcal{F}_{h,1}}  

\DeclareMathOperator{\sinc}{sinc}
\DeclareMathOperator*{\argmin}{arg\,min}

\DeclareMathOperator{\D}{D}
\DeclareMathOperator{\TV}{TV}
\DeclareMathOperator{\BV}{BV}
\DeclareMathOperator{\epi}{epi}

\usepackage[protrusion=true,factor=900]{microtype}

\usepackage[sc,osf]{mathpazo} 
\DeclareRobustCommand{\myacro}[1]{\textsc{\lowercase{#1}}} 

\setheadsepline{.4pt}[\color{DispositionColor}]

\addtokomafont{disposition}{\color{DispositionColor}}
\addtokomafont{caption}{\color{DispositionColor}\footnotesize}
\addtokomafont{captionlabel}{\color{DispositionColor}}

\usepackage[bottom]{footmisc}

\usepackage{enumitem}
\setlist{noitemsep}   

\usepackage[babel=true,strict=true,english=american,german=guillemets]{csquotes}

\usepackage{amsmath}
\usepackage{amsthm}
\usepackage{amsfonts}
\usepackage{amssymb}
\usepackage{dsfont}
\usepackage{mathtools}
\usepackage{algorithm}
\usepackage{placeins}
\usepackage{algorithmic}
\usepackage{siunitx}
\usepackage[labelformat=simple]{subcaption}
\usepackage[T1]{fontenc}
\usepackage{floatpag}

\swapnumbers
\theoremstyle{plain}
\newtheorem{lemma}{Lemma}[section]
\newtheorem{satz}[lemma]{Theorem}

\newtheorem{kor}[lemma]{Corollary}
\theoremstyle{definition}
\newtheorem{mydef}[lemma]{Definition}
\newtheorem{bsp}[lemma]{Example}

\numberwithin{equation}{chapter}
\usepackage[]{hyperref}
\hypersetup{
pdftitle={\mytitle}, %
pdfauthor={\myauthor}, %
pdfsubject={\mysubject}, %
pdfcreator={Accomplished with: pdfLaTeX, biber, and hyperref-package. No animals, MS-EULA or BSA-rules were harmed.},
pdfproducer={\myauthor},
pdfkeywords={\mykeywords}
}

\begin{document}

\frontmatter                    


\thispagestyle{empty}  
\large  

\begin{center}
{\LARGE\bfseries\myworktitle}

\vfill

zur Erlangung des akademischen Grades eines\\[5mm]
{\LARGE\mygrade}\\[5mm]
der Studienrichtung \mystudy\\
an der \myuniversity

\vfill

über das Thema\\[10mm]

{\LARGE\bfseries\mytitle}

\vfill

eingereicht am\\
\myinstitute

Begutachter: \mysupervisor

von\\[5mm]
{\Large\bfseries\myauthor}\\[5mm]

\vfill

\mysubmissiontown, \mysubmissionmonth~\mysubmissionyear

\end{center}
\normalsize 

\newpage

\ifthenelse{\boolean{myaddcolophon}}{
  \newpage
  \thispagestyle{empty}  

  ~
  \vfill
  \mycolophon
}{}
\newpage



\newcommand{\textfield}[2]{
 \vbox{
   \hsize=#1\kern3cm\hrule\kern1ex
   \hbox to \hsize{\strut\hfil\footnotesize#2\hfil}
 }
}

\ifthenelse{\boolean{english_affidavit}}{
 \section*{Affidavit}
 I declare that I have authored this thesis independently, that I have
 not used other than the declared sources/resources, and that I have
 explicitly indicated all material which has been quoted either
 literally or by content from the sources used. The text document
 uploaded to \myacro{TUGRAZ}online is identical to the present master‘s
 thesis.

 \hbox to \hsize{\textfield{4cm}{Date}\hfil\hfil\textfield{7cm}{Signature}}
 \newpage
 }
{

}
%
%
%


\cleardoublepage
\phantomsection
\addcontentsline{toc}{chapter}{Abstract}
\chapter*{Abstract}
\label{ch_abstract}
Total variation(TV) regularization is applied to X-Ray computed tomography(CT) in an effort to reduce metal artifacts. 
Tikhonov regularization with $L^2$ data fidelity term and total variation regularization is augmented in this novel model by inequality constraints on sinogram data affected by metal to model errors caused by metal.  
The formulated problem is discretized and solved using the Chambolle-Pock algorithm.
Faster convergence is achieved using preconditioning in a Douglas-Rachford spitting method as well as Advanced Direction Method of Multipliers(ADMM).
The methods are applied to real and synthetic data demonstrating feasibility of the model to reduce metal artifacts.
Technical details of CT data used and its processing are given in the appendix. 

\textbf{Keywords:} Computed tomography, metal artifact, total variation, regularization, preconditioning

\chapter*{Abstract}
\label{ch_abstract_de}
Totalvariation(TV) Regularisierung wird auf Computertomographie(CT) angewandt um Metallartefakte zu reduzieren.
Tikhonov-Regularisierung mit einem $L^2$-Datenterm und TV-Regularisierungsterm 
wird in diesem neuen Modell um Ungleichungsnebenbedingungen erweitert 
die durch Metall verursachte Fehler in den Sinogrammdaten modellieren.
Das formulierte Problem wird diskretisiert und mithilfe des Chambolle-Pock Algorithmus gelöst.
Schnellere Konvergenz wird durch Präkonditionierung erzielt, diese wird in der Advanced Direction Mehod of Multipliers(ADMM) sowie einem Douglas-Rachford splitting Algorithmus angewandt.
Ergebnisse für reale und synthetische Daten zeigen die grundsätzliche Fähigkeit des Models Artefakte zu verringern.
Technische Details zu den verwendeten CT-Daten und deren Verarbeitung finden sich im Appendix.
\textbf{Schlagworte:} Computertomographie, Metallartefakte, Totalvariation, Regularisierung, Präkonditionierung 

\tableofcontents    

\listoffigures

\mainmatter

\chapter{Introduction} \label{ch_intro}
Two sets of problems are treated in this thesis, the first is reconstruction that is to recover a cross-section or profile of an object given only projection data of X-rays passing through the object from different directions.
This classical computed tomography(CT) problem is compounded in the context of medical imaging by the fact that X-rays are harmful to living tissue and thus it is desirable to increase  image quality of recovered images using as little projections as possible, thus making the problem more ill-posed.

The second related problem is that of metal artifacts, that is severe degradation of image quality due to metal objects present in the tissue which strongly attenuate or even block X-rays.
As many surgical procedures involve metal objects such as dental fillings, artificial hips or spine implants and the fact that magnetic resonance imaging is no alternative when ferromagnetic objects are present, efforts to improve image quality have been sought for some time and come in a variety of forms~\cite{MARoverview}.

The approach in this thesis employs variational regularization which enjoy widespread success in image processing.
Regularization can be seen as making a trade-off between data fidelity and assumed image properties, this is augmented by inequality constraints.
For the problem of removing metal artifacts sinogram data effected by metal is considered degraded and replaced by inequality constrains, reasoning that the data is unknown but expected to be greater then a certain threshold.
Many for metal artefact reduction are based around the idea of replacing this faulty sinogram data in various ways like inpainting by linearly fitting the gaps~\cite{chen2012ct} or expectation maximization~\cite{wang1996iterative}.
The model of this thesis focuses only on relative errors which are greater in projection affected by metal and as such is a general model that does not specifically take into account beam scattering or hardening. 
While the data used in this thesis comes from a dual source CT scanner was used, for the reconstruction only data from only one of the two sources was used as if it was a regular single source scanner.
Indeed some methods for metal artifact reduction are based on the fact that two sources with different energy spectra are available~\cite{MARoverview}~\cite{dualMAR}, but the purpose of this thesis is to evaluate the basic feasibility of removing artefacts with a general image processing techniques.

\section{Thesis Overview}
Section~\ref{ch_intro} provides mathematical background and consists of two parts.
In Subsection~\ref{sec_radon_trhm} a model for X-ray tomography is presented, the Radon transform $\R$ is introduced together with its adjoint $\R^\ast$ and the filtered back projection, a common way of inverting the Radon transform. 
The second part, Subsection~\ref{sec_math}, introduces the mathematical framework including convex analysis and the total variation(TV) semi-norm which are used to formulate Tikhonov regularization problems.
In Section~\ref{sec_ProbMod} a Tikhonov regularization problem is formulated with the Radon transform as the operator using a $L^2$ discrepancy and a TV regularization term resulting in the main minimization problem~\eqref{opt_unconst} the solution of which can be seen as a regularized reconstruction.
The model was already used in~\cite{sidky2012convex} and~\cite{hamalainen2014total} in particular to gain reconstructions from sparsely sampled data that is from a limited number of angles.
A constraint variant of this problem is formulated yielding a model for reducing metal artefacts. 
This idea of using inequality constraints to model beam cancellation by metal was proposed by the author in the conference proceedings paper~\cite{Schiffer}.
In Section~\ref{sec_NumSol} the involved normed spaces and operators are discretized in particular the Radon transform is discretized in a way consistent with common implementations.
This leads to the formulation of discrete problems which can then by numerically solved by using the Chambolle-Pock algorithm, this was already done in~\cite{hamalainen2014total} for the problem without constraints.
In Section~\ref{sec_Precond} faster numerical methods are sought, recently developed algorithm from~\cite{Precond2} are presented and applied to the problems.
Some options for preconditioners are presented one new, based on the fact that $\R^\ast \R$ is equal to convolution with the reciprocal of the norm.
Another which was already used in~\cite{ramani2012splitting} for a different algorithm based on the alternating direction method of multipliers(\myacro{ADMM}), in this paper application for TV regularized was mentioned but not developed, this is done in this thesis and serves as a reference.
In the results Section~\ref{sec_res} details of the implementation are given and results of numerical experiments both on synthetically phantom and real world data are presented and a conclusion is given.
Particular speedup is achieved for application to metal artefact removal.
Appendix~\ref{app_data_ext} presents some methods for processing data before the algorithms can be applied, in particular converting to parallel beam geometry.
Appendix~\ref{app_file} gives details of the file format that had to be reverse engineered for this project.


\section{Theory of the Radon transform} \label{sec_radon_trhm}
This section gives the definition of the Radon transform, which is a simple model for X-ray tomography. 
First the physical process is modelled according to~\cite{epstein}, leading to the definition of the Radon transform $R$.
Next a way of inverting $R$, the filtered back projection, is discussed which will involve the adjoint $R^\ast$.
For a classical reference see~\cite{Kak_princ_ct}.
\subsection{Modelling and definition of the Radon transform}
First the physical process of a X-ray passing through material is modelled under the following assumptions: 
\begin{enumerate}
    \item There is no refraction or diffraction so that beams travel along straight lines. 
    \item The X-rays have a single energy level in the sense that all photons have the same energy, and thus the intensity $I$ of the beam can be described with only one value.
    \item The loss of intensity a beam suffers from traveling through a material is the product of the intensity and a linear absorption coefficient $u(x)$ specific to the traversed material, this is formulated as Beer's law:
\[
\frac{\diff I}{\diff t} = - u(x) I
\]
where $t$ refers to the arc length along a straight line segment.
\end{enumerate}
Let $x_0, x_1 \in \mathbb{R}^2$ be linearly independent vectors, let $s \in \mathbb{R}$ and consider the line  $L(t) = s x_0  + t x_1$ , let the intensity along $L$ be given by
\[
  i(t) = I(s x_0  +t x_1 )
\]
the linear absorption coefficient along $L$ is
\[
  u_L(t) = u( s x_0 + t x_1 ).
\]
Then Beer's law states that
\[
    i'(t) = -u_L(t)i(t) \text{ or }  -u_L(t) = \frac{i'(t)}{i(t)}.
\]
Integrating from $t=a$ to $t=b$ gives
\begin{equation} \label{eq_rad_atun}
  - \int_a^b u_L(t) \dd t = \log(i(b)) - \log(i(a)).
\end{equation}
Therefore, the process of an X-ray passing through an object, being weakened and received by a detector is modeled as integrating the absorption coefficient $u$ over a straight line $L$.

In this model for computed tomography a radiation source, which emits parallel beams and an opposing line detector rotate around an object to be scanned in a plane.
The object to be scanned is modeled as a function $u \colon \mathbb{R}^2 \to \mathbb{R}$, that maps a point $(x,y)$ to the linear absorption coefficient $u(x,y)$ at that point, which can be thought of as a density distribution of the X-ray attenuation, referred to as radiodensity or simply density.
The function $u$ is assumed to be zero outside a region of interest $\Om$.
Consider a ray traveling along the line $L=L(s,\theta)$ where $\varphi$ is the gradient angle of $L$ and $s$ its signed distance from the origin, the \emph{offset}.
If the ray is emitted at the source with intensity $I(x_S)$ and its intensity is measured at the detector as $I(x_D)$ then by equation~\eqref{eq_rad_atun}
\begin{equation}
  \int_L u(x) \dd t = -\log(I(x_D)) + \log(I(x_S)).
\end{equation}
The value of this line integral will be set as $(\R u)(\varphi, s)$.
The operator $\R$ which maps $u$ to the function $\R u$, which is commonly interpreted as a function of angle and offset,  is called the \emph{Radon transform}, in a technical context the data $\R u$ produced is called \emph{sinogram}.
While the Radon transform can in principle be computed whenever the integrals over all such lines exist, it makes sense to define it for functions in $L^2$ since then the result will also be in $L^2$ as the following theorem shows; it also serves as a formal definition of $\R$.
\begin{satz} \label{th_radon_l2}
Let $\rho > 0, \, \Om = B_\rho(0) \subset \mathbb{R}^2, \, \Dm = (-\rho,\rho) \times ( -\frac{\pi}{2}, \frac{\pi}{2})$ and $ T = T(s) = \sqrt{ \rho^2 - s^2}$ then the operator defined for $u \in L^2(\Om)$ by
\begin{equation}
    (\R u)(s,\theta) = \int_{L(s,\theta)} u \dd s 
    = \int_{-T}^{T} u( x_0(\theta) s + x_1(\theta) t) \dd t
\end{equation}
with
\begin{equation*}
    x_0(\theta) =
    \begin{pmatrix}
        \cos(\theta) \\
        \sin(\theta)
    \end{pmatrix},
    \quad
    x_1(\theta) =
    \begin{pmatrix*}[r]
        - \sin(\theta) \\
          \cos(\theta)
    \end{pmatrix*},
\end{equation*}
defines a continuous linear operator $\R \colon L^2(\Om) \to L^2(\Dm)$ with $\norm{\R} \leq \sqrt{2 \pi \rho}$.
It is called the \emph{Radon transform}.
\end{satz}

\begin{proof}
Let $ u \in \mathcal{C}^\infty(\Om) \cap L^2(\Om)$, then
\begin{align*}
    \norm[2]{\R u}^2
    &= \int_{\Dm} \abs{ \big(\R u\big)(s,\theta)} ^2 \, \dd (s,\theta)
     = \int_{-\rho}^\rho \int_{-\frac{\pi}{2}}^{\frac{\pi}{2}}
     	\abs{ \big(\R u\big)(s,\theta)} ^2 \dd \theta \dd  s \\
    &= \int_{-\rho}^\rho \int_{-\frac{\pi}{2}}^{\frac{\pi}{2}}
            \abs{ \int_{-T}^{T} u( x_0(\theta) s + x_1(\theta) t) \dd t } ^2 \dd \theta \dd s \\
    &\leq \int_{-\rho}^\rho
          \int_{-\frac{\pi}{2}}^{\frac{\pi}{2}}
          2\rho \int_{-T}^{T}
          	\abs{ u( x_0(\theta) s + x_1(\theta) t) } ^2 \dd t  \dd \theta \dd s 	\\
    &= 2\rho \int_{-\frac{\pi}{2}}^{\frac{\pi}{2}}
    		 \int_{-\rho}^\rho
    		 \int_{-T}^{T}
    			\abs{u( x_0(\theta) s + x_1(\theta) t) } ^2  \dd t \dd s \dd \theta 	\\
    &= 2\rho \int_{-\frac{\pi}{2}}^{\frac{\pi}{2}} \int_{\Om}
    			\abs{u(x)}^2 \dd x \dd \theta
     = 2\rho \pi \norm[2]{u}^2.
\end{align*}
Here the Cauchy-Schwarz inequality was used in the third line and the transformation $\phi(t,s) = x_0(\theta) s + x_1(\theta) t$ with $ \det(\D\phi) = \det( (x_1,x_0)) = 1$ in the forth line.
As $\mathcal{C}^\infty(\Om)\cap L^2(\Om)$ is dense in $L^2(\Om)$ the operator $\R$ can be extended uniquely to $\R \colon L^2(\Om) \to L^2(\Dm)$ and the inequality shows $\norm{\R} \leq \sqrt{2 \pi \rho}$.
\end{proof}

\subsection{Filtered Backprojection} \label{ssec_fbp}
The Radon transform is invertible, one way to show this is to use a connection to the Fourier transform.
Let $u \in L^2(\Om)$ and $v =  \R u$, define for fixed $\theta \in ( -\frac{\pi}{2}, \frac{\pi}{2})$ the function $v_\theta \colon s \mapsto v(s,\theta)$ for $\abs{s} < \rho$ extended with $v_\theta(s) = 0$ for $ \abs{s} \geq \rho$.
Then the 1-D Fourier transform $\fourld$ of $v_\theta$ satisfies
\begin{align} \label{eq_four_v_theta}
    \fourld(v_\theta)(\omega)
    &= \frac{1}{\sqrt{2\pi}} \int_{\mathbb{R}} v_{\theta}(s) e^{-\iu \omega s} \dd s  \\
    &= \frac{1}{\sqrt{2\pi}} \int_{\mathbb{R}} \int_{\mathbb{R}} \tilde u ( x_0(\theta) s + x_1(\theta) t)  \, e^{-\iu \omega s} \dd t \dd s
\end{align}
where $\tilde u(x) = u(x)$ for $x \in D$ and $\tilde u(x) = 0$ otherwise.
Setting $x =  x_0(\theta) s + x_1(\theta) t$ then $ s = x \cdot x_0(\theta)$ as $(x_0, x_1)$ is orthonormal, then equation \eqref{eq_four_v_theta} becomes
\begin{equation} \label{eq_four_slice}
    \fourld(v_\theta)\ndel{\omega}
    = \frac{1}{\sqrt{2\pi}} \int_{\mathbb{R}} u(x)  \, e^{-\iu \omega x \cdot x_0} \dd x
    = \sqrt{2\pi}\four( \tilde u)(\omega x_0 ),
\end{equation}
where $\four$ denotes the 2-D Fourier transform.
This fact is often referred to as the \emph{Fourier slice theorem} and expressed in words as: projecting a function $u$ to a line and then performing a 1-D Fourier transform is the same as taking the 2-D Fourier transform of $u$ and restricting the values to that line, up to a factor of $\sqrt{2\pi}$.

Most importantly, from \eqref{eq_four_slice} it can be seen that  $\four(u)$ can be calculated from the projection data $v$, and thus using the inverse Fourier transform, $u$ itself can be reconstructed.
For this purpose let $\four(\tilde u) = \hat u$, then with the transformation $\xi = \omega x_0(\theta)$ and $\dd \xi = \abs{\omega} \dd \theta \dd \omega$  for any $x \in D$
\begin{align}
    u(x) &= \tilde u(x)
    = \frac{1}{2\pi} \int_{\mathbb{R}^2} \hat u (\xi)  e^{\iu \xi \cdot x} \dd \xi \\
    &= \frac{1}{2\pi} \int_{-\frac{\pi}{2}}^{\frac{\pi}{2}} \int_{\mathbb{R}}
        \abs{\omega} \hat u (\omega x_0(\theta))  e^{\iu \omega x_0(\theta) \cdot  x} \dd  \theta \dd  \omega \\
    &= \frac{1}{2\pi} \frac{1}{\sqrt{2\pi}}
        \int_{-\frac{\pi}{2}}^{\frac{\pi}{2}} \int_{\mathbb{R}} \abs{\omega}
        \four(v_\theta)\ndel{\omega}  e^{\iu \omega x_0(\theta) \cdot  x} \dd  \theta \dd  \omega \\
    &= \int_{-\frac{\pi}{2}}^{\frac{\pi}{2}} w(x \cdot x_0(\theta), \theta) \dd  \theta  
        \label{eq_fbp_full}
\end{align}
with
\begin{equation} \label{eq_fbp_filter}
w(s, \theta) = \frac{1}{\sqrt{2\pi}} \int_{\mathbb{R}}  \frac{\abs{\omega}}{2 \pi} \fourld(v_\theta)\ndel{\omega}  e^{\iu \omega s} \dd  \omega =
    \fourld^{-1}\ndel{ \frac{\abs{\cdot}}{2\pi} \fourld(v_\theta)\ndel{\cdot} }(s, \theta).
\end{equation}
Thus a way to invert the Radon transform is found, which is called the \emph{filtered back-projection}, where calculating $w$ in \eqref{eq_fbp_filter} is called filtering, or applying the Ram-Lak filter and the right hand side of \eqref{eq_fbp_full} the linear back-projection of $w$.
In fact this linear back-projection is the adjoint of the Radon transform as the following lemma shows.
\begin{lemma} \label{lem_lin_back_proj}
The adjoint of the Radon transform $\R^\ast \colon L^2(\Dm) \to L^2(\Om)$ maps $w\in L^2(\Dm)$ to $\R^\ast w$  given by
\begin{equation*}
  (\R^\ast w)( x) = \int_{-\frac{\pi}{2}}^{\frac{\pi}{2}} w(x \cdot x_0(\theta), \theta) \dd  \theta.
\end{equation*}
\end{lemma}
\begin{proof}
Let $u \in L^2(\Om) \cap \mathcal{C}^\infty(\Om)$ and $w \in L^2(\Dm)\cap \mathcal{C}^\infty(\Dm)$ then
\begin{align*}
    \scp[L^2(\Dm)]{\R u}{w}
    &= \int_{-\frac{\pi}{2}}^{\frac{\pi}{2}} \int_{-\rho}^{\rho}
         	 w(s, \theta) \int_{-T}^{T} u( x_0(\theta) s + x_1(\theta) t) \dd t \dd s \dd \theta \\
    &= \int_{-\frac{\pi}{2}}^{\frac{\pi}{2}} \int_{\mathbb{R}^2}
        u(x) w(x \cdot x_0(\theta), \theta) \dd x \dd \theta
    = \scp[L^2(\Om)]{u}{\R^\ast w}.
\end{align*}
Since $L^2(\Om) \cap \mathcal{C}^\infty(\Om)$ is dense in $L^2(\Om)$ the same holds for $u,w \in L^2(\Om)$.
\end{proof}
The filtered back-projection is used is practice to numerically reconstruct the image, it has the advantage of being very fast but the disadvantage of amplifying any noise that might be present in the projection data.
One remedy for this is using a different filter \eg instead of multiplying with $\abs{\omega}$ in \eqref{eq_fbp_filter} one can multiply by $\abs{\omega} \sinc(\omega \frac{\pi}{2})$ thus suppressing high frequency modes, which should mostly correspond to noise, this is called the Shepp-Logan filter.
Nonetheless, iterative reconstruction algorithms have recently gained more and more interest as the provide increased insensitivity to noise.
Additionally, these methods are more flexible and can be used when only a small number of projections are available -- \eg in an effort to save on radiation dose -- or when additional assumptions are put into the model \eg to reduce artifacts.

Performing the Radon transform followed by back projection gives the same result as convolution with the reciprocal of the norm, as formulated in the following lemma.
\begin{lemma} \label{lem_RR_uk}
Let $u \in L^2(\Om)$ then for $z \in \Om$ it holds that
\begin{equation*}
(\R^\ast \R u)(z)  = \int_{\Om} u(\tilde{z}) \frac{1}{\norm{ z-\tilde z}} \dd \tilde{z}.
\end{equation*}
\end{lemma}
\begin{proof}
Let $u \in L^2(\Om) \cap \mathcal{C}^\infty(\Om)$ for fixed $z\in \Om$ calculate
\begin{align*}
(\R^\ast \R u)(z)
    &= \int_{-\frac{\pi}{2}}^{\frac{\pi}{2}} (Ru)(z \cdot x_0(\theta),\theta) \dd  \theta
    = \int_{-\frac{\pi}{2}}^{\frac{\pi}{2}} \int_{-T}^{T}
    		u(z \cdot x_0(\theta) x_0(\theta) + t x_1(\theta))  \dd t \dd  \theta.
\end{align*}
Use the transformation $\tilde z = z \cdot x_0(\theta) x_0(\theta) + t x_1(\theta)$.
Specifically let
\begin{equation}
M_z =
\Bigset{ (t,\theta)}{
\theta \in (-\frac{\pi}{2}, \frac{\pi}{2} ],
0<\abs{t} < \sqrt{ \rho^2 - (z \cdot x_0(\theta))^2}
}
\end{equation}
then $ \varphi \colon M_z \to \Om \setminus \sett{z}, (t,\theta) \mapsto \tilde z = (\tilde x, \tilde y)$ with
\begin{align*}
\tilde x &= z \cdot x_0(\theta) \cos(\theta) - t \sin(\theta) \\
\tilde y &= z \cdot x_0(\theta) \sin(\theta) + t \cos(\theta)
\end{align*}
is a coordinate transform.
Firstly $\varphi$ does indeed map to $\Om \setminus \sett{z}$ as $\tilde x^2 + \tilde y^2 =  (z \cdot x_0(\theta))^2 + t^2 < \rho^2$.
In order to show that $\varphi$ is surjective let $\tilde z \in D \setminus \sett{z}$ be arbitrary then for any $\theta$ the vectors $x_0(\theta),x_1(\theta)$ are a orthonormal basis of $\mathbb{R}^2$ and so $\tilde z = \tilde z \cdot x_0(\theta) x_0(\theta) + \tilde z \cdot x_1(\theta) x_1(\theta)$.
Choosing the unique $\theta \in (-\frac{\pi}{2}, \frac{\pi}{2} ]$ such that $(\tilde z - z) \cdot x_0(\theta) = 0$ as well as $t = \tilde z \cdot x_1(\theta)$ gives $\varphi((t,\theta)) = \tilde z$ with $(t,\theta) \in M_z$ since $t^2 = \abs{ \tilde z - \tilde z \cdot x_0(\theta) x_0(\theta)}^2 = \abs{\tilde z}^2 - 2 (z \cdot x_0(\theta))^2 +  (z \cdot x_0(\theta))^2 < \rho^2 - z \cdot x_0(\theta)$.
Injectivity is clear as multiplying $\tilde z = z \cdot x_0(\theta) x_0(\theta) + t x_1(\theta)$ by $x_0(\theta)$ gives $ z \cdot x_0(\theta) = \tilde z \cdot x_0(\theta)$ which is true for a unique $\theta \in (-\frac{\pi}{2}, \frac{\pi}{2} ]$, multiplying by $x_1(\theta)$ gives necessarily $t = \tilde z \cdot x_1(\theta)$.

Clearly $\varphi$ is differentiable with the Jacobian
\begin{equation}
J = \begin{pmatrix}
z \cdot x'_0(\theta) \cos(\theta) - z \cdot x_0(\theta) \sin(\theta) - t \cos(\theta) &-\sin(\theta) \\
z \cdot x'_0(\theta) \sin(\theta) + z \cdot x_0(\theta) \cos(\theta) - t \sin(\theta) & \cos(\theta) \\
\end{pmatrix}
\end{equation}
which has the determinant
\begin{equation}
\det(J)=z \cdot x'_0(\theta) - t = z \cdot x_1(\theta) - t.
\end{equation}
Since $(\tilde z - z)\cdot x_0(\theta) = z \cdot x_0(\theta) -z \cdot x_0(\theta) =0 $ it holds that $ \norm{ z - \tilde z } = \abs{ (z - \tilde z)\cdot x_1(\theta)} = \abs{ z \cdot x_1(\theta) - t}$ and therefore $\abs{\det(J)} = \norm{z- \tilde{z}}$.
Performing this coordinate transform yields
\begin{equation}
(\R^\ast \R u)(z)
    =  \int_{\Om} u(\tilde{z}) \frac{1}{\norm{ z-\tilde z}} \dd \tilde{z}
    = (u \ast \frac{1}{\norm{\cdot}} )(z).
\end{equation}
\end{proof}

\section{Mathematical Framework} \label{sec_math}
This section presents the mathematical framework for the methods used in this thesis, it is mostly a summary of the content of~\cite{KB} which is needed here.  
The overall goal will be the solution of optimization problems which are of the form $\min_{u \in X} \Phi( u ) + \lambda \Psi(u)$, where $\Phi$ is the data fidelity or discrepancy term measuring how close $u$ is to given data $\unull$ and $\Psi$ is the penalty or regularization term which favors desired properties of the solution.
The regularization parameter $\lambda> 0 $ balances these two  objectives.
Here $\Phi$ and $\Psi$ are assumed to be convex but possibly not smooth.
To give the context of the methods a brief introduction to convex analysis is given, this also covers the Fenchel conjugate, Rockafellar duality and saddle point problems.
Then a specific regularization term - the total variation semi-norm is presented and its role in image processing as a tool for edge preserving noise removal is discussed.
Finally the previous two sections are combined in a discussion of the Tikhonov regularization where the data fidelity term will not compare $u$ directly to $\unull$, as would be the case \eg with $\Phi(u) = \norm{u-\unull}$, but rather after applying an operator $A$ as in  $\Phi(u) = \norm{Au - \unull}$ since this will be the situation in Section \ref{sec_ProbMod}.
\subsection{Convex Analysis}
\begin{mydef}
Call the set $\extR = \mathbb{R} \cup \{\infty\}$ the \emph{extended reals}.
Let $X$ be a Banach space and $f \colon X \to \extR$ then
\begin{equation}
\epi(f) = \set{ (x,t) \in X \times \mathbb{R} }{ f(x) \leq t}
\end{equation}
is called the \emph{epigraph} of $f$.
Note that $t= \infty$ is not allowed and thus $f \equiv \infty$ if and only if $\epi(f) = \emptyset$, if this is not the case $f$ is called \emph{proper}.
The function $f$ is called \emph{convex} if its epigraph is a convex set, and it is called \emph{(weakly) lower semi-continuous} if its epigraph is (weakly) closed.
This definition of convexity is equivalent to $f(\lambda x +(1-\lambda)y) \leq \lambda f(x) + (1-\lambda)f(y)$ for all $x,y \in X$ and $\lambda \in (0,1)$, additionally if the inequality is strict whenever $x \neq y$ then $f$ is called \emph{strictly convex}.
\end{mydef}

\begin{lemma}
A function $f$ is (weakly) \lsc{} if and only if for every sequence $(u_n)$ (weakly) converging to $u$ in $X$ it holds that $f(u) \leq \liminf_{n \to \infty} f(u_n)$.
\end{lemma}
\begin{proof}
Let $\epi(f)$ be closed and let $u_n \to u$.
Let $\bar f = \liminf_{n \to \infty} f(u_n)$ and $(u_{n_k})$ be a subsequence of $(u_n)$ such that $f(u_{n_k}) \to \bar f$.
Since $(u_n, f(u_n)) \in \epi(f)$ and $(u_{n_k}, f(u_{n_k})) \to (u, \bar f)$ it follows that $(u, \bar f) \in \epi(f)$, since $\epi(f)$ is closed by assumption.
Thus $f(u) \leq \bar f$.

For the implication from left to right, let $(u_n,t_n) \in \epi(f)$ with $( u_n, t_n) \to (u,t)$ it needs to be shown that $(u,t) \in \epi(f)$.
Using the assumption and $f(u_n) \leq t_n$ it follows that $f(u) \leq \liminf_{n \to \infty} f(u_n) \leq \liminf_{n \to \infty} t_n = t$ and thus $(u,t) \in \epi(f)$.
The proof for weak convergence is completely analogous.
\end{proof}

A convex function $f$ is weakly \lsc{} if and only if it is \lsc{} since its epigraph is a convex set, which is closed if and only if it is weakly closed.

The following standard argument guarantees the existence of a minimizer of the optimization problem $\min_{u \in X} F(u)$.
\begin{lemma}[Direct method] \label{lem_direct_method}
Let $X$ be a reflexive Banach space and $f \colon X \to \mathbb{R}_\infty$ be convex, proper, weakly \lsc{} and bounded from below.
Furthermore let $f$ be radially unbounded, \ie $\norm{u_n} \to \infty$ implies $f(u_n) \to \infty$.
Then a solution $u^\ast$ of $\min_{u \in X} f(u)$ exists.
\end{lemma}
\begin{proof}
Let $ m = \inf_{u \in X } f(u)$ then $m > - \infty$ as $f$ is bounded from below.
Choose a sequence $ (u_n) $ with $f( u_n) \to m$.
This sequence must be bounded as otherwise radially unboundedness of $f$ would imply $f(u_n) \to \infty$, which is not the case as $(u_n)$ is a minimizing sequence.
Since $X$ is a reflexive Banach space the Eberlein-\v{S}mulian theorem guarantees that this bounded sequence must have a weakly convergent subsequence ${ u_{n_k} } \rightharpoonup u^\ast$.
Since $f$ is weakly \lsc{} it holds that $f(u^\ast) \leq \liminf_{k \to \infty} f( u_{n_k}) = m$ and therefore $f(u^\ast) = m = \min_{u \in X} f(u)$.
\end{proof}

\subsection{Fenchel Duality}
\begin{mydef}
Let $X$ be a normed vector space and $X^\ast$ its dual and let ${\scp{\cdot}{\cdot}}_{X,  X^\ast}=\scp{\cdot}{\cdot}$ denote the duality pairing.
Let $f \colon X \to \extR$ be proper, then the function $f^\ast \colon X^\ast \to \extR$ defined by
\begin{equation*}
f^\ast( \xi ) = \sup_{x \in X} \, \scp{\xi}{x} - f(x)
\end{equation*}
is called the \emph{convex conjugate} of $f$.
Note that the assumption of $f$ being proper ensures that $f^\ast$ is well defined, since then for fixed $\xi$ the expression $\scp{\xi}{x} - f(x)$ cannot be $-\infty$ for all $x$, it can however be unbounded from above leading to $f^\ast(\xi) = \infty$.
\end{mydef}

\begin{bsp}
Calculate the conjugate of the function $f\colon X \to \mathbb{R}, x \mapsto f(x) = \lambda\norm{x}$ for $ \lambda > 0 $, for $\xi \in X^\ast$:
\begin{align*}
f^\ast( \xi ) &= \sup_{x \in X} \, \scp{\xi}{x} - \lambda \norm{x}
= \sup_{\mu > 0} \sup_{ \norm{x} = 1 }      \, \scp{\xi}{\mu x} - \lambda \mu \\
&=\sup_{\mu > 0} \mu (\sup_{ \norm{x} = 1 }  \, \scp{\xi}{x}     - \lambda)
=\sup_{\mu > 0} \mu   (\norm[X^\ast]{\xi} - \lambda).
\end{align*}
Thus  $f^\ast(\xi) = \infty$ for $\norm[X^\ast]{\xi}  > \lambda$ and $f^\ast(\xi) = 0$ for $\norm[X^\ast]{\xi} \leq \lambda$.
This can be compactly written as $f^\ast= \indicator_{ B_\lambda^\ast}$, where $\indicator_M$ denotes the indicator function of a set $M$ defined as
\begin{equation*}
\indicator_M(x)
=  \begin{cases}
    0      &\text{if } x \in M  \\
    \infty &\text{if } x \notin M. \\
    \end{cases}
\end{equation*}
\end{bsp}
\begin{mydef}
Let $f \colon X \to \extR$ and $x \in X$ then the set
\begin{equation*}
\partial f (x) = \set{\xi \in X^\ast}{f(x) + \scp{\xi}{y-x} \leq f(y), \forall y \in X}.
\end{equation*}
is called the \emph{subdifferential} of $f$ at $x$.
\end{mydef}
The subdifferential is indeed a generalization of the classical differential, since it can be shown that if  $f$ is differentiable at a point $x$ with the derivative $f'(x)$ then $\partial f(x) = \{ f'(x) \}$.
The fact that for a convex functions $f$  which is differentiable the tangent at a point $x$ lies below the graph of $f$, \ie  $ f(x) + \scp{f'(x)}{y- x} \leq f(y) $ for all $y$, is generalized by the subdifferential which is the collection of all slopes for which this is the case.
Furthermore the notion of critical points, that is $f'(x) = 0$ as a necessary condition for $x$ being a local minimum, is generalized by the subdifferential in the following way.
It holds that $ u^\ast \in \argmin_{u\in X} f(u)$ if and only if $0 \in \partial f(u^\ast)$, this follows immediately from the definition and is true also for non-convex functions.

The following lemma, known as Fenchel's equality, provides a useful connection between the subdifferential of a function and its conjugate.
\begin{lemma} \label{lem_Fenchel_eq}
Let $f \colon X \to \extR$ be proper, convex and \lsc{} then for $x \in X$, $\xi \in X^\ast$:
\begin{equation}
\xi \in \partial f(x) \Leftrightarrow f(x) + f^\ast(\xi) = \scp{\xi}{x}
\Leftrightarrow x \in \partial f^\ast(\xi).
\end{equation}
\end{lemma}
From now on let $X$ be a Hilbert space and identify $X^\ast = X$ via the Riesz representation.
In particular $\scp{\cdot }{\cdot }$ will now denote both duality pairing and the scalar product on $X$.
The subdifferential can be thought of as a set valued operator denote by $2^X$ the set of all subsets of $X$ then $\partial f \colon X \to 2^X, x \mapsto \partial f(x)$, sometimes also written as $\partial f \colon X \rightrightarrows X$.
Such a set valued operator $M \colon X \to 2^X$ is called \emph{singleton} if it maps every $x \in X$ to a singleton set, the set valued inverse is defined by $M^{-1} \colon X \to 2^X, u \mapsto \{x \in X | u \in Mx  \}$.
Solving the optimization problem $\min_{u \in X} f(u)$ is equivalent to solving the operator equation $0 \in \partial f(x)$, or equivalently $x \in x + \sigma \partial f(x)$ for any $\sigma > 0$ which means that $x$ is a fixed point of the operator $\id + \sigma \partial f$.
It can be shown that under if $f$ is proper, convex and \lsc{} then the operator $\id + \sigma \partial f$ has a singleton inverse.
\begin{mydef}
Let $X$ be a Hilbert space and $f \colon X \to \extR$ be proper, convex and \lsc{} and $\sigma >0$  then the mapping $(\id + \sigma \partial f)^{-1}$ which maps $\unull$ to the unique solution of
\[
\min_{u \in X} \frac{1}{2}\norm[X]{u - \unull}^2 + \sigma f(u)
\]
is called the \emph{resolvent} of $\partial f$ to parameter $\sigma$.
\end{mydef}
Let $x,y \in X$ and  $\xi \in \partial f(x)$, $ \eta \in \partial f(y)$ then
\begin{align}
  f(x) + \scp{\xi}{y-x}  \leq f(y) \\
  f(y) + \scp{\eta}{x-y} \leq f(x)
\end{align}
adding both together gives
\begin{equation}
  \scp{ \xi - \eta}{x - y} \geq 0
\end{equation}
in general a operator $M$ with this property is called \emph{monotone}.
If in addition to that $ M + \id $ is surjective the operator is called \emph{maximal monotone}.
Generally in trying to solve $0 \in M(u)$ for a maximal monotone operator $M$ one can show that for given $z$ and $\sigma >0$ there exists a unique $u$ such that $z \in (\id + \sigma M) u$.
This means that the operator $P = (\id + \sigma M)^{-1}$ is singleton, $P$ is called the proximal mapping of $\sigma M$.
The fact that a  solution of $0 \in M(u)$ is a fixed point of $P$ leads to the algorithm of starting with a value $z^0 \in X$ and performing the iteration $z^{k+1} = P( z^{k})$, it is called the \emph{proximal point algorithm}, which under certain conditions generates a sequence weakly converging to a solution $z^\ast$ of $0 \in M(z^\ast)$, if such a solution exists.
\begin{lemma} \label{lem_moreau_id}
Let $M$ be a maximal monotone operator then
\begin{equation*}
(\id + M)^{-1} + (\id + M^{-1})^{-1} = \id
\end{equation*}
\end{lemma}
Lemma \ref{lem_Fenchel_eq} states that for $M = \partial f$ the inverse is $(\partial f)^{-1} = \partial f^\ast$ which combined with the previous lemma  gives Moreau's identity
\[
(\id + \partial f)^{-1} + (\id + \partial f^\ast)^{-1} = \id.
\]

Let $X$ and $Y$ be Hilbert spaces, consider the problem
\begin{equation} \label{opt_abst_primal}
\min_{u\in X} F(u) + G(Ku),
\end{equation}
with  $F \colon X \to \extR$ and $G \colon Y \to \extR$ both proper, convex and \lsc{} functions and $K \colon X \to Y$ a linear and continuous operator.
If the primal problem \eqref{opt_abst_primal} has a solution $u^\ast$ and there exists a $\unull \in X$ such that $F(\unull) < \infty$ and $G(K\unull) < \infty$ and $G$ is continuous at the point $K\unull$ then
\begin{equation} \label{eq_prim_eq_dual}
\min_{u\in X} F(u) + G(Ku) = \max_{w \in Y^\ast} -F^\ast(-K^\ast w) - G^\ast(w).
\end{equation}
Where
\begin{equation} \label{opt_abst_dual}
\max_{w \in Y^\ast} -F^\ast(-K^\ast w) - G^\ast(w).
\end{equation}
is called the dual problem, thus \eqref{eq_prim_eq_dual} means the optimal values of the primal and dual problem coincide.
On the other hand let \eqref{eq_prim_eq_dual} hold, then $u^\ast \in X$ and  $w^\ast \in Y^\ast$ are solutions of the primal and dual problem respectively if and only if
\begin{equation*}
-K^\ast w^\ast \in \partial F(u^\ast), \quad w^\ast \in \partial G(K u^\ast).
\end{equation*}

Solving the primal and dual equation simultaneously can be interpreted as finding saddle-points of the function $L: \dom(F) \times \dom(G^\ast) \to \mathbb{R}$ defined by
\begin{equation} \label{opt_abst_saddle}
L(u,w) = \scp{w}{Ku} + F(u) - G^\ast(w).
\end{equation}

\subsection{Total variation}
The minimization problems discussed in this thesis will be of the form
\begin{equation} \label{eq_abst_regul}
\min_{u \in X} \Phi(u) + \lambda \Psi(u).
\end{equation}
Where $\Phi$ is the data fidelity term and $\Psi$ is the regularization term.
One particularly successful penalty term $\Psi$ is the total variation(TV) semi-norm.
It has been used by Rudin, Osher, Fatemi in the classical ROF model~\cite{rudin1992nonlinear} to remove noise from images, in that case $\Phi(u) = \norm[2]{u - \unull}^2$.
The main advantage of TV is that it penalizes the perimeter of the level sets of $u$ thus penalizing many small jumps(noise) while allowing for sharp edges, therefore it is particularly well suited for the reconstruction of piecewise constant images.

Let $\Om$ be a bounded Lipschitz domain subset of $\mathbb{R}^d$.
For $u \in L^1_{loc}(\Om)$ a measure $\mu \in \mathfrak{M}(\Om, \mathbb{R}^d )$ is called the \emph{weak gradient} of $u$ if for every $\varphi \in \mathcal{D}(\Om, \mathbb{R}^d)$
\begin{equation*}
\int_\Om u \dive \varphi \dd x = - \int_\Om \varphi \dd \mu.
\end{equation*}
If such a $\mu$ exists call $\TV(u) = \norm[\mathfrak{M}]{\mu} = \norm[\mathfrak{M}]{\nabla u}$ the total variation of $u$.
It then holds that
\begin{equation}
\TV(u) = \sup \Bigset{ \int_\Om u \dive \varphi \, \dd x }
{ \varphi \in \mathcal{D}(\Om, \mathbb{R}^d), \norm[\infty]{\varphi} \leq 1}.
\end{equation}
It can be shown that
\begin{equation}
\BV(\Om) = \bigset{ u \in L^1(\Om)}
{ \nabla u \in \mathfrak{M}(\Om, \mathbb{R}^d)},
\end{equation}
the space of functions of bounded variation, with the norm $\norm[BV]{\cdot} = \norm[1]{\cdot} + \TV(\cdot)$ is a Banach space.
\begin{lemma} \label{lem_lq_tv}
Let $\varphi \colon [0, \infty) \to \extR$ be proper, convex, \lsc{} and increasing, then every function of the form
\begin{equation*}
\Psi(u) =
\begin{cases}
\varphi( \TV(u)), & \text{if } u \in \BV(\Om) \\
\infty , & \text{else }
\end{cases}
\end{equation*}
is proper, convex and \lsc{} on $L^q(\Om)$ for $q \in [1, \infty)$.
\end{lemma}
\begin{lemma} \label{lem_PoinWirt_TV}
There exists a continuous embedding $\BV(\Om) \hookrightarrow L^q(\Om)$ whenever $1 \leq q \leq \frac{d}{d-1}$.
Furthermore a Poincar\'{e}-Wirtinger inequality holds, \ie there exists a $c >0$ such that
\begin{equation} \label{eq_PoinWirt_TV}
\norm[q]{u - \frac{1}{\abs{\Om}} \int_{\Om} u \dd x} \leq c \norm[\mathfrak{M}]{\nabla u} = c\TV(u).
\end{equation}
\end{lemma}


\begin{kor} \label{kor_direct_method}
Let $X$ be a reflexive Banach space and $\Phi\colon X \to \extR $ and $\Psi\colon X \to \extR$ be convex, weakly \lsc{} and bounded from below such that $F = \Phi + \lambda \Psi$ is proper and radially unbounded then $\min_{ u \in X} F(u)$ has a minimizer.
\end{kor}
\begin{proof}
The direct method of Lemma~\ref{lem_direct_method} can be applied to $F$, which is convex as the sum of convex functions and \lsc{} as the sum of to functions with this property.
\end{proof}

\subsection{Tikhonov regularization} \label{ssec_tichi}
Consider an optimization problem of the form
\begin{equation} \label{opt_tikh_TV}
\min_{u \in L^q(\Om)} \norm[Y]{ A u - \unull }^r + \lambda \TV(u)
\end{equation}
with $Y$ a Banach space and $A \in \mathcal{L}( L^q(\Om), Y)$, \ie $A$ linear and bounded.
Then $\Phi \colon L^q(\Om) \to \extR, \Phi(u) = \norm[Y]{ A u - \unull }^r $ is proper, convex, weakly \lsc{} and bounded from below for any $r \in [1, \infty)$ and finite for every $u$.
Let $\Psi \colon L^q(\Om) \to \extR, \Psi(u) = \TV(u)$ in the sense of Lemma \ref{lem_lq_tv}, \ie $\Psi(u) = \infty$ if $ u \notin \BV(\Om)$.
Then $ F = \Phi + \lambda \Psi $ is proper, convex, weakly \lsc{} and bounded from below on the reflexive Banach space $L^q(\Om)$.
In order to show the existence of a minimizer using Lemma~\ref{kor_direct_method} it still needs to be shown that the combined term is radially unbounded, this gives a restriction on $q$ as well as a condition for the operator $A$.
\begin{lemma} \label{lem_combined_term_coercive}
Let $1 < q \leq \frac{d}{d-1}$ and $A \in \mathcal{L}( L^q(\Om), Y)$ be a linear mapping that does not map constant function to zero then with $\Psi(u) = \TV(u)$ and $\Phi(u) = \norm[Y]{ A u - \unull }^r $ the combined term $ F = \Phi + \lambda \Psi $ is radially unbounded.
Under these assumption the problem \ref{opt_tikh_TV} has as solution.
\end{lemma}
\begin{proof}
For $u \in L^q(\Om)$ set
\begin{equation}
P u = \frac{1}{\abs{\Om}} \int_{\Om} u \dd x,
\end{equation}
in the sense that $Pu$ is the constant function then $u =  u - Pu + Pu$.
Let  $(u_n)$ be sequence with $\norm{u_n} \to \infty$, since $\norm{u_n} \leq \norm{u_n - Pu_n} + \norm{Pu_n}$ either $ \norm{u_n - Pu_n} \to \infty$ or $\norm{u_n - Pu_n}$ is bounded and $\norm{Pu_n} \to \infty$ must hold.
In the first case it follows from Lemma~\ref{lem_PoinWirt_TV} that $\TV(u_n) \to \infty$, and thus $F(u_n) \to \infty$.
In the second case let $\mathds{1}$ be the constant function with value $1$, then $P u_n = c_n \mathds{1}$ for a sequence of real values $(c_n)$ with $\abs{c_n} \to \infty$.
Since it is assumed that $A$ does not map constant functions to zero $A \mathds{1} \neq 0$ it follows that $\norm[Y]{A P u_n } = \norm[Y]{ c_n A \mathds{1} } = \abs{c_n} \norm[Y]{ A \mathds{1}} \to \infty$ and thus also $\norm[Y]{A Pu_n - \unull} \to \infty$.
Since $ \norm[Y]{A u_n - \unull} \geq \norm[Y]{A P u_n - \unull} - \norm[Y]{A (u_n - P u_n)}$ and by assumption $\norm{u_n - P u_n}$ is bounded it follows that $\norm[Y]{A u_n - \unull} \to \infty$ which implies $\Phi(u_n) \to \infty$.
Therefore $F$ is radially unbounded in either case and the direct method of Lemma~\ref{lem_direct_method} can be applied to guarantee the existence of a minimizer of~\ref{opt_tikh_TV}.
\end{proof}

\chapter{Problem and Mathematical Model} \label{sec_ProbMod}
\section{Regularized Reconstruction}
Rather then directly invert the Radon transform as discussed in Section~\ref{ssec_fbp} which is prone to amplifying noise, one can use regularized reconstruction by solving Tikhonov optimization problems as discussed in Section~\ref{ssec_tichi}.
Let $\unull \in L^2(\Dm)$ be the given data.
Let $R \colon L^2(\Om) \to L^2(\Dm)$ be the Radon transform as discussed in Theorem \ref{th_radon_l2} which showed that $R$ is linear and continuous.
Consider the problem
\begin{equation} \label{opt_unconst}
\min_{ u \in X  } \quad \frac{1}{2} \norm[2]{ \R u - \unull}^2
+ \lambda \TV(u) \tag{$P$}
\end{equation}
where $\lambda >0 $ is  a regularization parameter.
Solving \eqref{opt_unconst} will yield  a regularized reconstruction.
This functional has been used in~\cite{sidky2012convex}~\cite{hamalainen2014total}, for this very porpose.
As applying $R$ to constant functions gives positive multiples of $(R \mathds{1})(s,\theta) = 2 \sqrt{\rho^2 - s^2} $ which does not result in zero so that Lemma~\ref{lem_combined_term_coercive} can be applied to yield that Problem~\eqref{opt_unconst} has a solution.
Since $\norm[2]{\cdot}$ is strictly convex the entire functional is strictly convex and thus the minimizer is unique.


\section{Constrained Problem}
Now a model for metal artifacts is devoloped and used to augment the regularization problem with inequality contraints.  
If a ray passes through a very dense area, and is attenuated so strongly that the detector cannot differentiate its signal from noise, \ie the received signal is unknown but lower than a certain threshold $\varepsilon$, then $I(x_D) < \varepsilon$.
The intensity at the source  $I(x_S)$ is considered constant, and therefore $(\R f)(\varphi, s) > -\log(\varepsilon) + \log(I(x_S)) = \colon C$.
Thus, this process of beam cancellation is represented by ``capping'' the sinogram,
\ie defining a certain threshold $C$ and considering the values for each point in the sinogram that exceeds $C$
as unknown, but at least as great as the threshold.
The domain of the sinogram is denoted as $\Dm$, while the portion where the threshold is exceeded is denoted as $\Dm_0$.
The capped sinogram is called $\unull$ and considered to be the given data.
It is assumed - as in total variation denoising - that $f$ has a certain spatial structure for the purpose of this thesis it is assumed
to be piecewise constant and therefore to admit a low total variation semi-norm.
In trying to reconstruct $f$ look for an approximation $u$ that has also low TV-norm and
for which an application of the Radon transform will produce a sinogram that is similar to that of $f$ in $\Dm \setminus \Dm_0$
and also has values greater or equal than $C$ in $\Dm_0$.
This leads to the optimization problem

\begin{equation} \label{opt_const}
\begin{cases}
\min_{ u \in X  } \quad \frac{1}{2} \norm[2]{ \R u - \unull}^2
+ \lambda \TV(u) \\
\text{s.t. }  (\R u) \mid_{\Dm_0} \geq C. \tag{$C$}
\end{cases}
\end{equation}
Here $\lambda > 0$ serves as a balancing parameter between
regularization via the total variation semi-norm, and the discrepancy
between $\R u$ and $\unull$, which is measured on $\Dm \setminus \Dm_0$
\ie only on sinogram data that is considered to be correct.

Using the indicator function \eqref{opt_const} can be reformulated as
an unconstrained problem
\begin{equation}	\label{opt_const_I}
  \min_{u\in X} \quad \frac{1}{2} \norm[2]{ \R u - \unull}^2
  + \lambda \TV(u) + \indicator_{ \{\R u\mid_{\Dm_0} \geq C \}}(u).
\end{equation}
%

The set $ K = \set{u \in L^2(\Om)}{\R u \mid_{\Dm_0} \geq C \text{ almost everywhere}}$ is convex and closed which implies that $\indicator_K$ is convex and \lsc.
Assuming that
\begin{equation}
\Dm_0 \subset [-\rho + \varepsilon, \rho - \varepsilon] \times [-\frac{\pi}{2},\frac{\pi}{2}]
\end{equation}
for some $\varepsilon >0 $ guarantees that $\indicator_{ \{\R u\mid_{\Dm_0} \geq C \}}$ is proper as then choosing $ u \equiv m$ constant for $ 0< s \leq \rho - \varepsilon$ will result in
\begin{equation*}
(\R u)(s, \phi) =
\int_{-\sqrt{\rho^2 - s^2}}^{\sqrt{\rho^2 - s^2}} u( s x_0 + t x_1) \dd t
= 2 \sqrt{\rho^2 - s^2} m
\geq 2 m \varepsilon \sqrt{\frac{\rho}{\varepsilon} - 1}
\geq C
\end{equation*}
if $m \geq \frac{1}{2} (\rho \varepsilon - \varepsilon^2)^{-\frac{1}{2}}$ is chosen.
Thus the terms $\Phi(u) = \frac{1}{2} \norm[2]{ \R u - \unull}^2 + \indicator_{ \{\R u\mid_{\Dm_0} \geq C \}}(u)$
and $\Psi(u)  = \TV(u) $ are convex and \lsc{} on $X$ and sufficiently large constant functions $u$ will satisfy the constraints thus $\Phi(u)$ is finite as well as $\Psi(u) = 0$ thus the combined term $F = \Phi + \lambda \Psi$ is proper.
In Lemma~\ref{lem_combined_term_coercive} it was shown that
that $F$ without the indicator function is radially unbounded so $F$ with the constraints must be radially unbounded as well.
Thus Lemma~\ref{lem_combined_term_coercive} can again be used to yield the existence and uniqueness of the minimizer.
\section{Hard vs Soft Constraints} \label{ssec_HC}
The parameter $\lambda$ is not needed if one enforces $ \R u = \unull$ in $\Dm \setminus \Dm_0$ as a hard constraint.
Then, the minimization problem becomes

\begin{equation} \label{opt_main_hc}
  \begin{cases}
  \min_{ \substack {u\in X } } \quad  \TV(u) \\
  \text{s.t. } \R u\mid_{\Dm_0} \geq C \quad \text{and } \quad \R u\mid_{\Dm \setminus \Dm_0} = \unull
  \end{cases}
\end{equation}
which is equivalent to

\begin{equation} \label{opt_main_hc_unconstr}
  \min_{ \substack {u\in X } } \quad  \TV(u)
  +  \indicator_{ \{\R u\mid_{\Dm_0} \geq C \}}(u) + \indicator_{ \{\R u\mid_{\Dm \setminus \Dm_0} = \unull\}}(u).
\end{equation}

\chapter{Numerical Solution} \label{sec_NumSol}
In this section the unconstrained problem \eqref{opt_unconst} and the constrained problem \eqref{opt_const} are discretized and numerical algorithms applied to the discrete problems.
The domain and image are discretized as an array of pixels and finite differences are used to represent  the TV seminorm similar to the image processing problems of denoising and deblurring in ~\cite{KB} and~\cite{Pock_dual}.
A discrete version of the Radon transform and its properties are presented, which is consistent with existing implementations.
In order to numerically solve the problems the primal dual algorithm of  Chambolle and Pock~\cite{Pock_dual} is used.
For the unconstrained problem this total variation regularization for tomography was already used in~\cite{sidky2012convex} and~\cite{hamalainen2014total}, as a reconstruction method.
\section{Discretization}
First the image domain is discretized as a regular grid of pixels.
Let $n \in \mathbb{N}$ and $\Om_h = \set{(i,j)}{ 1\leq i,j \leq n}$ then for a given grid spacing $h>0$ let the positions of the corresponding pixels be at
\begin{equation*}
  G_h = \Bigset{(x_i,y_j)}{x_i = h\left(i-\frac{n+1}{2}\right),  y_j = h\left(j - \frac{m+1}{2}\right), (i,j) \in \Om_h },
\end{equation*}
thus $G_h$ is a regular grid centered around the origin.
Each pixel associated with $(i,j) \in \Om_h$ is assigned a value in $\mathbb{R}$ and thus the space of real valued functions on $\Om_h$ is identified with $X = \mathbb{R}^{n\times n}$.
This finite dimensional vector space is equipped with the scalar product
\begin{equation}
  \scp[X]{u}{v} = \sum_{i,j} u_{i,j} v_{i,j}, \quad u,v \in X
\end{equation}
where the sum is understood to be over all $(i,j) \in \Om_h$.
In order to discretize the gradient $\nabla$ a finite difference scheme with spacing $h$ is used, employing forward differences.
Neumann boundary condition are prescribed at the boundary of the image, due to the forward difference scheme this is only visible at the right and bottom boundary.
This results in the operator $\nabh \colon \mathbb{R}^{n\times n} \to \mathbb{R}^{n\times n \times 2}$, $u \mapsto (\partial_x u,\partial_y u)$ with
\begin{equation*}
  \partial_x u =
  \begin{cases}
  \frac{u_{i+1,j} - u_{i,j}}{h} &\text{if } i < n \\
  0                             &\text{if } i = n
  \end{cases}
  \quad \text{and} \quad
  \partial_y u =
  \begin{cases}
  \frac{u_{i,j+1} - u_{i,j}}{h} &\text{if } j < m \\
  0                             &\text{if } j = n.
  \end{cases}
\end{equation*}
The vector space $V = \mathbb{R}^{n \times n \times 2}$ is equipped with the scalar product
\begin{equation}
\scp[V]{p}{q} = \sum_{i,j} p_{i,j,1} q_{i,j,1}+\sum_{i,j} p_{i,j,2} q_{i,j,2},
\end{equation}
which induces the norm $\norm[2]{p} = \sqrt{\scp[V]{p}{p} }$.
The adjoint of $\nabh$ is then a discretization of the negative divergence with backward differences and Dirichlet boundary conditions.
Specifically the divergence is the operator $\divh \colon V \to X$ for which $\scp[V]{\nabh u}{p} = \scp[X]{u}{-\divh p}$ it is calculated to be $\divh p = \frac{1}{h}(d_1(p) + d_2(p))$ with
\[
  d_1(p)_{i,j} =
  \begin{cases}
    p_{1,j,1}             &\text{if } i = 1     \\
    p_{i,j,1}-p_{i-1,j,1} &\text{if } 1 < i < n \\
   -p_{n-1,j,1}           &\text{if } i = n \\
  \end{cases}
  ,\quad
\]
\[
  d_2(p)_{i,j} =
  \begin{cases}
    p_{i,1,2}             &\text{if } j = 1     \\
    p_{i,j,2}-p_{i,j-1,2} &\text{if } 1 < j < m \\
   -p_{i,n-1,2}           &\text{if } j = m \\
  \end{cases}.
\]
Note that then $\norm{\nabh}^2 < \frac{8}{h^2}$ holds, as calculated in~\cite{KB}, additionally $ \nabh u = 0 $ if and only if $u$ is constant.
The one-norm on $V$ is defined as
\begin{equation}
  \norm[1]{q} = \sum_{i,j} \sqrt{q_{i,j,1}^2 + q_{i,j,2}^2}, \quad q \in V
\end{equation}
this is needed because the one-norm of the gradient will serve as a discretization of the $\TV$-semi-norm.

Next the Radon transform will be discretized.
It is simpler to start by presenting a discretization of linear backprojection -- the adjoint of the Radon transform as presented in Lemma~\ref{lem_lin_back_proj}.
Once this discretization  $\Rh^\ast$ is obtained its adjoint  $(\Rh^\ast)^\ast = \Rh$ will be a discretization of the Radon transform.
Let $\theta_1, \dots, \theta_N$ be given angles and assume there are $M$ discrete beams, parallel and evenly spaced such that the distance between to neighbouring beams is $\Delta s >0$.
The number of these discrete beams corresponds to the number of detector elements on a CT scanner and $\Delta s$ length of the detector elemens, in this discrete setting these can be thought of as bins.
Then $\Rh^\ast \colon \mathbb{R}^{N \times M} \to \mathbb{R}^{n\times n}$ refer to Figure~\ref{fig_beam_geom_num} for an illustration of the geometry used.
Assume further that $M$ is even and that the center of rotation is exactly between the middle two beams.
Let $v \in \mathbb{R}^{N \times M}$ be the data to be backprojected.
Every point $P=(i,j)$ of the image domain is assigned a position $x = h(i-\frac{n+1}{2})$ and $y = h(j - \frac{m+1}{2})$, according to the grid $G_h$.
For each angle $\theta_l$ the point $P$ is projected onto the detector line where it has the position $s = x \cos(\theta_l) + y \sin( \theta_l)$, thus the virtual beam through $P$ will be between the beams with indices $k = \lfloor \frac{s}{\Delta s } - \frac{1}{2} \rfloor + \frac{M+1}{2}$ and $k+1$.
Performing linear interpolation of the data $v$ between these two beams will give a value $\beta_l$ which is the result of the back projection from \emph{this} angle.
Then the total value of the back projection at point $P$ will be the sum over all angles $(\Rh^\ast(v))_{i,j} = \sum_{l=1}^{N} \beta_l$.
In summary $\Rh^\ast v$ is defined by
\begin{equation}
  (\Rh^\ast v)_{i,j} = \sum_{l=1}^{N} \alpha(i,j,l) v_{l,k(i,j,l)} + \big(1-\alpha(i,j,l)\big) v_{l,k(i,j,l)+1}
\end{equation}
with $k(i,j,l) = \lfloor \sigma(i,j,l) \rfloor  + \frac{M+1}{2} $ and $\alpha(i,k,l) = \sigma(i,j,l) - \lfloor \sigma(i,j,l) \rfloor$ where $\sigma(i,j,l) = \frac{1}{\Delta s} \big[ h(i- \frac{n+1}{2}) \cos(\theta_l) + h(j - \frac{m+1}{2})\sin(\theta_l) \big]- \frac{1}{2}$.
This defines a linear function $\Rh^\ast \colon \mathbb{R}^{N \times M} \to \mathbb{R}^{n\times n}$, whose adjoint $\Rh$ is a discretization of the Radon transform.
This discrete Radon trasform in turn uses the same geometry but applies the adjoint of linear interpolation which is ``splitting'' the value of $u_{i,j}$ into bins associated to the beams $k$ and $k+1$ proportional to $\alpha$ and $1- \alpha$ respectively.
Formally $\Rh$ is defined as:
\[
  (\Rh u)_{l,k} = \sum_{i,j} u_{i,j} \gamma_{i,j,k} \quad \text{with }
  \gamma_{i,j,k} =
  \begin{cases}
    \alpha(i,j,l)      &\text{if } k = k(i,j,l) \\
    (1- \alpha(i,j,l)) &\text{if } k = k(i,j,l) +1 \\
    0                  &\text{else}.
  \end{cases}
\]
Then the discrete Radon transform is an operator $\Rh \colon X \to Z$ with $Z = \mathbb{R}^{N \times M}$.

\begin{figure}[ht]
        \centering
        \includegraphics[width=\textwidth,trim=0cm 8cm 0cm 8cm, clip]{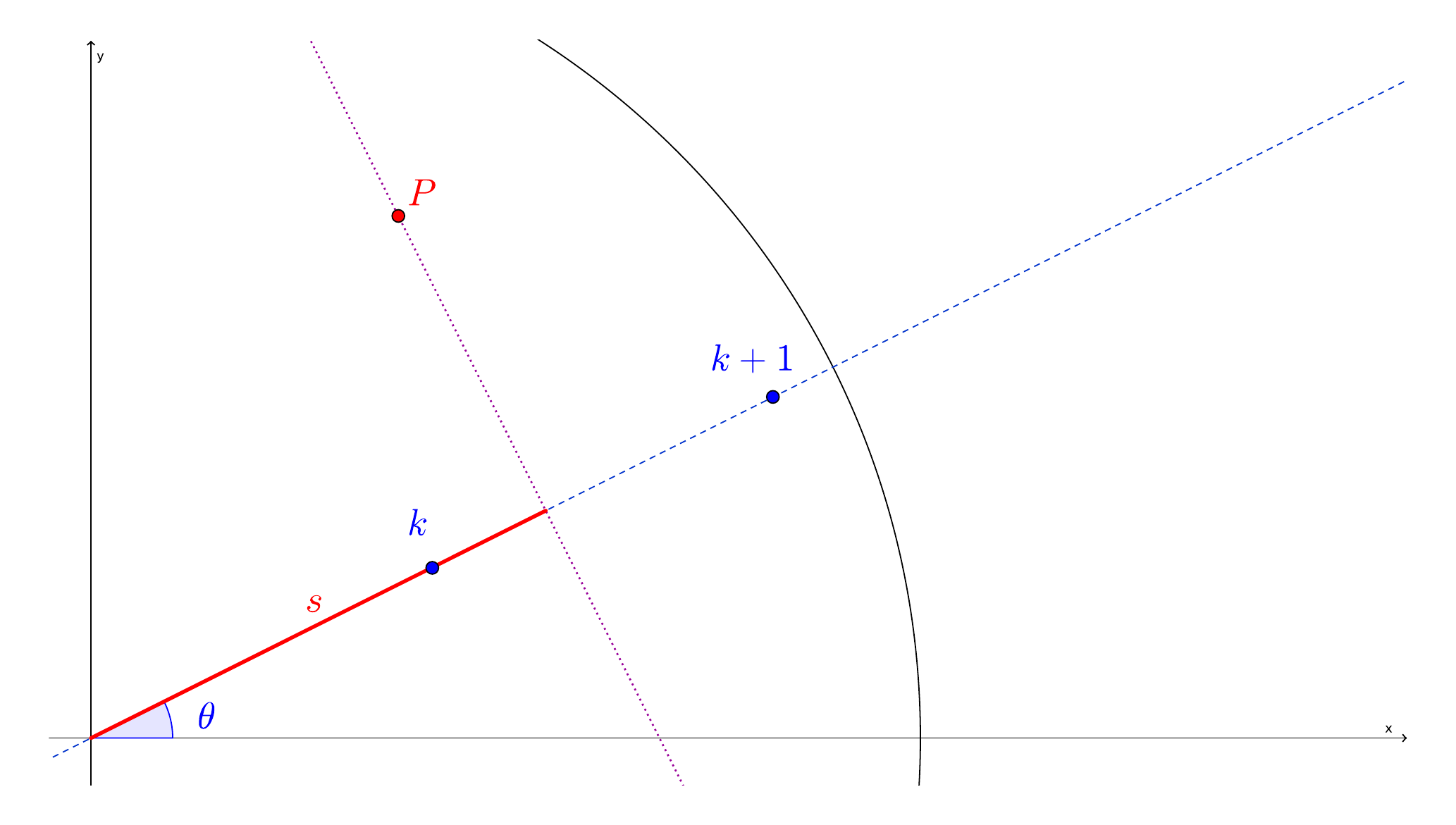}
                \caption[]{Discrete Radon transform: A point $P$ when projected onto the detector line where it has a offset $s$, the value at $P$ is split between bins $k$ and $k+1$ proportionally to the distance from the bins. For the linear back projection the value at $P$ is determined by linear interpolation between values of bins $k$ and $k+1$.}
                \label{fig_beam_geom_num}
\end{figure}
It is clear that $\Rh$ is linear and that constant functions are not mapped to zero $\Rh \mathds 1 \neq 0$.
The next lemma established a crude bound on $\norm{\Rh}$.
\begin{lemma} \label{lem_Rh_bound}
    For the discrete Radon transform with with $\Delta s < h \sqrt{2}$ it holds that $\norm{\Rh} \leq \sqrt{2 N (\sqrt{2} n + 1)}$
\end{lemma}
\begin{proof}
\begin{align}
    \norm[2]{\Rh u}^2 
    &= \sum_{l,k} (\sum_{i,j} u_{i,j} \gamma_{i,j})^2
    \leq \sum_{l,k} (\sum_{i,j} u_{i,j}^2)(\sum_{i,j} \gamma_{i,j}^2)   \\
    &\leq \norm[2]{u}^2 \sum_{l=1}^{N} \sum_{i,j} \gamma^2_{i,j,k(i,j)} + \gamma^2_{i,j,k(i,j,l)+1} 
    \leq N \norm[2]{u}^2 (2n +2)
\end{align}
where it was used that $\abs{\gamma_{i,j,k}} \leq 1$ and $\gamma_{i,j,k} \neq 0 $ only for two $\tilde k =  k(i,j,l)$ and $\tilde k +1$. 
    Also for fixed $l$ and $k_0$ the set $\set{(i,j)}{k(i,j,l)=k_0}$ contains at most $n + 1$ elements, since this is maximal number of points in the grid $G_h$ in a strip of width $\Delta s < h\sqrt{2}$.
\end{proof}
This implementation yields the same results as the MATLAB-function \emph{radon}, when using $2\times2$ oversampling that is defining the Radon transform on the finer grid $\Om_{h/2}$ with the same angles and number of beams and apply this function $\R_{h/2}$ to $Su \in \mathbb{R}^{2n \times 2m}$ given by $(Su)_{i,j} = \frac{1}{4}u( \lfloor \frac{i -1}{2} ) \rfloor +1, \lfloor \frac{j-1}{2} +1 \rfloor)$, this helps with disretization errors.
This discretized Radon transform can easily be implemented as a linear operator without explicitly calculating the corresponding matrix.
Additionally it can be parallelized to as a GPU function, one strategy for doing so which was implemented and used here can be found in Appendix~\ref{gpu_radon}, a much more elaborate implementation can be found in the Astra Toolbox~\cite{astragpu}, sadly its implementation of the back-projection is not numerically adjoint to the discrete Radon transform.

\section{The discrete problem}
To summarize the involved spaces and operators let an image $u \in X = \mathbb{R}^{n \times n}$ be mapped by $\nabh \colon X \to V $ to $V = \mathbb{R}^{n \times n \times 2}$ where the semi-norm $\norm[1]{\nabh u }$ serves as a discretization of $\TV(u)$.
The linear operator $\Rh \colon X \to Z$ with $ Z = \mathbb{R}^{N \times M}$ maps $u$ to a sinogram $\Rh u$ which is compared to given data $\unull \in Z$ by the 2-norm on $Z$.
Putting these things together gives the discretized version of the unconstrained problem:
\begin{lemma} \label{lem_ex_unconst_h}
The minimization problem
\begin{equation*}  \label{opt_unconst_h}
  \min_{ u \in X} \quad \frac{1}{2} \norm[2]{ \Rh u - \unull}^2 + \lambda \norm[1]{ \nabh u} \tag{$P_h$}.
\end{equation*}
has a solution.
\end{lemma}
\begin{proof}
The problem is of the form $\min_{u \in X} \Phi(u) + \lambda \Psi(u)$ where both $\Phi(u) = \frac{1}{2} \norm[2]{ \Rh u - \unull}^2$ and $\Psi(u) =  \norm[1]{ \nabh u}$ are continuous and convex.
As noted the operator $\Rh$ does not map constant vectors to zero.
In this discrete setting a Poincar\'{e} type inequality holds $ \norm[2]{u - \bar u} \leq C \norm[1]{\nabh u}$, for some $C>0$, where $\bar u = \frac{1}{n m} \sum_{i,j} u_{i,j}$ is the mean value of $u$.
In order to see this define the function $f(u) = \frac{\norm[2]{ u - \bar u}}{\norm[1]{\nabh u}}$, for $\norm[1]{\nabh u} \neq 0$ and show that $f$ is bounded.
Since scaling $u$ does not change the value of $f$  only consider $u \in S_1 = \set{u \in X}{\norm[2]{u} = 1}$  and also since $\nabh u = 0$ if and only if $ u = \bar u$ if and only if  $Q u = u- \bar u = 0$ where $Q$ is a projection it is sufficient to find a bound for $f$ on the subspace $W = \range(Q) \cap S_1$.
Since $W$ is compact then $\inf_{W} \norm[1]{\nabh u } = c > 0$ thus $f$ is bounded on $W$.
With this the same argument as in the continuous case in Lemma~\ref{lem_combined_term_coercive} can be made to show that the combined term in radially unbounded and thus a minimizer exists by the direct method.
Or indeed as the functional in the finite-dimensional case is continuous radially unboundedness gives the existence of a minimizer directly.
\end{proof}
Because $\norm[2]{\cdot}$ is strictly convex the minimizer is unique if $\Rh$ is injective, which in turn depends on the dimensions of the involved spaces.

Now the constrained problem is discussed.
Let 
\[
    \Dm =  \set{(l,k)}{1\leq l \leq N, 1\leq k \leq M}
\]
be the domain of the discrete sinogram and $\Dm_0 \subset \Dm$ be the indices associated with pixels in the sinogram domain that are considered to be affected by metal.
As in the continuous case a solution $u$ will be required to have projection data which in $\Dm_0$ will be greater or equal than some given $C>0$,
this is written as $\Rh u \mid_{\Dm_0} \; \geq C$ meaning 
\[
    u \in  F = \set{v \in X}{(\Rh x)_{i,j} \geq C, \text{ for all } (i,j) \in \Dm_0}.
\]
The feasible set $F$ is convex and closed, therefore $\indicator_{F}$ is convex and \lsc{}.

To compare the projection data to the given data $u_0$ it makes sense to only compare it outside of  $\Dm_0$ where an inequality constraint is already active thus the discrepancy term
\[
    \frac{1}{2} \norm[L^2(\Dm_h \setminus \Dm_0)]{\Rh u - \unull}^2 = \frac{1}{2}\sum_{(i,j) \in \Dm_h \setminus \Dm_0} ( (\Rh u)_{i,j} - \unull_{i,j} ) ^2 
\]
is used which is a convex and continuous function on $Z$.
Then $\phi(u) =  \frac{1}{2} \norm[2]{ \Rh u - \unull}^2 +  \indicator_{ \{\Rh u\mid_{\Dm_0} \geq C\}}(u)$ is a continuous and convex function and the same arguments as in Lemma~\ref{lem_ex_unconst_h} show that there is a solution to the constrained problem
\begin{equation} \label{opt_constr_h}
  \min_{ u \in X} \quad \frac{1}{2} \norm[L^2(\Dm \setminus \Dm_0)]{ \Rh u - \unull}^2 + \lambda \norm[1]{ \nabh u} + \indicator_{ \{\Rh u\mid_{\Dm_0} \geq C\}}(u). \tag{$C_h$}
\end{equation}

\section{Numerical optimization}
For the application of numerical algorithms to the problems~\eqref{opt_unconst_h} and~\eqref{opt_constr_h} only the constrained problem is considered as the unconstrained problem as a special case with $\Dm_0 = \emptyset$.
The optimization problem \eqref{opt_constr_h} is of the form
\footnote{Note different notations: In~\cite{KB} what is $G^\ast$ here is called $G$ and vice versa, whereas in~\cite{Pock_dual} $F$ and $G$ are interchanged. This thesis sticks with the notation of~\cite{Precond}.}
\begin{equation} \label{opt_abst}
\min_{u\in X} F(u) + G(Ku),
\end{equation}
with $F \colon X \to \extR$ defined by $F=0$ which is clearly  convex, lower semi-continuous and proper and $K \colon X \to Y$ linear and continuous, where $Y = V \times Z$ and $G \colon Y \to \extR$ defined as
\begin{equation}
    K =
    \begin{bmatrix}
    \Rh \\ \nabh
    \end{bmatrix},
    \qquad
    G(x,y) = \frac{1}{2}\norm[L^2(\Dm \setminus \Dm_0)]{x - \unull}^2
    +  \indicator_{ \{x\mid_{\Dm_0} \geq C \}}(x) +
    \lambda \norm[1]{y}.
\end{equation}


The problem satisfies the sufficient conditions for the Fenchel-Rockafellar duality~\cite{KB}.
The dual problem reads as
\begin{equation}
\max_{w \in Y^\ast} -F^\ast(-K^\ast w) - G^\ast(w).
\end{equation}
Solving the primal and dual problem simultaneously can be interpreted as finding the saddle-point of the function $L \colon \dom(F) \times \dom(G^\ast) \to \mathbb{R}$ defined by
\begin{equation} \label{L_saddle}
L(u,w) = \scp{w}{Ku} + F(u) - G^\ast(w).
\end{equation}
A numerical algorithm for solving this problem is given by the Chambolle-Pock algorithm \cite{Pock_dual}
\begin{equation} \label{alg_abst_CP}
    \begin{cases}
    w^{k+1} = (\id + \tau   \partial G^\ast)^{-1}(w^{k} + \tau   K \bar{u}^k )   \\
    u^{k+1} = (\id + \sigma \partial F)     ^{-1}(u^{k} - \sigma K^\ast w^{k+1}) \\
    \bar{u}^{n+1} = 2u^{k+1} - u^k
    \end{cases}
\end{equation}
which converges to a saddle point of \eqref{L_saddle} for any parameter choice $\sigma, \tau >0$ satisfying  $\sigma \tau \norm{K}^2 < 1$.

Thus, it is still needed to calculate the resolvents, for $F = 0$ the resolvent clearly is $ (\id + \sigma \partial F)^{-1} = \id$.
Now the resolvent of $G^\ast$ calculated, first it is shown that taking the dual and the resolvent can be done componentwise and pointwise.

\begin{lemma} \label{lem_loc}
  Let $F \colon \mathbb{R}^r \to \mathbb{R}$ be given by $F(v) = \sum_{i= 1}^{r} f_i(v_i)$, with $f_i$ proper, convex and \lsc{}  then the dual of $F$ is given by
\begin{equation}
F^\ast(w) = \sum_{i=1}^r f_i^\ast(w_i).
\end{equation}
Additionally for $\sigma >0$
\begin{equation}
(\id + \sigma \partial F^\ast)^{-1}(u) = \ndel{(\id + \sigma \partial f_i^\ast)^{-1}(u_i)}_{i=1}^r.
\end{equation}
\end{lemma}
\begin{proof}
The dual of $F$ is given by
\begin{equation*}
F^\ast(w) =  \sup_{v \in \mathbb{R}^r} \sum_{i=1}^r v_i w_i - \sum_{i=1}^r  f_i(v_i)
= \sum_{i=1}^r \sup_{v_i \in \mathbb{R}} v_i w_i - f_i(v_i) = \sum_{i=1}^r f_i^\ast(w_i).
\end{equation*}
thus it holds that
\begin{align*}
(\id + \sigma \partial F^\ast)^{-1}(u) = v   \Leftrightarrow
u \in \argmin_{x \in \mathbb{R}^r} \frac{1}{2} \norm{x-v}^2 + \sigma F^\ast(x) \\ \Leftrightarrow
\forall i  \, u_i \in \argmin_{x \in \mathbb{R}} \frac{1}{2} (x-v_i)^2 + \sigma f_i^\ast(x)  \Leftrightarrow
\forall i  \, u_i \in (\id + \sigma \partial  f_i^\ast)(v_i).
\end{align*}
\end{proof}

Because of the first part of  Lemma~\ref{lem_loc}, that the dual $G^\ast$ of $G$ can be calculated componentwise, \ie $G^\ast (x,y) = G^\ast_1(x) + G^\ast_2(y)$, because of the second part of Lemma \ref{lem_loc} the resolvent of $G^\ast$ can be evaluated componentwise, \ie
\begin{equation}
(\id + \tau \partial G^\ast)^{-1}(\bar v, \bar w) =
    \begin{pmatrix}
    (\id + \tau \partial G^\ast_1)^{-1}(\bar v) \\
    (\id + \tau \partial G^\ast_2)^{-1}(\bar w)
    \end{pmatrix}.
\end{equation}

The dual of $G_2(y) = \lambda \norm[1]{y}$ is $G^\ast_2(\eta) = \indicator_{B_\lambda^\infty(0)}(\eta)$, where ${B_\lambda^\infty(0)}$  denotes the norm ball of radius $\lambda$ around the origin in the maximum norm, consequently the resolvent of $G_2$ is the projection onto ${B_\lambda^\infty(0)}$:
\begin{equation}
\ndel{(\id + \tau \partial G^\ast_2)^{-1}(\bar w)}_{i,j}
 = \ndel{ \mathcal{P}_{B_\lambda^\infty(0)}(\bar w) }_{i,j}
 = \frac{\bar w_{i,j}}{ \max \{ 1, \abs{ \bar w_{i,j} } / \lambda \} }.
\end{equation}

Now the dual of $G_1$ and its resolvent are calculated. It makes sense to only try to fit the data where it is considered to be valid thus the discrepancy term $\norm[2]{\cdot - \unull} = \norm[L^2(\Dm \setminus \Dm_0)]{ \cdot - \unull} $ is used.
    Also the inequality constraints can be written pointwise $\indicator_{x\mid_{\Dm_0 \geq C }} =  \sum_{(i,j) \in \Dm_0} \indicator_{ \{x \geq C\} }(x_{i,j})$.
Thus, the fidelity term is
\begin{equation*}
  G_1(x) = \frac{1}{2}\sum_{(i,j)\in \Dm \setminus \Dm_0} (x_{i,j}-\unull_{i,j})^2
+ \sum_{(i,j) \in \Dm_0} \indicator_{ \{x \geq C\} }(x_{i,j}).
\end{equation*}
With the pointwise functions
\begin{equation*}
g_{1,1}(z) =  \frac{1}{2}(z-\unull)^2, \quad g_{1,2}(z) = \indicator_{ \{z \geq C\} }(z)
\end{equation*}
the fidelity term will have the dual
\begin{equation}
G^\ast_1(\xi) = \sum_{(i,j) \in \Dm \setminus \Dm_0 } g_{1,1}^\ast (\xi_{i,j}) +  \sum_{(i,j) \in  \Dm_0 } g_{1,2}^\ast (\xi_{i,j}).
\end{equation}

\begin{lemma} \label{lem_f_ast1}
The function $g_{1,1} \colon \mathbb{R} \to \mathbb{R} $, $ z \mapsto \frac{1}{2}(z-\unull)^2$ has the dual
\begin{equation}
g_{1,1}^\ast (\xi) = \frac{1}{2}\xi^2 + \xi \unull
\end{equation}
and the function $g_{1,2} \colon \mathbb{R} \to \mathbb{R}_{\infty} $, $ z \mapsto \indicator_{ \{x \geq C\} }(z)$ has the dual
\begin{equation}
 g_{1,2}^\ast (\xi) =
 \begin{cases}
    \infty, &\text{if } \xi >    0\\
    \xi C,  &\text{if } \xi \leq 0.
 \end{cases}
\end{equation}
\end{lemma}
\begin{proof}
For $ g_{1,1}$ the dual is given by
\begin{equation*}
g_{1,1}^\ast(\xi)
= \sup_{z \in \mathbb{R}} \scp{\xi}{z} - g_{1,1}(z)
= \sup_{z \in \mathbb{R}} \scp{\xi}{z} - \frac{1}{2}(z-\unull)^2.
\end{equation*}
Let $h(z) = \scp{\xi}{z} - \frac{1}{2}(z-\unull)^2$ for any $\xi$ the function $h$ is differentiable with $h'(z) = \xi - (z-\unull)$ and $h''(z) = -1$, thus $h(z)$ will reach its maximum at $ z = \xi + \unull$.
Plugging this back in gives
$g_{1,1}^\ast(\xi) = \xi ( \xi + \unull) - \frac{1}{2} \xi^2 = \frac{1}{2} \xi^2 + \xi \unull$ as claimed.
For $g_{1,2}^\ast$ the dual is
\begin{equation*}
g_{1,2}^\ast(\xi)
= \sup_{z \in \mathbb{R}} \scp{\xi}{z} - g_{1,2}(z)
= \sup_{z \in \mathbb{R}} \scp{\xi}{z} - \indicator_{ \{x \geq C\} }(z).
\end{equation*}
For $\xi > 0 $  increasing $z$ beyond $C$ the indicator function vanishes and $\lim_{z \to \infty} \scp{\xi}{z} = \infty$,  thus for $\xi>0$ it holds that $g_{1,2}^\ast(\xi) = \infty$.
For $\xi \leq 0 $ the term $\scp{\xi}{z}$ is maximized by choosing $z$ as small as as possible while keeping $g_{1,2}(z)$ finite this is the case at $z = C $, thus the supremum is $g_{1,2}^\ast(\xi) = \xi C$.
\end{proof}

\begin{lemma} \label{lem_f_ast_resolv1}
For the functions $g^\ast_{1,1}$ and $g^\ast_{1,2}$ from Lemma \ref{lem_f_ast1} the resolvents are
\begin{align}
(\id + \tau \partial g_{1,1}^\ast)^{-1}v &=  \frac{v - \tau \unull}{1+\tau} \\
(\id + \tau \partial g_{1,2}^\ast)^{-1}v &= \min \{  v - \tau C, 0 \}.
\end{align}
\end{lemma}
\begin{proof}
As $g^\ast_{1,1}$ is differentiable it holds that
\begin{equation}
    \partial g^\ast_{1,1}(\xi) = \sett{(g^\ast_{1,1})'(\xi)} = \sett{\xi + \unull}
\end{equation}
and thus
\begin{equation}
(\id + \tau \partial g_{1,1}^\ast)^{-1}v = \xi
\Leftrightarrow v \in \xi + \tau \partial g^\ast_{1,1}(\xi)
= \xi + \tau ( \xi + \unull)
\Leftrightarrow \xi = \frac{v - \tau \unull }{1+ \tau}.
\end{equation}
For $ g^\ast_{1,2}$ the resolvent is
\begin{equation}
(\id + \tau \partial g_{1,2}^\ast)^{-1}v = \xi
\Leftrightarrow v \in \xi + \tau \partial g^\ast_{1,2}(\xi)
=
\begin{cases}
    \xi + \tau C,    &\text{if } \xi < 0 \\
    [\tau C,\infty), &\text{if } \xi = 0 \\
    \emptyset,       &\text{if } \xi > 0.
\end{cases}
\end{equation}
Thus, if $v \geq \tau C $, the first case cannot be true so $\xi = 0$ must hold,
whereas if $ v < \tau C $ the first case must apply and so $ \xi = v - \tau C$.
Terefore, $ \xi = \min \{ v - \tau C, 0 \}$.
\end{proof}
Putting the last lemmas together gives the resolvent of $G_1^\ast$
\begin{equation}
(\id + \tau \partial G^\ast_1)^{-1}(\bar v_{i,j}) =
    \begin{cases}
    \frac{\bar v_{i,j}-\tau (\unull)_{i,j}}{1+\tau},        &\text{if } (i,j) \in \Dm    \setminus \Dm_0   \\
\min \{ \bar v_{i,j} - \tau C, 0 \},                 &\text{if } (i,j) \in  \Dm_0 .
    \end{cases}.
\end{equation}

\section{Chambolle-Pock algorithm}
Now everything can be plugged together into the abstract Chambolle-Pock iteration
\begin{equation} \label{alg_abst_CP2}
    \begin{cases}
    y^{k+1} = (\id + \tau   \partial G^\ast)^{-1}(y^{k} + \tau   K \bar{u}^k )   \\
    u^{k+1} = (\id + \sigma \partial F)     ^{-1}(u^{k} - \sigma K^\ast y^{k+1}) \\
    \bar{u}^{n+1} = 2u^{k+1} - u^k.
    \end{cases}
\end{equation}
Using the operator
$K = \begin{bmatrix} \Rh, \nabh \end{bmatrix}^T $ calling the two components $y = (v,w) $ and as a shorthand use $\bar y = (\bar v, \bar w) = y + \tau K \bar u$.
Then the just calculated resolvents will directly apply to $\bar v$ and $\bar w$.
Also then $K^\ast y =  \Rh^\ast v - \divh w$.
Since $\norm{K}^2 = \norm{\nabh }^2 + \norm{\Rh}^2 $ and $\norm{\nabh}^2 < \frac{8}{h^2}$  the condition $\sigma \tau \norm{K}^2 <1$ will be satisfied if $\sigma \tau < ( \norm{\Rh}^2 + \frac{8}{h^2})^{-1}$.
Then the abstract Chambolle-Pock iteration~\eqref{alg_abst_CP2} yields Algorithm~\ref{alg_CP}.
\begin{algorithm}
\caption{Chambolle-Pock algorithm for minimizing \eqref{opt_constr_h}}
\begin{algorithmic} \label{alg_CP}
\REQUIRE  \textbf{Input}:  $\unull \in \mathbb{R}^{N \times M},\, \Dm_0,\, C >0,\, h>0 $
\REQUIRE \textbf{Parameters}: $\sigma >0, \, \tau >0, \text{with }\sigma \tau < (\norm{\Rh}^2 + \frac{8}{h^2})^{-1}$
\REQUIRE \textbf{Initial values}:$\bar u$
\FOR{$k < k_{max}$}
\STATE $\bar w \leftarrow w + \tau\nabh \bar u$
\STATE $\bar v \leftarrow v + \tau \Rh \bar u$
\STATE $w_{i,j} \leftarrow  \frac{\bar w_{i,j}}{ \max \{ 1, \abs{ \bar w_{i,j} } / \lambda \} }$
\STATE $v_{i,j} \leftarrow
\begin{cases}
\min \{ \bar v_{i,j} - \tau C, 0 \}, &\text{if } (i,j) \in \Dm_0  \\
\ndel{\bar v_{i,j} - \tau (\unull)_{i,j}}/(1+\tau), &\text{else }
\end{cases}
$
\STATE $ u_{new} \leftarrow u + \sigma \divh(w) -\sigma \Rh^\ast(v) $
\STATE $ \bar u \leftarrow 2 u_{new} - u$
\STATE $u \leftarrow u_{new}$
\ENDFOR
\end{algorithmic}
\end{algorithm}

\section{Variants and Extensions}
\addtocontents{toc}{\protect\setcounter{tocdepth}{1}}
In this section some variants of the problems are discussed.
One simple variant is to let the threshold $C$ be dependent of the on the location in the sinogram, since the resolvents are calculated pointwise all above resolvents can be modified with $C = C_{i,j}$.
\subsection{Weighted fitting}
It makes sense to introduce weights to the data fidelity term $\norm{ \Rh u - \unull}^2$ indeed~\cite{ramani2012splitting} underlines its importance.
Assigning a positive weight $\omega_{i,j}$ to each pixel
By slightly modifying the previous lemma the new resolvent is calculated.
\begin{lemma} \label{lem_f_ast1_w}
The function $g_{1,1} \colon \mathbb{R} \to \mathbb{R} $, $ z \mapsto \frac{1}{2} \omega (z-\unull)^2$ has the dual
\begin{equation}
    g_{1,1}^\ast (\xi) = \frac{1}{2} \frac{1}{\omega}\xi^2 + \xi \unull
\end{equation}
and thus for $\tau >0 $ the corresponding resolvent is:
\begin{equation}
     (\id + \tau \partial g_{1,1}^\ast)^{-1}v =  \frac{v - \tau \unull }{1+  \frac{1}{\omega}\tau}.
\end{equation}
\end{lemma}
\begin{proof}
For $ g_{1,1}$ the dual is given by
\begin{equation*}
g_{1,1}^\ast(\xi)
= \sup_{z \in \mathbb{R}} \scp{\xi}{z} - \omega \frac{1}{2}(z-\unull)^2.
\end{equation*}
    Let $h(z) = \scp{\xi}{z} - \frac{1}{2}(z-\unull)^2$ for any $\xi$ the function $h$ is differentiable with $h'(z) = \xi - \omega (z-\unull)$ and $h''(z) = -\omega <0$, thus $h(z)$ will reach its maximum at $ z = \frac{1}{\omega}\xi + \unull$.
Plugging this back in gives
$g_{1,1}^\ast(\xi) = \frac{1}{2} \frac{1}{\omega} \xi^2 + \xi \unull$ as claimed.
As $g^\ast_{1,1}$ is differentiable it holds that
\begin{equation}
    \partial g^\ast_{1,1}(\xi) = \sett{(g^\ast_{1,1})'(\xi)} = \sett{\frac{1}{\omega}\xi + \unull}
\end{equation}
and thus
\begin{equation}
(\id + \tau \partial g_{1,1}^\ast)^{-1}v = \xi
    = \xi + \tau ( \frac{1}{\omega}\xi + \unull)
    \Leftrightarrow \xi = \frac{v - \tau \unull }{1+  \frac{1}{\omega}\tau}.
\end{equation}
\end{proof}
\subsection{Ignoring data}
Since neither the region $\Dm_0$ nor the threshold $C$ is known in an application setting a heuristic strategy for finding $\Dm_0$ is presented.
The basic idea is to increase this region while ignoring its contents until the metal artifacts are removed. This is accomplished by solving the problem:
\begin{equation} \label{opt_main_ignore}
  \min_{u\in X} \quad \frac{1}{2} \norm[L^2(\Dm \setminus \Dm_0)]{ \R_h u - \unull}^2 + \lambda \norm[1]{ \nabla_h u}.
\end{equation}
Then the function $g_{1,1}$ will remain unchanged, but now $g_{1,2} = 0$, the corresponding resolvent is the $0$ as is calculated in the following lemma
\begin{lemma} \label{lem_f_ast_resolv_ignore}
The function $g_{1,2} = 0$ has the resolvent
\begin{equation*}
(\id + \tau \partial g_{1,2}^\ast)^{-1}(v) = 0
\end{equation*}
\end{lemma}
\begin{proof}
The dual of $g_{1,2}$ is
\begin{equation*}
g_{1,2}^\ast( \xi ) = \sup_{x} \scp{\xi}{x} - 0 =
\begin{cases}
0,&\text{if } \xi = 0 \\
\infty, &\text{if } \xi \neq 0 \\
\end{cases}.
\end{equation*}
Thus, the resolvent of $ g_{1,2}^\ast$(is
\begin{equation*}
(\id + \tau \partial g_{1,2}^\ast)^{-1}(v) = \xi
\Leftrightarrow v \in \xi + \tau \partial g_{1,2}^\ast(\xi) =
\begin{cases}
\mathbb{R}, &\text{if } \xi =    0 \\
\emptyset,  &\text{if } \xi \neq 0
\end{cases},
\end{equation*}
so the solution is $\xi = 0$ for any $v$.
Or indeed one could argue that for any $v$ only $\xi = 0$ solves $\sup_{\xi} \frac{1}{2}\norm[2]{\xi - v}^2 + \tau g_{1,2}^\ast$.
\end{proof}
Therefore in this case the resolvent of $G^\ast_1$ will be
\begin{equation} \label{eq_resolv_ignore}
(\id + \tau \partial G^\ast_1)^{-1}( v_{i,j}) =
\begin{cases}
\frac{ v_{i,j}-\tau (\unull)_{i,j}}{1+\tau},    &\text{if } (i,j) \in \Dm    \setminus \Dm_0   \\
0,                 &\text{if } (i,j) \in  \Dm_0 .
\end{cases}.
\end{equation}
This can be plugged into Algorithm~\ref{alg_CP} to obtain a algorithm converging to a solution of~\eqref{opt_main_ignore}.

\subsection{Data fit everywhere} \label{sssec_Ind_p_quad}
In  trying to fit the data in all of $\Dm$, \ie $\norm[2]{\cdot - \unull} = \norm[L^2(\Dm)]{ \cdot - \unull} $ thus trying to fit the data also in $\Dm_0$ where it is also demanded that the inequality constraint $\Rh u \geq C$ holds.
Then the function $g_{1,2}$ will remain unchanged, but $g_{1,1}(x) = \frac{1}{2}(x-\unull)^2 + \indicator_{[C,\infty)}(z)$.
With this the following resolvents are obtained.
\begin{lemma} \label{lem_f_ast_other}
The function $f \colon \mathbb{R} \to \mathbb{R} $, $ z \mapsto \frac{1}{2}(z-a)^2 +
\indicator_{[C,\infty)}(z)$ has the dual
\begin{equation*}
f^\ast(y) =
\begin{cases}
\frac{1}{2} y^2 + y a,     & \text{if } y  \geq   C - a \\
yC - \frac{1}{2}(C-\unull)^2,   & \text{if } y  <      C - a,
\end{cases}
\end{equation*}
which is continuous as $f^\ast(C-a) = \frac{1}{2}( C^2 - a^2) = \lim\limits_{y \uparrow C-a} f^\ast(y).$
\end{lemma}
\begin{proof}
Set  $g(z) = yz - \frac{1}{2}(z-a)^2$, then
\begin{equation*}
f^\ast(y) = \sup_{z \in \mathbb{R}} \, yz - f(z)
= \sup_{z \geq C}         g(z).
\end{equation*}
As the function $g$ is differentiable and concave its maximum in $[C,\infty)$ is $\max\{g(z^\ast),g(C)\}$ with
\begin{equation*}
0 = g'(z^\ast) = y - (z^\ast - a)  \Leftrightarrow z^\ast = y + a,
\end{equation*}
if $z^\ast \geq C$ i.\,e. $y \geq C - a$. The values are
\begin{align*}
g(C)      &= yC - \frac{1}{2}(C-a)^2. \\
g(z^\ast) &= y(y+a) - \frac{1}{2}y^2 = \frac{1}{2} + ya.
\end{align*}
Because of
\begin{equation*}
\ndel{ y- (C-a)}^2 \geq 0 \\
y^2 - 2y(C-a) + (C-a)^2 \geq 0 \\
\frac{1}{2}y^2 + ya \geq yC - \frac{1}{2}(C-a)^2,
\end{equation*}
for $ y \geq C- a  $ this is
\begin{equation*}
f^\ast(y) = \max\{g(z^\ast),g(C)\} = g(z^\ast) = \frac{1}{2}y^2 + ya.
\end{equation*}
In the case $  y < C - a $, $z^\ast $ is not in $[C,\infty)$, thus the maximum is reached at $C$, so $f^\ast(y) = g(C)$.
\end{proof}

\begin{lemma} \label{lem_f_ast_resolv_other}
For the function $f^\ast$ from Lemma~\ref{lem_f_ast_other}, the resolvent is
\begin{equation*}
(\id + \tau \partial f^\ast)^{-1}(v) =
\begin{cases}
\frac{v-\tau a}{1+\tau}, &\text{if } v \geq C(1+\tau) - a  \\
v- \tau C,                 &\text{if } v <    C(1+\tau) - a. \\
\end{cases}
\end{equation*}
\end{lemma}
\begin{proof}
Calculating the subdifferential of $f^\ast$ yields
\begin{equation*}
\partial f^\ast(y) =
\begin{cases}
\{ y + a \}, &\text{if }   y > C - a \\
\{ C       \}, &\text{if } y < C - a \\
[C,C],         &\text{if } y = C - a,
\end{cases}
\end{equation*}
which is in particular singleton. Thus, for the resolvent
\begin{equation*}
    (\id + \tau \partial f^\ast)^{-1}(v) = u \Leftrightarrow
    v = u + \tau \partial f^\ast(u) =
    \begin{cases}
        u + \tau (u + a),  &\text{if } u \geq C - a \\
        u + \tau C,        &\text{if } u < C - a.
    \end{cases}
\end{equation*}
Let $ v\geq C(1+\tau) - a$, then the second case $ u = v - \tau C \geq C - a$ is not possible,
but $u = \frac{v - \tau a}{1+\tau} \geq \frac{C(1+\tau) - (1+\tau) a}{1+\tau} \geq C - a $ fits the first case.
Let $ v < C(1+\tau) - a$, then the first case $ u = \frac{v - \tau a}{1+\tau} < \frac{C(1+\tau) - (1+\tau) a}{1+\tau} = C - a$ is not possible,
whereas $ u = v - \tau C < C - a$ fits the second case.
\end{proof}

\subsection{Hard constraints}
Rather then try to approximate the measurements $\unull$ one can try to force the solution to be exactly $\unull$ using equality constraints.
This lead to the following problem.
\begin{equation} \label{opt_hc_h}
\begin{cases}
  \min_{ u \in X  } \quad  \norm[1]{\nabh u} \\
\text{s.t. }  (\Rh u) \mid_{\Dm_0} \geq C \quad \text{and } \Rh u \mid_{\Dm \setminus \Dm_0 } = \unull \mid_{\Dm \setminus \Dm_0 }
\end{cases}
\end{equation}
In the case of hard constraints the discrepancy term is
\[
    G_1 =  \indicator_{\{x\mid_{\Dm_0} \geq C\}} + \indicator_{\{x\mid_{\Dm \setminus \Dm_0} = \unull\}} 
\]
and the resolvent of $G^\ast_1$ is again
$(\id + \sigma \partial G_1^\ast)^{-1}(\bar v_{i,j}) = \min \{ \bar v_{i,j} - \sigma C, 0 \}$ in $\Dm_0$ and
in $\Dm \setminus \Dm_0$ it is  $(\id + \sigma \partial G_1^\ast)^{-1}(\bar v_{i,j}) = \bar v_{i,j} - \sigma (\unull)_{i,j}$.
Finally choose $\lambda>0$ arbitrarily  \eg $\lambda = 1$.
\addtocontents{toc}{\protect\setcounter{tocdepth}{3}}

\chapter{Preconditioning} \label{sec_Precond}
In this section preconditioning methods are applied to the problem, this is motivated in particular by the slow convergence of the Chambolle-Pock algorithm for the constrained problem \eqref{opt_constr_h}.
The main idea is to approximately invert $\R^\ast R$ in each step of the iteration.
A similar approach can be found in~\cite{ramani2012splitting} were the Alternating Direction Method of Multipliers(\myacro{ADMM}) is used with a preconditioned conjugate gradient(\myacro{PCG}) step in each iteration of \myacro{ADMM}.
This method has been implemented and serves as a reference it is discussed in Section~\ref{sec_ADMM}.
Recently developed methods~\cite{Precond} for preconditioning the Douglas-Rachford splitting method which are well suited for image processing are applied to the problems.
Additionally the fact that the application of $\R^\ast R$ is essentially the same as the convolution with $\frac{1}{\abs{\cdot}}$, will be used to construct a deconvolution type preconditioner.
\section{A preconditioned Douglas Rachford splitting method}
An overview of the Douglas-Rachford splitting method and its preconditioned version is given, following the presentation in~\cite{Precond}.
Let $X$ and $Y$ be Hilbert spaces, $F \colon X \to \extR$ and $G \colon Y \to \extR$ be proper, convex and \lsc, and let $K \colon X \to Y$ be linear and continuous then the problem
\begin{equation}  \label{opt_main_pdr}
  \min_{u \in X} F(u) + G(Ku),
\end{equation}
has the  optimality condition  $0 \in \partial F(u) + K^\ast \partial G (Ku)$.
Let $ v \in \partial G (Ku)$ which is equivalent to $Ku\in \partial G^\ast (v)$ by Lemma~\ref{lem_Fenchel_eq} then the optimality condition is equivalent to the primal-dual optimality system
\begin{align} \label{opt_primal_dual}
\begin{cases}
&0 \in - Ku     + \partial G^\ast(v) \\
&0 \in K^\ast v + \partial F(u).
\end{cases}
\end{align}
Multiplying by $\sigma >0$ and using axillary variables $\tilde u \in \sigma \partial F(u)$ and $\tilde v \in \sigma \partial G^\ast(v)$ gives $0 \in -\sigma Ku +\tilde v$ and $0 \in K^\ast v + \tilde u$.
Thus the system \eqref{opt_primal_dual} can be written as $ 0 \in \mathcal A w$ with $w=(u,v,\tilde u, \tilde v)$ and
\begin{equation} \label{eq_mathcalA}
\mathcal{A} =
\begin{pmatrix}
&0         &\sigma K^\ast &\id  &0 \\
&-\sigma K &0             &0    &\id \\
&-\id      &0             &(\sigma \partial F)^{-1} &0\\
&0         &-\id          &0 &(\sigma \partial G)^{-1}
\end{pmatrix}.
\end{equation}
The problem $ 0 \in \mathcal A w$ can be solved by applying the proximal point method~\cite{rockafellar1976monotone}
 for $ 0 \in \mathcal M (w^{k+1} - w^{k}) + \mathcal A w^{k+1}$ which results in the iteration $w^{k+1} = (\mathcal{M} + \mathcal{A})^{-1} \mathcal{M} w^k$.
Using the operator
\begin{equation} \label{eq_mathcalM}
\mathcal{M} =
\begin{pmatrix}
&N_1    &0      &-\id   &0      \\
&0      &N_2    &0      &-\id   \\
&-\id   &0      &\id    &0      \\
&0      &-\id   &0      &\id
\end{pmatrix},
\end{equation}
for now assume $N_2 = \id$, then $\mathcal M$ is linear, continuous and self-adjoint positive semi-definite under the assumption that
\begin{equation} \label{asmp_N1N2}
N_1 \colon X \to X \text{ linear, continuous, self-adjoint and } N_1 - \id \geq 0.
\end{equation}
With $\bar u = u - \tilde u$ and $ \bar v = v - \tilde v$ the proximal point method  $w^{k+1} = (\mathcal{M} + \mathcal{A})^{-1} \mathcal{M} w^k$
becomes
\begin{align*}
u^{k+1} &= N_1^{-1}( (N_1 - \id) x^k + \bar{u}^k - \sigma K^\ast v^{k+1} \\
v^{k+1} &= \bar{v}^k + \sigma K u^{k+1} \\
\bar{u}^{k+1} &= \bar{u}^k + ( \id + \sigma \partial F)^{-1}(2 u^{k+1} - \bar{u}^k) - u^{k+1} \\
\bar{v}^{k+1} &= \bar{v}^k + ( \id + \sigma \partial G^\ast)^{-1}(2 v^{k+1} - \bar{v}^k) - v^{k+1},
\end{align*}
using the Moreau identity from Lemma~\ref{lem_moreau_id}.
In~\cite{Precond} it is shown that this iteration converges weakly to a fixed point $w^\ast = (u^\ast, v^\ast, \bar{u}^\ast, \bar{v}^\ast)$ of which $ (u^\ast, v^\ast)$ is a solution of the primal dual optimality system \eqref{opt_primal_dual} if such a solution exits.
Because the iteration for $u$ and $v$ is implicit plug the second equation into the first to obtain
\[
(N_1 + \sigma^2 K^\ast K) u^{k+1} = (N_1 - \id) u^k + \bar{u}^k - \sigma K^\ast \bar{v}^k.
\]
Let $ M = N_1 + \sigma^2 K^\ast K$ and $T = \id + \sigma^2 K^\ast K$ then the assumption on $N_1$ in~\eqref{asmp_N1N2} is true if and only if
\begin{equation*} \label{asmp_M}
M \colon X \to X \text{ linear, continuous, self-adjoint and } M - T \geq 0.
\end{equation*}
Therefore $M$ can be interpreted as a preconditioner for $T$.
Putting these things together yields the following explicit algorithm.
\begin{algorithm}
\caption{\textbf{ PDR } Algorithm for minimizing \eqref{opt_main_pdr}}
\begin{algorithmic} \label{alg_PDR}
\REQUIRE \textbf{Initial values}: $(\unull, \bar{u}^0, v^0, \bar{v}^0) \in X^2 \times Y^2, \sigma >0$ \\
\STATE $ b^k  = \bar{u}^k - \sigma K^\ast \bar{v}^k$  \\
\STATE $ u^{k+1} = u^k + M^{-1}(b^k - Tu^k)$  \\
\STATE $ v^{k+1} = \bar{v}^k + \sigma K u^{k+1} $\\
\STATE $\bar{u}^{k+1} = \bar{u}^k + ( \id + \sigma \partial F)     ^{-1}(2 u^{k+1} - \bar{u}^k) - u^{k+1}$ \\
\STATE $\bar{v}^{k+1} = \bar{v}^k + ( \id + \sigma \partial G^\ast)^{-1}(2 v^{k+1} - \bar{v}^k) - v^{k+1}$ \\
\end{algorithmic}
\end{algorithm}

\subsection{Linear-quadratic functionals}
Assume a linear-quadratic term can be split off from $F$ and $G^\ast$ off while still retaining convexity and lower semi-continuity, that is
\begin{equation} \label{eq_abs_LQ}
  \begin{aligned}
    F(u)      &= \frac{1}{2} \scp{Qu}{u} + \scp{f}{u} + \tilde{F}(u), \\
    G^\ast(v) &= \frac{1}{2} \scp{Sv}{v} + \scp{g}{v} + \tilde{G^\ast}(v)
  \end{aligned}
\end{equation}
with $Q \colon X \to X$, $S \colon Y \to Y$ linear, continuous, self-adjoint, positive semi-definite and $f \in X$, $g \in Y$ and still $\tilde{F} \colon X \to \extR $, $\tilde{G^\ast} \colon Y \to \extR$ proper, convex and \lsc{}.
Then the primal-dual optimality condition becomes
\begin{align} \label{opt_primal_dual_LQ}
\begin{cases}
&0 \in Sv + g - Ku       + \partial \tilde{G^\ast}(v) \\
&0 \in Qu + f + K^\ast v + \partial \tilde{F}(u).
\end{cases}
\end{align}
This is again equivalent to $0 \in \mathcal{A}w$ with
\begin{equation}
\mathcal{A} =
\begin{pmatrix}
&\sigma Q + \sigma \mathds{1} &\sigma K^\ast    &\id &0 \\
&-\sigma K      &\sigma S +\sigma \mathds{1}    &0 &\id \\
&-\id           &0  &(\sigma \partial F)^{-1}   &0      \\
&0  &-\id       &0         &(\sigma \partial G^\ast)^{-1}
\end{pmatrix}.
\end{equation}
Assume that $S = \nu \id$ and choose $\mathcal{M}$ as in \eqref{eq_mathcalM} but letting $N_2 = \mu \id$ for some $\mu \geq 1$ then the iteration $w^{k+1} = (\mathcal{M} + \mathcal{A})^{-1} \mathcal{M} w^k$ as before becomes
\begin{equation} \label{eq_PDRQ_mu_lam}
\left\{
\begin{aligned}
b^k      &= \bar{u}^k - \tfrac{\sigma}{\mu + \sigma \nu}
K^\ast \big( (\mu -1) v^k + \bar{v}^k - \sigma f \big)\\
u^{k+1}  &= u^k + M^{-1}( b^k - T u^k)\\
v^{k+1}  &= \tfrac{1}{\mu + \sigma \nu} \big(
(\mu - 1) v^k + \bar{v}^k + \sigma( K u^{k+1} - g)  \big)\\
\bar{u}^{k+1} &= \bar{u}^k + (\id + \sigma \partial \tilde{F})^{-1}( 2 u^{k+1} - \bar{u}^k ) - u^{k+1}\\
\bar{v}^{k+1} &= \bar{v}^k + (\id + \sigma \partial \tilde{G^\ast})^{-1}( 2 v^{k+1} - \bar{v}^k ) - v^{k+1}.
\end{aligned}
\right.
\end{equation}
But now with $M = N_1 + \sigma Q + \tfrac{\sigma^2}{\mu + \sigma \nu}$ and $T = \id + \sigma Q + \tfrac{\sigma^2}{\mu + \sigma \nu}$.
This iteration converges if $N_1 - \id = M -T \geq 0$.

In the special case of $F$ being purely quadratic-linear, \ie $\tilde{F} = 0$, where the corresponding resolvent $(\id + \sigma \partial \tilde{F})^{-1} = \id$ it holds that $\bar{u}^k = u^k$ for all $k$ and so $\bar{u}$ is not needed.
But moreover in~\cite{Precond} it is shown that the assumption $N_1 - \id \geq 0 $ can be replaced by the assumption $N_1 \geq 0$ and $N_1 + \sigma Q > 0$.
This means that since $M = N_1 + \sigma Q + \tfrac{\sigma^2}{\mu + \sigma \nu} K^\ast K$ it is not needed that $M$ is a preconditioner for  $T = \id + \sigma Q + \tfrac{\sigma^2}{\mu + \sigma \nu} K^\ast K$ but only for $\tilde T = \sigma Q + \tfrac{\sigma^2}{\mu + \sigma \nu} K^\ast K$, \ie only $M-\tilde T \geq 0$ and $M >0$ must be assumed, which allows for more flexibility in choosing $M$.
Plugging this into \eqref{eq_PDRQ_mu_lam} PDRQ for the purely linear quadratic problem becomes
\begin{equation} \label{eq_PDRQ_pure_lq}
\left\{
\begin{aligned}
b^k      &= - \tfrac{\sigma}{\mu + \sigma \nu}
K^\ast \big( (\mu -1) v^k + \bar{v}^k - \sigma f \big)\\
u^{k+1}  &= u^k + M^{-1}( b^k - \tilde{T} u^k)\\
v^{k+1}  &= \tfrac{1}{\mu + \sigma \nu} \big(
(\mu - 1) v^k + \bar{v}^k + \sigma( K u^{k+1} - g)  \big)\\
\bar{v}^{k+1} &= \bar{v}^k + (\id + \sigma \partial \tilde{G^\ast})^{-1}( 2 v^{k+1} - \bar{v}^k ) - v^{k+1}.
\end{aligned}
\right.
\end{equation}
\section{Application to the unconstrained problem} \label{ssec_pdrq_unconst}
The theory from the last section will now be applied to the discretized unconstrained problem
\begin{equation} \label{op_unconst_h}
  \min_{u \in \mathbb{R}^{n \times n} } \, \frac{1}{2} \|  \Rh u - \unull  \|_2^2 + \lambda \| \nabh u \|_1, \tag{$P_h$}
\end{equation}
which is of the form $\min_u F(u) + G(K u)$.

The goal of splitting off a linear-quadratic term as in \eqref{eq_abs_LQ} which was
\begin{align}
  F(u)      &= \frac{1}{2} \scp{Qu}{u} + \scp{f}{u} + \tilde{F}(u), \\
  G^\ast(v) &= \frac{1}{2} \scp{Sv}{v} + \scp{g}{v} + \tilde{G^\ast}(v)
\end{align}
can be achieved for the first line be setting $Q = \Rh^\ast \Rh$ and $f = - \Rh^\ast \unull$ and $\tilde F =0$ where the constant term $\frac{1}{2} \norm{\unull}^2$ was omitted since it only shifts the value of the optimization problem $\min_u F(u) + G(K u)$.
It turn out to be useful to scale the gradient with a factor $\tau >0$ which then plays a similar role as the parameter of the same name in the Chambolle-Pock Algorithm~\ref{alg_abst_CP}.
The involved operators will then be $K= \tau \nabh$, $K^\ast = - \tau\divh$  this can be compensated for by scaling the regularization parameter $ \tilde \lambda = \frac{\lambda}{\tau}$, \ie actually the algorithms are applied to the problem 
\[ 
\min_{u \in X } \, \frac{1}{2} \|  \Rh u - \unull  \|_2^2 + \frac{\lambda}{\tau} \| \tau \nabh u \|_1 .
\]
Thus for $G^\ast$ it holds that $G^\ast = \indicator_{\{ \norm[\infty]{y} <  \tilde \lambda\}}$ and so $S= g = 0$ and  $\tilde G^\ast = \indicator_{\{ \norm[\infty]{y} < \tilde\lambda\}}$ can be set.
The corresponding resolvent is $(\id + \sigma \partial G^\ast)^{-1} = \mathcal{P}_{B^\infty_{\tilde \lambda}(0)}$.
Since $S = 0 = \nu \id$  it holds that $\nu =0$.
Choosing $\mu = 1$ the operator  $ T = \sigma Q + \tfrac{\sigma^2}{\mu + \sigma \nu} K^\ast K= \sigma \Rh^\ast \Rh - (\sigma \tau)^2 \laph$, where $\laph = -\divh \nabh$ is the discrete Laplacian.
Plugging $b^k = - \sigma K^\ast \bar{v}^k - \sigma f$ directly into the iteration \eqref{eq_PDRQ_pure_lq} reduces to
\begin{align} \label{iter_pdrq_unconst_unfin}
u^{k+1}  &= u^k + M^{-1}( \sigma \tau \divh \bar{v}^k - \sigma f - \sigma \Rh^\ast \Rh u^k - (\sigma \tau)^2 \laph u^k)\\
v^{k+1}  &= \bar{v}^k + \sigma \tau \nabh u^{k+1}   \\
\bar{v}^{k+1} &= \bar{v}^k + (\id + \sigma \partial \tilde G^\ast)^{-1}( 2 v^{k+1} - \bar{v}^k ) - v^{k+1}.
\end{align}
This is formulated as a ready to use algorithm in Algorithm~\ref{alg_PDRQS_unconst} where $M$ was chosen as an inverse norm preconditioner, this and some other possible choices for the  preconditioner $M$  are now discussed.
\section{Suitable Preconditioners} \label{ssec_suit_precond}
The goal of this section is to find a preconditioner $M$ for $ T = \alpha  \Rh^\ast \Rh - \beta \laph$ with $\alpha, \beta >0$, that is $M$ a symmetric positive definite operator such that $M-T$ is positive semi-definite as required by PDRQ in the purely linear quadratic case.
One simple choice is to use $M = m \id$ with a constant $m > \norm{T}$, clearly $M$ is symmetric and $M-T \geq 0$, this is called Richardson preconditioning.
In practice $m$ can be found by applying  power iteration~\cite{saad2011numerical} to $T$ which gives $m \geq \norm{T}$.
It is however desirable to use more of the structure of the operator, as a better approximation of the  inverse is expected and thus faster convergence.
In order to find a preconditioner $M$ for $T$  preconditioners $M_1$ for $\Rh^\ast \Rh$ and $M_2$ for $-\laph$ are constructed separately then $M = \alpha M_1 + \beta M_2$ will be a preconditioner for $T$.
For $M_2$ it is possible to make use of the periodic Laplacian $ \Delta_p$ in the form of the operator defined by  $M_2 u = - \Delta_p u =  u \ast \kappa_\Delta$ where $\ast$ denotes periodic convolution and 
\begin{equation}
\kappa_\Delta =
\begin{bmatrix}
 0 & -1 &  0  \\
-1 &  4 & -1  \\
 0 & -1 &  0
\end{bmatrix}.
\end{equation}
Then $M_2$ is symmetric and it holds that
\begin{equation*}
\scp{-\Delta_p u - (-\laph)  u}{u} = \scp{-\Delta_p u}{u}-\scp{ -\Delta_h  u}{u}
= \norm{ \nabla_p u }^2 - \norm{\nabh u }^2 \geq 0
\end{equation*}
since the periodic gradient $\nabla_p$ coincides with $\nabh$ except on the boundary were the entries of $\nabh$ are $0$ but those of $\nabla_p$ might not.
For the choice of $M_1$ imagine that the operator $\Rh$  were the periodic convolution with some kernel \ie $\Rh u = u \ast \kappa$ then simply letting $M u = \alpha \kappa' \ast \kappa \ast u + \beta \kappa_\Delta \ast u$ -- where $\kappa'$ denotes the mirrored kernel of $\kappa$ -- would give a feasible preconditioner since convolution with $\kappa'$ is the adjoint to convolution with $\kappa$.
This could then be inverted using the discrete 2d Fourier transform $\fzdh$ and its inverse due to the convolution theorem:
\begin{equation}
M^{-1} u = \fzdh^{-1}\left(
\frac{\fzdh(u)}{\alpha \fzdh(\kappa)^2 + \beta \fzdh(\kappa_{\Delta})} \right).
\end{equation}
Calculation $\fzdh(\kappa_{\Delta})$ yields:
\begin{align*}
    (\fzdh(\kappa_{\Delta}))_{i,j}
&=
    \sum_{k,l=0}^{n-1} x_{k,l} e^{-\frac{2 \pi \iu}{n} (ik + jl)}  \\
&=4 + x_{1,0} e^{-\frac{2 \pi \iu}{n} i } + x_{n-1,0} e^{\frac{2 \pi \iu}{n} i }
      x_{0,1} e^{-\frac{2 \pi \iu}{n} j } + x_{0,n-1} e^{\frac{2 \pi \iu}{n} j }   \\
&= -( e^{-\frac{2 \pi \iu}{n} i } -2 + e^{\frac{2 \pi \iu}{n} i }) +
   -( e^{-\frac{2 \pi \iu}{n} j } -2 + e^{\frac{2 \pi \iu}{n} j }  ) \\
&= -( e^{-\frac{  \pi \iu}{n} i }    + e^{\frac{  \pi \iu}{n} i }  )^2 +
   -( e^{-\frac{  \pi \iu}{n} j }    + e^{\frac{  \pi \iu}{n} j }  )^2 \\
&= 4 \sin\ndel{\frac{ \pi i}{n} }^2 + 4 \sin\ndel{\frac{ \pi i}{n} }^2
\end{align*}
which is always non-negative.
Thus the denominator is not $0$, since $(\fzdh(\kappa_{\Delta}))_{i,j} = 4\sin(i\pi/n)^2 + 4 \sin(j\pi / n)^2 =0 $ if and only if $i = j =0$ this means that it needs to be ensured that $\fzdh(\kappa)_{0,0}$ which equals $\sum_{i,j} \kappa_{i,j}$ is not $0$.
This approach  is used in~\cite{Precond} for the deblurring case, where the operator which is blurring the image is indeed the convolution with a kernel.
It is not the case that $\R$ is a convolution with some kernel however $\R^\ast \R$ is, this motivates the following approaches.

\subsection{Inverse norm preconditioning} \label{sssec_precond_inv_norm}
In Lemma~\ref{lem_RR_uk} it was shown that $\R^\ast \R u = \frac{1}{\norm{\cdot}} \ast u$ which by the convolution theorem implies that $ \four( \R^\ast \R u) = 2 \pi \four( \frac{1}{\norm{\cdot}}) \four(u)$, where $\four$ denotes the 2-d Fourier transform.
In addition to that from the theory of the Radon transform it is known that
\begin{equation}
  \frac{r}{2\pi} \four( \R^\ast \R u)(r\omega) = \four(u)(r \omega)
\end{equation}
see \cite{epstein}*{p 142} where a different scaling of the Fourier transform and the Radon transform for angles $\varphi \in [0, \pi]$ are used.
Both combined imply that
\begin{equation}
  \four\ndel{\frac{1}{\norm{x}}} = \frac{1}{\norm{x}}.
\end{equation}
This is also known from a connection of the Fourier transform of radially symmetric functions and the Hankel transform see~\cite[sec 9.3]{poularikas2010transforms}.
This means that solving $\R^\ast R u = y $ can be achived by deconvolution with a kernel whose Fourier transform is known.

This idea is now transferred to the discrete setting, where it is generally not true that $\Rh^\ast \Rh$ can be described by a discrete convolution but the idea from the continuous setting is used to efficiently find an approximation of the solution of the linear system $\Rh^\ast \Rh u = y$ for given $y$.
Let $\norm[\varepsilon]{x} = \sqrt{\norm[2]{x}^2 + \varepsilon^2 }$ for some small $\varepsilon >0$ be set as an approximate of norm of $x$ this is done in order to guarantee that the reciprocal can be calculated at every point.
Let $\fzdh$ be the 2-dimensional discrete Fourier transform and $\fzdh^{-1}$ its inverse then then by the convolution theorem $\fzdh( u \ast \kappa) = \fzdh (u)\fzdh(\kappa)$
Therefore the above idea to use deconvolution with a kernel whose Fourier transform equals $\frac{1}{\norm[\varepsilon]{x}}$ leads to a preconditioner of the form $\tilde M u =\fzdh^{-1}( \frac{1}{\| x \| } \fzdh(u))$.
In order to guarantee the feasibility of the preconditioner set $M_1 = c \tilde M  $ and choose $c>0$ great enough to ensure that  $M_1 - \Rh^\ast \Rh \geq 0$.
In practice this can be achieved by performing power iteration on $\tilde M^{-1} \Rh^\ast \Rh$ and then choosing $c$ greater that the largest eigenvalue will ensure that $c \geq \tilde M ^{-1} \Rh^\ast \Rh $, which gives $c \tilde M - \Rh^\ast \Rh \geq 0 $.
Putting this together with the  periodic Laplacian gives
\begin{align}
Mu = c \alpha \fzdh^{-1} \left( \frac{1}{\norm[\varepsilon]{x}} \fzdh(u)\right) + \beta \Delta_p u = y \\
\alpha c \frac{1}{\norm[\varepsilon]{x}} \fzdh(u) + \beta \fzdh( \kappa_\Delta) \fzdh(u) = \fzdh(y) \\
\fzdh(u) = \fzdh(y) \cdot
\left( \alpha c \frac{1}{\norm[\varepsilon]{x}} + \beta \fzdh( \kappa_\Delta) \right)^{-1} \\
u = M^{-1} y = \fzdh^{-1} \left( \fzdh(y) \frac{\norm[\varepsilon]{x}}{ \alpha c + \beta \fzdh(\kappa_\Delta) \norm[\varepsilon]{x}} \right).
\end{align}
A ready-to-use algorithm for solving the discrete problem~\eqref{op_unconst_h} with this preconditioner is provided in Algorithm~\ref{alg_PDRQS_unconst}.
Note that in the algorithm the calculation of $u_{tmp}$ is formulated such that only one application of $\Rh^\ast$ and $\Rh$ is used per iteration.
\begin{algorithm}
\caption{\textbf{PDRQS2} algorithm for minimizing \eqref{op_unconst_h}}
\begin{algorithmic} \label{alg_PDRQS_unconst}
\REQUIRE  \textbf{Input}:  $\unull \in \mathbb{R}^{N \times M}, \,c>0, \, \sigma >0, \, \tau >0, \, h>0 $ \\
\REQUIRE \textbf{Initial values}: $u, \bar v $
\FOR{$k < k_{max}$}
\STATE $u_{tmp} \leftarrow \sigma \tau \divh \bar v + \sigma \Rh^\ast( \unull -  \Rh u)  -(\sigma \tau)^2 \laph u$
\STATE $u_{new} \leftarrow u + \fzdh^{-1}  \ndel{ \fzdh(u_{tmp})
\frac{ \norm[\varepsilon]{x}}{ \sigma c + (\sigma \tau)^2 \fzdh(\kappa_\Delta)\norm[\varepsilon]{x}} } $
\STATE $v \leftarrow \bar v + \sigma \tau \nabh u_{new}$

\STATE $w_{i,j} \leftarrow \frac{2v_{i,j}-\bar v_{i,j}}{ \max \{ 1, \tau \abs{ 2 v_{i,j} - \bar v_{i,j} } / \lambda \} }$
\STATE $\bar v \leftarrow \bar v + w - v$
\STATE $u \leftarrow u_{new}$
\ENDFOR
\end{algorithmic}
\end{algorithm}


\subsection{Impulse response preconditioners} \label{ssec_psf_precond}
Another idea for a preconditioner is to apply $\Rh^\ast \Rh $ to a discrete point function $\delta$.
Again if $\Rh^\ast \Rh $ were a convolution with some kernel this impulse response $\kappa = \Rh^\ast \Rh \delta$ would recover it, so it makes sense to apply deconvolution with $\kappa$.
In the case where the image dimension $n$ is even use $\delta$ = $ \delta_{ \frac{n}{2}, \frac{n}{2}}$ that is the lower right of the middle four entries.
The obtained kernel $\tilde k = \Rh^\ast \Rh \delta$ needs to be be made symmetric via $k = \frac{1}{2}(\tilde{k} + \tilde{k}^T)$, in order to guarantee that the resulting preconditioner is symmetric.
Then the full kernel is $ K = \alpha k + \beta \kappa_{\Delta}$.
This gives the preconditioner
\begin{equation}
    M^{-1}u = \fzdh^{-1}\ndel{ \frac{\fzdh(u)}{\fzdh(K)}}.
\end{equation}
Again a constant $c>$ must be chosen such that $cM -T \geq 0$, again in practice this can be achieved  with power iteration.

\subsection{Circulant matrix preconditioner} \label{ssec_circ_precond}
Another idea similar to the previous was used in~\cite{ramani2012splitting}, it is to approximate $\Rh^\ast \Rh$ by a circulant matrix, that is a matrix of the form
\begin{equation}
  C =
  \begin{pmatrix}
      c_{1  } & c_{p} & \cdots & c_{2} \\
      c_{2}   & c_{1} & \cdots & c_{3} \\
      \vdots  & \vdots  & \ddots & \vdots  \\
      c_{p}    & c_{p-1} & \cdots & c_{1}
  \end{pmatrix}.
\end{equation}
Then the Matrix $C$ can be described by a single vector $c = (c_1, c_2, \ldots, c_p)$ and the application of $C$ to a vector $x$ can be written as the circular convolution of $c$ with $x$.
Consequently the solution of $Cx = b$ is $x = \fldh^{-1}( \fldh(b)/ \fldh(c))$, where $\fldh$ denotes the 1-d discrete Fourier transform.
Again $\tilde k = \Rh^\ast \Rh \delta$ is computed, in order to ensure symmetry of the preconditioner the circulant matrix with $c= k$ with $k_i = \frac{1}{2}(\tilde k_i + \tilde k_{ (n^2 -i \mod n^2)})$ for $i = 1, \ldots, n^2$.
The same procedure is used to obtain a Laplace kernel $\kappa_\Delta$ from $\laph \delta$ then the combined kernel $K= \alpha k + \beta \kappa_\Delta$ is used as a preconditioner
\begin{equation}
    M^{-1} u = \fldh^{-1} \ndel{ \frac{ \fldh(u)}{\fldh(K)} }.
\end{equation}
\section{Constrained Problem}
For the constrained problem
\begin{equation} \label{op_const_h}
  \min_{u\in X} \quad \frac{1}{2} \norm[2]{ \Rh u - \unull}^2 + \lambda \norm[1]{\nabh u} + \indicator_{ \{\Rh u\mid_{\Dm_0} \geq C \}}(u), \tag{$C_h$}
\end{equation}
the same splitting as in~\ref{ssec_pdrq_unconst} cannot be applied since the terms depending on $\Rh$ are not quadratic, due to the indicator function.
Instead  the problem \eqref{op_const_h} is formulated as $\min_{u \in X} F(u) + G( K u)$ with
\begin{equation*}
    G(x,y) =
    \frac{1}{2}\norm[L^2(\Dm \setminus \Dm_0)]{x - \unull}^2  + \lambda \norm[1]{y}
    +  \indicator_{ \{x\mid_{\Dm_0} \geq C \}}(x),
    \qquad K =
    \begin{bmatrix}
    \Rh \\ \nabh
    \end{bmatrix}.
\end{equation*}
Brining it in the quadratic-linear form of \eqref{eq_abs_LQ} which reads as
\begin{align}
F(u) &= \frac{1}{2} \scp{Qu}{u} + \scp{f}{u} + \tilde{F}(u), \\
G^\ast(v) &= \frac{1}{2} \scp{Sv}{v} + \scp{g}{v} + \tilde{G}^\ast(v)
\end{align}
can be achieved with $F = \tilde{F} = Q = f = 0$ and $S = g = 0$ and $ \tilde{G}^\ast = G^\ast $.
In this case $G^\ast(\xi,w) = G^\ast_1(\xi) + G^\ast_2(w)$ with:
\begin{equation} \label{eq_G_ast_1}
G^\ast_1(\xi) = \sum_{(i,j) \in \Dm \setminus \Dm_0} g_1^\ast (\xi_{i,j}) +
           \sum_{(i,j) \in \Dm_0} g_2^\ast( \xi_{i,j})
\end{equation}
with
\begin{equation*}
g_{1,1}^\ast (\xi) = \frac{1}{2}\xi^2 + \xi \unull, \quad \quad
g_2^\ast(\xi_{i,j}) =
    \begin{cases}
    \infty,                     &\text{if } \xi_{i,j} \geq 0  \\
    \xi_{i,j}C,                 &\text{if } \xi_{i,j} < 0.
    \end{cases}
\quad
\end{equation*}
As well as $G^\ast_2(w) = \indicator_{B_{\lambda}^{\infty}(0)}(w)$.
This fits the case of a purely quadratic-linear $F$, as derived in\eqref{eq_PDRQ_pure_lq},thus PDRQ reduces to:
\begin{align*}
b^k           &= - \sigma K^\ast \bar{v}^k     \\
u^{k+1}       &= u^k + M^{-1}( b^k - T u^k) =  \\
              &= u^k + M^{-1}( - \sigma K^\ast \bar{v}^k - \sigma^2 K^\ast K u^k)\\
v^{k+1}       &= \bar{v}^k + \sigma K u^{k+1}  \\
\bar{v}^{k+1} &= \bar{v}^k + (\id + \sigma \partial \tilde G^\ast)^{-1}( 2 v^{k+1} - \bar{v}^k ) - v^{k+1}.
\end{align*}
Now with $\bar v = ( \bar v_1, \bar v_2)$ so that $K^\ast \bar v = \Rh^\ast \bar v_1 - \dive \bar v_2$, and
\begin{equation*}
K^\ast K u = K^\ast (\Rh u, \nabh u)^T = \Rh^\ast \Rh u - \laph u.
\end{equation*}
Thus the argument that $M^{-1}$ is applied to is
\begin{align}
  &- \sigma (\Rh^\ast \bar v_1 - \divh \bar v_2 )  - \sigma^2 (\Rh^\ast \Rh u - \laph u)  \\
= &- \sigma \big( \Rh^\ast( \bar v_1 + \sigma \Rh u) -\divh( \bar v_2 + \sigma \nabh u) \big),
\end{align}
which means that $\Rh^\ast$ will need to be applied once per iteration, also if $\Rh u$ was saved from the previous iteration $\Rh$ will also be applied once per iteration.

The resolvents were calculated in~\ref{lem_f_ast_resolv1} to be
\begin{equation}
(\id + \sigma \partial G^\ast)^{-1}(\bar v, \bar w) =
    \begin{pmatrix}
    (\id + \sigma \partial G^\ast_1)^{-1}(\bar v) \\
    (\id + \sigma \partial G^\ast_2)^{-1}(\bar w)
    \end{pmatrix}
\end{equation}
with $(\id + \sigma \partial G^\ast_2)^{-1}(\bar w) = \mathcal{P}_{B_{\lambda}^\infty(0)}(\bar w)$ and
\begin{equation}
(\id + \sigma \partial G^\ast_1)^{-1}(\bar v_{i,j}) =
    \begin{cases}
    \frac{\bar v_{i,j}-\sigma (\unull)_{i,j}}{1+\sigma},     &\text{if } (i,j) \in \Dm \setminus \Dm_0  \\
    \min \{\bar v_{i,j} - \sigma C,0 \} ,                 &\text{if } (i,j) \in \Dm_0. \\
    \end{cases}
\end{equation}
In this case let $Q=0$, so now $T=\sigma^2 K^\ast K $ needs to be preconditioned which can be achieved by the preconditioner from~\ref{sssec_precond_inv_norm}
\begin{equation*}
M^{-1} y = \fzdh^{-1} \left( \fzdh(y) \frac{\| x \|}{ \sigma^2 c + (\sigma \tau)^2 \fzdh(\kappa_\Delta) \| x\| }\right).
\end{equation*}
Combining these gives Algorithm~\ref{alg_PDRQS_const} for minimizing the constraint problem \eqref{op_const_h}.

\begin{algorithm}
\caption{\textbf{PDRQS1} algorithm for minimizing \eqref{op_const_h}}
\begin{algorithmic} \label{alg_PDRQS_const}
\REQUIRE  \textbf{Input}:  $\unull \in \mathbb{R}^{N \times M},\Dm_0, C >0,\,c>0, \, \sigma >0, \, \tau >0, \, h>0 $ \\
\REQUIRE \textbf{Initial values}: $u, \bar v_1, \bar v_2$
\FOR{$k < k_{max}$}
\STATE $u_{tmp} \leftarrow - \sigma \big(
\Rh^\ast( \bar v_1 + \sigma \Rh u) - \tau \dive( \bar v_2 + \sigma \tau \nabla u) \big)$
\STATE $u_{new} \leftarrow u + \fzdh^{-1}  \ndel{ \fzdh(u_{tmp})
\frac{ \norm{x} }{ \sigma^2 c + (\sigma \tau)^2 \fzdh(\kappa_\Delta)\norm{x}} } $
\STATE $v_1 \leftarrow \bar v_1 + \sigma \Rh u_{new}$
\STATE $v_2 \leftarrow \bar v_2 + \sigma \tau \nabla u_{new}$

\STATE $(v_{tmp})_{i,j} \leftarrow
\begin{cases}
    \frac{ (2v_1 - \bar v_1)_{i,j}-\sigma (\unull)_{i,j}}{1+\sigma},
    &\text{if } (i,j) \in \Dm \setminus \Dm_0\\
\min \{(2v_1 - \bar v_1)_{i,j} - \sigma C,0 \} ,
    &\text{if } (i,j) \in \Dm_0. \\
\end{cases}$
\STATE $ \bar v_1 \leftarrow \bar v_1 + v_{tmp} - v_1$

\STATE $w_{i,j} \leftarrow \frac{2(v_2)_{i,j}- (\bar v_2)_{i,j}}{ \max \{ 1, \tau\abs{ 2 (v_2)_{i,j} - (\bar v_2)_{i,j} } / \lambda \} }$
\
\STATE $\bar v_2 = \bar v_2 + w - v_2$

\STATE $u \leftarrow u_{new}$
\ENDFOR
\end{algorithmic}
\end{algorithm}
So far two ways of splitting of a quadratic functional were used.
Another possibility is to take the quadratic term of $G^\ast_1$ from~\eqref{eq_G_ast_1} that is the one which comes from $g^\ast_{1,1} = \frac{1}{2} \xi^2 + \xi \unull$ into $S$ and $g$.
This gives
\begin{equation}
  S = \begin{bmatrix}\id_\Dm &0 \\0 &0\end{bmatrix}
\end{equation}
with $\id_\Dm = \chi_{\Dm \setminus \Dm_0}$ and $ g = \unull\cdot \chi_{\Dm \setminus \Dm_0}$.
The corresponding resolvent will then be
\begin{equation}
(\id + \sigma \partial G^\ast_1)^{-1}(\bar v_{i,j}) =
    \begin{cases}
    \bar v_{i,j} ,                &\text{if } (i,j) \in \Dm \setminus \Dm_0  \\
    \min( \frac{\bar v_{i,j}}{1+\sigma C},0),     &\text{if } (i,j) \in \Dm_0.  \\
    \end{cases}
\end{equation}
The following table gives an overview of the splitting used and names the corresponding algorithms. No major difference in the performance of these algorithms were observed.
\begin{table}[h]
  \centering
    \begin{tabular}{l | ccc ccc c c }
         &$Q$ 	         &$f$  & $K$ & $S$ & $g$ & $ \tilde{G^\ast}$ & constraints \\ \hline
    PDRQ1	 &$0$ 	         &$0$                   & $\begin{bmatrix} \Rh \\  \tau\nabh \end{bmatrix}$  &$0$       & $0$ & $G^\ast$ & yes  \\
    PDRQ2	 &$\Rh^\ast \Rh$ &$- \Rh^\ast \unull$   & $\tau\nabh $                                       &$0$       & $0$ & $\indicator_{ \norm[\infty]{y} < \frac{\lambda}{\tau}}$  & no \\
    PDRQ3	 &$0$ 	         &$0$                   & $\begin{bmatrix} \Rh \\  \tau\nabh \end{bmatrix}$  &$\begin{bmatrix} \id_{\Omega} &0 \\0 &0\end{bmatrix}$ &$\begin{bmatrix}\unull \\ 0 \end{bmatrix}$ &$\indicator_{ \norm[\infty]{y} < \frac{\lambda}{\tau}}$ & yes \\

    \end{tabular}
  \caption{Various splittings for the quadratic-linear problem keeping $\tilde{F}=0$.} \label{tab_splitting}
\end{table}
\FloatBarrier
\section{Non-negativity constraints} \label{ssec_non_neg_u}
As a possible extension of the model one may try to enforce physical constrain of non-negative radio density.
Since solutions of the above mentioned problems might have negative values, it is possible to try to enforce non-negativity constraints \ie $u \geq 0$.
Thus try to solve
\begin{equation} \label{opt_min_const_non_neg}
\min_{u\in X} \quad \frac{1}{2} \norm[2]{ \Rh u - \unull}^2 + \lambda \norm[1]{\nabh u}
+ \indicator_{ \{\Rh u\mid_{\Dm_0} \geq C \}}(u) + \indicator_{\{u \geq 0\}}.
\end{equation}
Then $F$ will not be linear quadratic anymore but still $F = \tilde F = \indicator_{[0, \infty)}$.
Or more precisely $\tilde{F}(u) = \int_M \tilde{f}(u(x)) \dd x$ with  $\tilde f = \indicator_{[0, \infty)}$.
Therefore calculate $ (\id + \sigma \partial \tilde F)^{-1}$.
\begin{lemma} \label{lem_resolv_non_neg}
The function $\tilde f(x) = \indicator_{[d,\infty)}(x)$ has the resolvent
\begin{equation*}
(\id + \sigma \partial \tilde f)^{-1} (y) = \max( y, d).
\end{equation*}
for any $ \sigma >0$.
\end{lemma}
\begin{proof}
It holds that
\begin{equation*}
\partial \tilde f (x)=
    \begin{cases}
    0, &\text{if } x > d, \\
    (-\infty,0], &\text{if } x = d, \\
    \emptyset, &\text{if } x < d,
    \end{cases}
    \qquad
(\id + \sigma \partial \tilde f )(x)=
    \begin{cases}
    x, &\text{if } x > d, \\
    (-\infty,d], &\text{if } x = d, \\
    \emptyset, &\text{if } x < d.
    \end{cases}
\end{equation*}
Thus in $(\id + \sigma \partial \tilde f )(x) = y$ for $y > d$ only $x=y$ is a solution and for $y \leq d$ only $ x = d$ is solution. Thus $(\id + \sigma \partial \tilde f )^{-1}(y) = \max(y,d)$.
This corresponds to solving $\min_x \frac{1}{2} \norm{x-y} + \sigma \indicator_{[d,\infty)}(x)$.
\end{proof}
This time $F$ will not be linear quadratic, thus use the iteration (3.5) from \cite{Precond}, again using $\mu = 1$ and $\nu = 0$.
Now to precondition $ T = \id + \sigma^2 K^\ast K$ choose the operator $ Mu = u + c \sigma^2 \fzdh^{-1} \left( \frac{1}{\| x \|} \fzdh(u)\right) - (\sigma \tau)^2 \Delta_p u $ and calculate $M^{-1}$:
\begin{align}
u + c \sigma^2 \fzdh^{-1}\left( \frac{1}{\norm{x}} \fzdh(u)\right) - (\sigma \tau)^2 \Delta_p u = Mu = y \\
\fzdh(u) + c \sigma^2  \frac{1}{\norm{x}} \fzdh(u) + (\sigma \tau)^2 \fzdh( \kappa_\Delta) \fzdh(u) = \fzdh(y) \\
\fzdh(u) = \fzdh(y) \cdot
\left( 1 +  c \sigma^2  \frac{1}{\norm{x}} + (\sigma \tau)^2 \fzdh( \kappa_\Delta) \right)^{-1} \\
u = M^{-1} y = \fzdh^{-1} \left( \fzdh(y) \frac{\norm{x}}{ \norm{x} + c \sigma^2  + (\sigma \tau)^2 \fzdh(\kappa_\Delta) \norm{x} }\right).
\end{align}
This can now be used as a preconditioner in Algorithm~\ref{alg_PDRQS3_const_non_neg}.
\begin{algorithm}
\caption{\textbf{ PDRQS1} algorithm for minimizing \eqref{opt_min_const_non_neg}}
\begin{algorithmic} \label{alg_PDRQS3_const_non_neg}
\REQUIRE  \textbf{Input}:  $\unull \in \mathbb{R}^{N \times M}, C >0,\,c>0, \, \sigma >0, \, \tau >0, \, h>0 $ \\
\REQUIRE \textbf{Initial values}: $u, \bar{u}, \bar v_1, \bar v_2$
\FOR{$k < k_{max}$}
\STATE $u_{tmp} \leftarrow \bar u - u - \sigma \big(
\Rh^\ast( \bar v_1 + \sigma \R u) - \tau \dive( \bar v_2 + \sigma \tau \nabla u) \big) $
\STATE $u_{new} \leftarrow u + \fzdh^{-1} \ndel{ \fzdh(u_{tmp})
\frac{ \norm{x} }{ \norm{x} + \sigma^2 c + (\sigma \tau)^2 \fzdh(\kappa_\Delta)\norm{x}} } $
\STATE $v_1 \leftarrow \bar v_1 + \sigma \R u_{new}$
\STATE $v_2 \leftarrow \bar v_2 + \sigma \tau \nabh u_{new}$

\STATE $(v_{tmp})_{i,j} \leftarrow
\begin{cases}
    \frac{ (2v_1 - \bar v_1)_{i,j}-\sigma (\unull)_{i,j}}{1+\sigma},
    &\text{if } (i,j) \in \Dm \setminus \Dm_0\\
\min \{(2v_1 - \bar v_1)_{i,j} - \sigma C,0 \} ,
&\text{if } (i,j) \in \Dm_0. \\
\end{cases}$
\STATE $ \bar v_1 \leftarrow \bar v_1 + v_{tmp} - v_1$

\STATE $w_{i,j} \leftarrow \frac{2(v_2)_{i,j}- (\bar v_2)_{i,j}}{ \max \{ 1, \tau \abs{ 2 (v_2)_{i,j} - (\bar v_2)_{i,j} } / \lambda \} }$
\
\STATE $\bar v_2 = \bar v_2 + w - v_2$

\STATE $ \bar u \leftarrow \bar u + \max \{ 2 u_{new} - \bar u, 0 \} - u_{new}$
\STATE $u \leftarrow u_{new}$
\ENDFOR
\end{algorithmic}
\end{algorithm}

\section{Alternating Direction Method of Multipliers} \label{sec_ADMM}
In this section the method proposed in~\cite{ramani2012splitting} of using the Alternating Direction Method of Multipliers(\myacro{ADMM}) is discussed and implemented for the problems in question. 
Consider again the unconstrained problem:
\begin{equation*}  \label{opt_unconst_h2}
  \min_{ u \in X} \quad \frac{1}{2} \norm[2]{ \Rh u - \unull}^2 + \lambda \norm[1]{ \nabh u} \tag{$P_h$}.
\end{equation*}
This can be rewritten as a constrained problem by introducing the additional variables $v$ and $w$:
\begin{equation}
\min_{u, v, w } \, \frac{1}{2} \norm[2]{w - \unull}^2 + \lambda \norm[1]{v} \\
, \, \text{s.t. } w = \Rh u, \, v = \nabh u.
\end{equation}
Since both of the constraints are linear this can be more compactly written with a single constraint matrix:
\begin{equation} \label{opt_admm_abst}
\min_{u, z } \, f(z)
\text{ s.t. } z = K u
\end{equation}
with $z = \begin{bmatrix}w, v \end{bmatrix}^T$ and
    $ K = \begin{bmatrix}\Rh \\ \nabh \end{bmatrix}$ as well as $f(z) = f(\begin{bmatrix}w, v \end{bmatrix}^T) =  \frac{1}{2} \norm[2]{w - \unull}^2 + \lambda \norm[1]{v}$.

The associated augmented Lagrangian with Lagrange multiplier $\gamma $ and parameter $\mu \in \mathbb{R}$ is
\begin{equation*}
\mathcal{L}(u,z,\gamma) = f(z) + \gamma^T(z-K u) + \frac{\mu}{2}\norm[\Lambda]{z - K u}^2
\end{equation*}
where $\Lambda$ is a symmetric weighting matrix and  $\norm[\Lambda]{x} = \norm{\Lambda x}$.
Indeed if $\lambda$ is chosen as
\begin{equation*}
\Lambda =
\begin{bmatrix}
I_1 &0 \\
0 &\tau I_2.
\end{bmatrix}
\end{equation*}
where with $I_1$ the identity on $\mathds{R}^{N \times M}$ and $I_2$ the identity on $\mathds{R}^{n \times n \times 2}$, this results in the same scaling of the gradient with a factor $\tau$ as in the PDRQ and Primal-Dual algorithms, specifically $\norm[\Lambda]{[w,v]^T} = \norm[2]{w} + \tau \norm[2]{v}$.
Putting the terms containing $z-Ku$ together and completing the square gives a term $c_\eta$ that does not depend on $u$ or $z$
\begin{equation} \label{eq_admm_lagr}
\mathcal{L}(u,z,\eta) = f(z) + \frac{\mu}{2}\norm[\Lambda]{z - K u - \eta}^2 + c_\eta
\end{equation}
with $\eta = -\frac{1}{\mu} \Lambda^{-1}\gamma$
and $c_\eta = - \frac{\mu}{2}\norm[\Lambda]{\eta}^2$ independent of $u$ and $z$.
The alternating direction method consists of minimizing in each component separately
\begin{align}
    u^{(j)}_\ast &= \argmin_u \mathcal{L}(u,z^{(j)}  ,\eta^{(j)}) \label{eq_admm_u_upd}\\
    z^{(j)}_\ast &= \argmin_z \mathcal{L}(u^{(j+1)},z,\eta^{(j)}) \label{eq_admm_z_upd}
\end{align}
and then updating $\eta$ according to
\begin{equation*}
\eta^{(j+1)} = \eta^{(j)} - \mu ( z^{(j+1)} - K u^{(j+1)} ).
\end{equation*}
In~\cite{ramani2012splitting} it is shown that if~\eqref{eq_admm_u_upd} and~\eqref{eq_admm_z_upd} are solved approximately, that is choosing $\norm{u^{(j+1)} - u^{(j)}_\ast}< \varepsilon^{(j)}_u$ and $\norm{z^{(j+1)} - z^{(j)}_\ast}< \varepsilon^{(j)}_z$ with $\sum \varepsilon^{(j)}_u < \infty$ and  $\sum \varepsilon^{(j)}_z < \infty$, then the generated sequence $(u^{(j)},z^{(j)})$ will converge to a solution of~\eqref{opt_admm_abst}.
Since both $f(z)$ and $c_\eta$ in~\eqref{eq_admm_lagr} do not depend on $u$ the minimization with respect to $u$ in~\eqref{eq_admm_u_upd} is equivalent to minimizing $\norm{\Lambda (z^{(j)} - \eta^{(j)} ) - \Lambda K u }^2$ which can be achieved by solving the normal equation
\begin{equation}
K^T \Lambda^2 K u = K^T \Lambda^2 (z^{(j)} -\eta^{(j)} ).
\end{equation}
By the definitions of $K$ and $\Lambda$ this is
\begin{equation} \label{eq_admm_pcg}
    (\Rh^\ast \Rh - \tau^2 \laph)u = \Rh^\ast(z^{(j)} _w - \eta^{(j)} _w) - \tau^2 \divh (z^{(j)} _v-\eta^{(j)} _v)
\end{equation}
which can be solved approximately using the preconditioned conjugate gradient(\myacro{PCG}) method where the matrix to be preconditioned $K^\ast K$  is the same as in Section~\ref{ssec_suit_precond}, in the original paper the cone-type preconditioner presented in Section~\ref{ssec_circ_precond} was used, but all other preconditioners shown from this thesis can be used instead.
The minimization in \eqref{eq_admm_z_upd} decouples into one over $w$ and one over $v$, as both in the quadratic term in~\eqref{eq_admm_lagr} and $f(z)$ the two components of $z = (w,v)$ never appear together.
The term to be minimized for $w$ is
\begin{align}
w^{(j+1)} &= \argmin_w \frac{1}{2} \norm[2]{  w - \unull }^2 + \frac{\mu}{2}\norm[2]{w - \Rh u^{(j+1)} - \eta^{(j)}_w}^2
\end{align}
the solution of this quadratic problem is
\begin{align}
u^{(j+1)} &= \frac{1}{1 + \mu} ( \unull + \mu( \R u^{(j+1)} + \eta^{(j)}_w ) ).
\end{align}
The minimization over $v$ is
\begin{align*}
v^{(j+1)} &= \argmin_v \lambda \norm[1]{ v } + \frac{\mu \nu }{2}\norm[2]{ v - \nabla u^{(j+1)} - \eta^{(j)}_v }^2
\end{align*}
the solution of which can be calculated explicitly see Lemma~\ref{lem_soft_thres2d}.
Note that in the definition of the discretization of the $TV$ semi-norm the isotropic norm was chosen, that is
\begin{equation*}
\norm[1]{ v } = \sum_{i,j} \sqrt{ (v_{i,j,1})^2 + (v_{i,j,2})^2}.
\end{equation*}
The corresponding resolvent is called soft thresholding and is calculated in the following lemma.
\begin{lemma} \label{lem_soft_thres2d}
For $ \tau > 0 $ and $u \in \mathbb{R}^2$ with the usual Euclidean norm $\norm{u} = \sqrt{u_1^2 + u_2^2}$ the solution of
\begin{equation}
    \min_{x \in \mathbb{R}^2} f(x) = \min_{x \in \mathbb{R}^2} \, \frac{1}{2} \norm{x - u}^2 + \tau \norm{x}
\end{equation}
is attained at $x = 0$  if $ \norm{u} \leq \tau$ and at
$ x = (1 - \frac{\tau}{\norm{u}}) u $ if $\norm{u} > \tau$.
\end{lemma}
\begin{proof}
Let $f(x)  = \frac{1}{2} \norm{x - u}^2 + \tau \norm{x}  $, choose $v$ with $\scp{u}{v} = 0$, then $ x = \alpha u + \beta  v$.
Then $f$ is as a function of $\alpha$ and $\beta$ is
\begin{align}
f(\alpha u +  \beta v)
&= \frac{1}{2} \norm{(\alpha -1) u + \beta v}^2 + \tau \sqrt{ \norm{\alpha u + \beta v}^2} \\
&= \frac{1}{2}(\alpha -1)^2 \norm{u}^2 + \beta^2 \norm{v}^2 + \tau \sqrt{ \alpha^2 \norm{u}^2 + \beta^2 \norm{v}^2}
\end{align}
which decouples $\alpha$ and $\beta$ as the mixed terms disappear due to $u \perp v$.
Also note that $f(\alpha,\beta) \geq f(\alpha,0)$ thus it suffices to minimize $g(\alpha) = \frac{1}{2}(\alpha-1)^2 \norm{u}^2 + \tau\abs{\alpha} \norm{u}$.
Critical points of $g$ are $ \alpha = 0$ and the roots of its in $\alpha \neq 0$ defined derivative $g'(\alpha) = (\alpha -1) \norm{u}^2 + \frac{\alpha}{\abs{\alpha}} \tau \norm{u} = 0$, which leads to
$\alpha \norm{u} = \norm{u} - \frac{\alpha}{\abs{\alpha}} \tau$
the only solution of which is $\alpha^\ast = 1 - \frac{ \tau }{\norm{u}}$ if $\norm{u} > \tau$.
Comparing the values of $f$ at the critical points gives $f(\alpha^\ast u) = \tau \norm{u} - \frac{1}{2} \tau^2 < \frac{1}{2} \norm{u}^2 = f(0)$ thus the minimum is attained at $x = \alpha^\ast u$ if $\norm{u} > \tau$ and at $x = 0$ otherwise.
\end{proof}
Therefore the minimizer of~\eqref{eq_admm_z_upd} can be calculated explicitly which together with solving~\eqref{eq_admm_u_upd} up to accuracy of $\varepsilon^{(j)}_u$ the \myacro{ADMM} method will produce a sequence $u^{(j)}$ converging to the solution $u^\ast$ of~\eqref{opt_admm_abst}.

\chapter{Results} \label{sec_res}
\section{Regularized Reconstruction}
Solutions of the unconstrained problem 
\begin{equation*}  \label{opt_unconst_h_res}
  \min_{ u \in X} \quad \frac{1}{2} \norm[2]{ \Rh u - \unull}^2 + \lambda \norm[1]{ \nabh u} \tag{$P_h$}.
\end{equation*}
can be seen as a regularized reconstruction and as such as an alternative to the filtered backprojection or other reconstruction methods.
Increasing $\lambda$ will smooth the image while decreasing it will leave the effects of noise.  
\subsection{Synthetic data}
In order to test the described methods they are first applied to synthetic data. 
The Shepp-Logan Phantom, which is standard test image for image reconstruction methods, is compromised of ellipses of varying density, let it be the true image $P_\ast$.
A sinogram $\unull$ is produced by applying $\Rh$ to $P_\ast$ and adding five percent noise, that is random samples from a normal distribution with mean zero and standard deviation 5 percent that of the true sinogram $\Rh P_\ast$.
In Figure~\ref{fig_syn_unconst} these images can be seen together with the result of numerically solving the unconstrained problem with $\lambda = 0.3$, thus providing a regularized reconstruction of the phantom.
\begin{figure}[p]
    \centering
    \begin{subfigure}[b]{0.4\textwidth}
        \includegraphics[width=\textwidth]{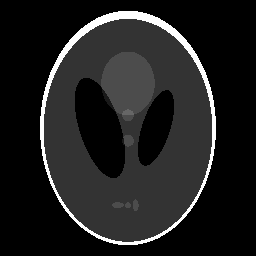}
        \caption{Phantom $P_\ast$} \label{syn_unconst_p}
    \end{subfigure}
    \begin{subfigure}[b]{0.4\textwidth}
        \includegraphics[height=\textwidth]{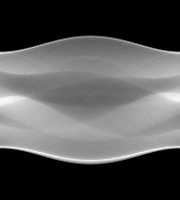}
        \caption{Sinogram $\Rh P_\ast$} \label{syn_unconst_sinot}
    \end{subfigure}
    \begin{subfigure}[b]{0.4\textwidth}
        \includegraphics[width=\textwidth]{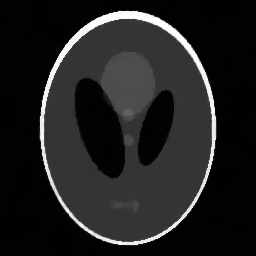}
        \caption{Result}\label{syn_unconst_res}
    \end{subfigure}
    \begin{subfigure}[b]{0.4\textwidth}
        \includegraphics[height=\textwidth]{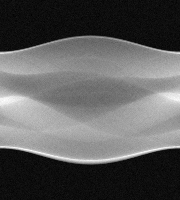}
        \caption{Noisy sinogram $\unull$} \label{syn_unconst_sino}
    \end{subfigure}
    \caption[Regularized reconstruction from synthetic data]{The discretized Radon transform is applied to the Shepp-Logan phantom \subref*{syn_unconst_p} producing a sinogram \subref*{syn_unconst_sinot}, after adding noise this provides the data $\unull$ in \subref*{syn_unconst_sino}. In \subref*{syn_unconst_res} the solution of the unconstrained problem with $\lambda=0.3$ can be seen.}
\label{fig_syn_unconst}
\end{figure}
The grid size $h$ was chosen to be $h=1$. 
For numerical stability it is desirable to have an operator $A$ with $\norm{A} <1$.
As a constant $\beta$ such that $ \| \Rh \| \leq \beta$ is known from Lemma~\ref{lem_Rh_bound} and since solving
\begin{equation}
\min_{u \in \mathbb{R}^{n \times n} } \, \frac{1}{2}  \|  \Rh u -  \unull  \|_2^2
    + \tilde\lambda \beta^2 \norm[1]{\nabh u}
\end{equation}
with $\lambda = \tilde \lambda \beta^2$ is equivalent to solving
\begin{equation}
\min_{u \in \mathbb{R}^{n \times n} } \, \frac{1}{2} \|  (\beta^{-1}\Rh) u -  \beta^{-1}\unull  \|_2^2
 +  \tilde \lambda \| \nabh u \|_1  \\
\end{equation}
proceed as follows:
Apply all algorithms to the operator $\beta^{-1} \Rh$ and its adjoint $\beta^{-1} \Rh^\ast$ with data $ \beta^{-1} \unull$ and parameter $\tilde \lambda = \beta^{-2} \lambda$.
Then the condition on the parameters of the Chambolle-Pock algorithm which read $\sigma \tau < ( \norm{A}^2 + \frac{8}{h^2})^{-1}$ reduces to $ \sigma \tau < \frac{1}{9}$

\subsection{Real data}
In~\cite{hamalainen2014total} the TV-regularized reconstruction model was used to reconstruct from projections from a limited number of angles, which was also applied to tomographic X-ray data of a walnut.
Thankfully this data was made publicly available at~\cite{walnut}.
The data consists of a sinogram with 328 beams and 120 angles and an operator in the form of a sparse matrix.
The results of applying to these data can be seen in Figure~\ref{fig_walnut} the provided sparse matrix was used rather then the discretized Radon transform.  

\begin{figure}[h]
    \centering
    \begin{subfigure}[b]{0.4\textwidth}
        \includegraphics[width=\textwidth]{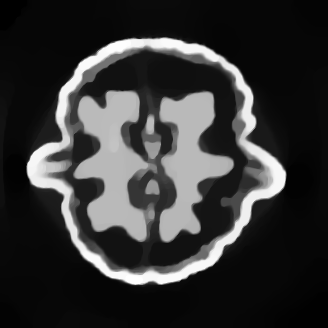}
        \caption{$\lambda=0.5$} \label{real_walnut0}
    \end{subfigure}
    \begin{subfigure}[b]{0.4\textwidth}
    \includegraphics[width=\textwidth]{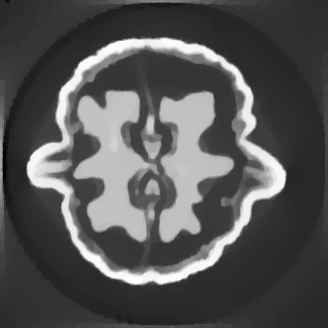}
    \caption{$\lambda = 0.05$} \label{real_walnut1}
    \end{subfigure}
    \caption[Regularized reconstruction from real data: Walnut]{Regularized reconstruction from real data: Walnut}
\label{fig_walnut}
\end{figure}

Raw data from a CT scan experiment of a box with test tubes of varying density was used.
The data format is described in Appendix~\ref{app_file}, it is called dataset02.
The extracted data was preprocessed using techniques from Appendix~\ref{app_data_ext} to produce sinogram data $\unull$ fitting the model of the radon transform a parallel beam used in this thesis.
The data associated with one z-position by the function \code{get\_z\_slice} are projections from $1152$ angles each with 736 entries corresponding to the physical detectors of the machine.
In the preprocessing phase this was downsampled to $360$ angles and $400$ beams for quicker processing producing a  $256 \times 256$ pixel image of the cross section.
The results can be seen in Figure~\ref{figs_rr_pdrq}.
\begin{figure}[p]
    \centering
    \begin{subfigure}[b]{0.4\textwidth}
        \includegraphics[width=\textwidth]{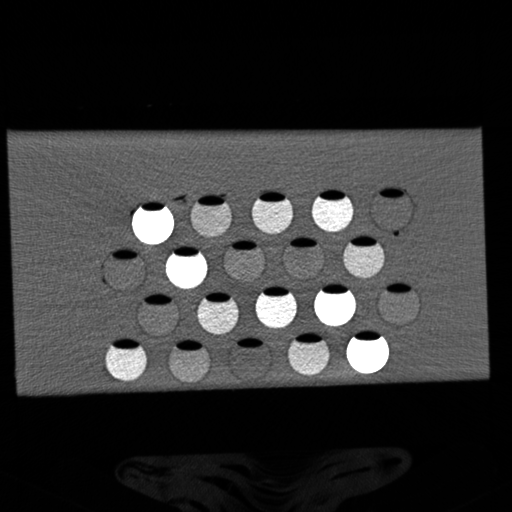}
        \caption{Original} \label{figs_rr_prov_reconst}
    \end{subfigure}
    \begin{subfigure}[b]{0.4\textwidth}
        \includegraphics[width=\textwidth]{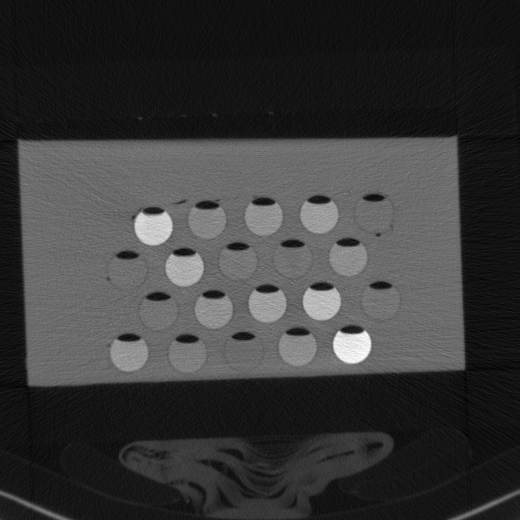}
        \caption{FBP full} \label{figs_rr_fpb_full}
    \end{subfigure}
    \begin{subfigure}[b]{0.4\textwidth}
        \includegraphics[width=\textwidth]{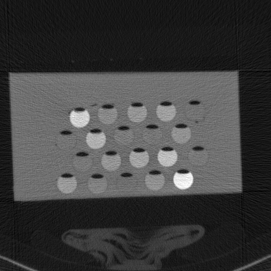}
        \caption{FBP undersampled}\label{figs_rr_fbp_under}
    \end{subfigure}
    \begin{subfigure}[b]{0.4\textwidth}
        \includegraphics[width=\textwidth]{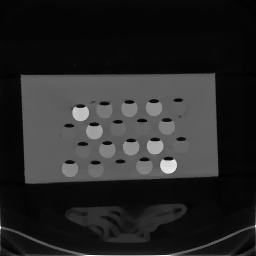}
        \caption{PDRQ} \label{figs_rr_pdrq}
    \end{subfigure}
    \caption[Regularized reconstruction from real data: Box]{\subref{figs_rr_prov_reconst} The reconstruction provided with the dataset.  \subref{figs_rr_fpb_full} Filtered Backprojection of all data obtained by the function \code{get\_z\_slice}. \subref{figs_rr_fbp_under} Filtered backprojection using undersampled data.  \subref{figs_rr_pdrq} Result of regularized reconstruction obtained with PDRQ with $\lambda = 0.5$ after 100 iteration using same undersampled data. }
\label{fig_results_real}
\end{figure}



\begin{figure}[h]
    \centering
    \includegraphics[width=.5\textwidth]{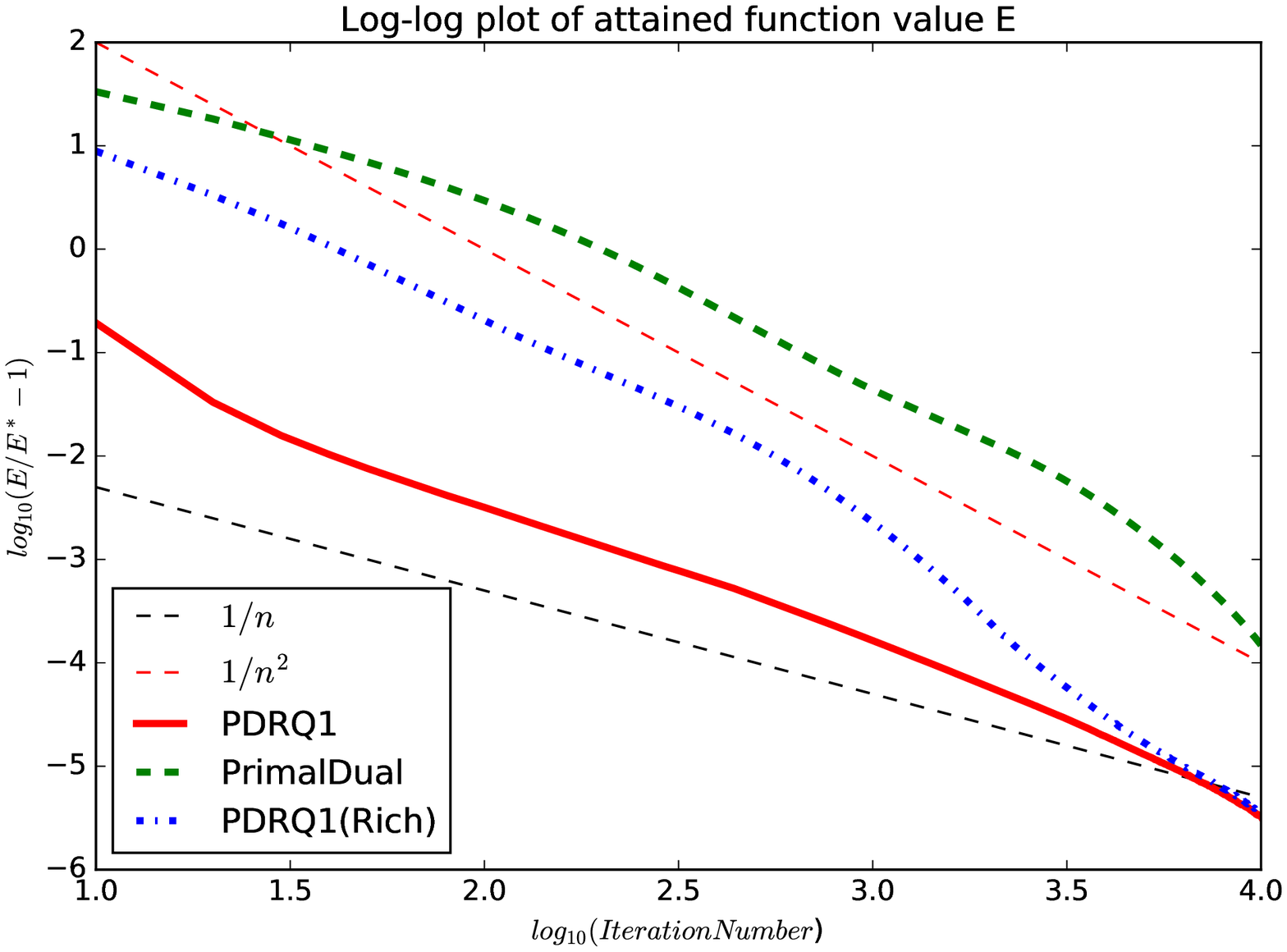}
    \caption[Log-log plot, PrimalDual vs Preconditioned]{Effects of preconditioning. Log-log plot of attained function value relative to smallest function value attained for phantom data with $\lambda =0.3$ shows convergence rate of the Primal-Dual Champolle-Pock algorithm and PDRQ using the norm preconditioner and simple Richardson preconditioning.}
    \label{fig_PDRQvsPD}
\end{figure}

\begin{figure}[h]
    \centering
    \includegraphics[width=.5\textwidth]{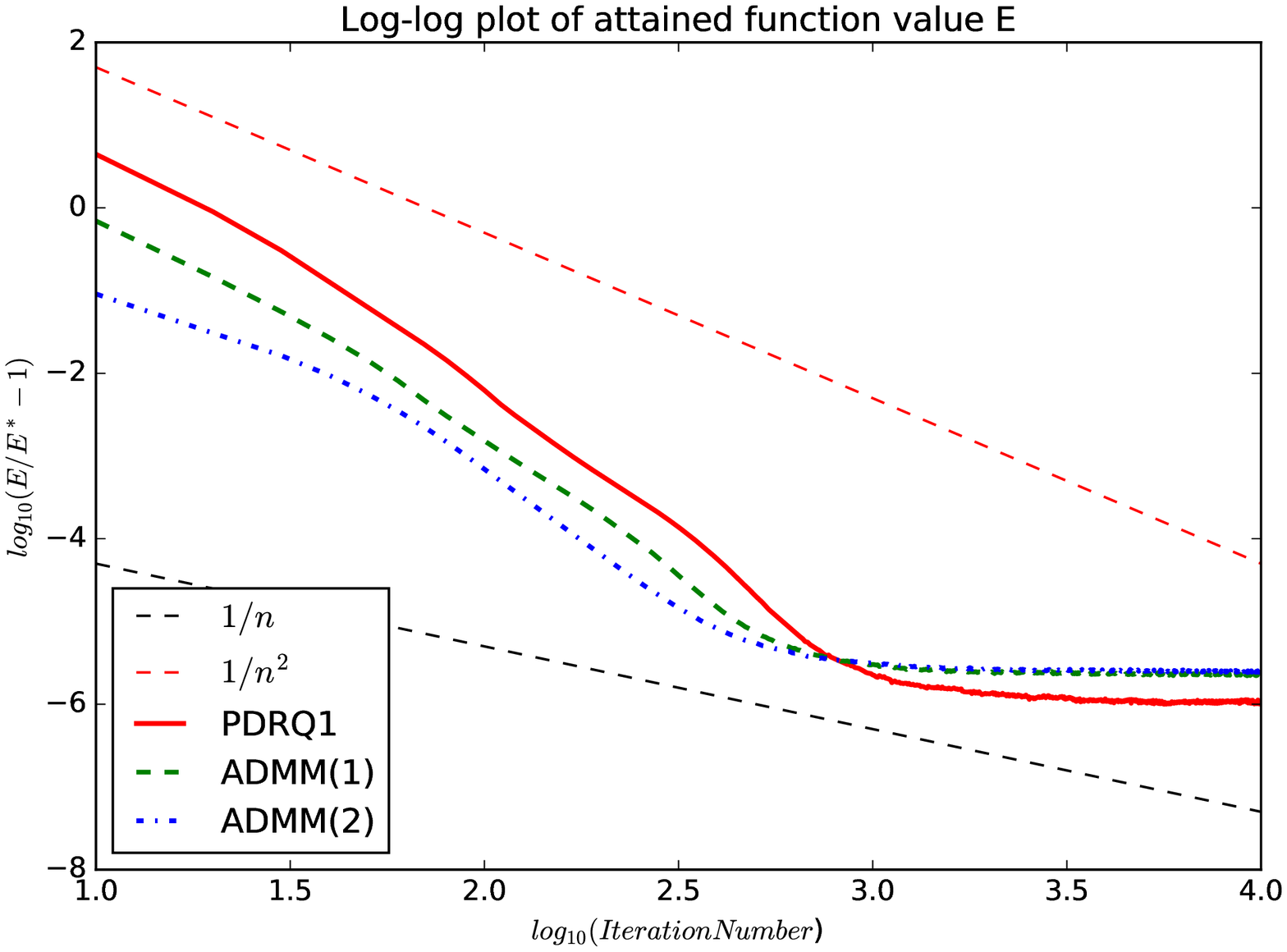}
    \caption[Log-log plot, ADMM vs PDRQ]{Better convergene on a different data set, using $\lambda = 0.1$ and the Radon transform from the ASTRA package. PDRQ is competitive with \myacro{ADMM}, note that each inner \myacro{PCG}-iteration requires one more application of the operator. Log-log plot of attained function value relative to smallest function value Algorithm~\ref{alg_PDRQS_const} for synthetic data without constraints. \myacro{ADMM} with one and two \myacro{PCG} iteration vs PDRQ. }
    \label{fig_admm_vs_pdrq}
\end{figure}

\FloatBarrier
\section{Metal artifacts}
\subsection{Synthetic data}
In this section artifacts are generated by damaging synthetic sinogram data by capping it, that is replacing values greater than a certain threshold $c$ by this value $c$.
This capping is the assumed origin of the metal artefacts which was used in Section~\ref{sec_ProbMod} to formulate the inequality constraints. 
Trying to reconstruct from such capped data using filtered backprojection does indeed produce artefacts similar to those produced by real metal inclusion.
For the experiment the Shepp-Logan phantom that was used in the previous section is now furnished with a metal inclusion, a block of value 3, as opposed to the usual values ranging from 0 to 1 representing organic tissue, is inserted in the phantom representing the much denser metal.
This phantom reconstruction was already done by the author using the Campolle-Pock algorithm in~\cite{Schiffer}, showing the feasibility by very slow convergence. 
This is significantly sped up requiring only a few hundred iteration to visibly reduce artifacts, using any of the preconditioned algorithms presented in this thesis, Figure~\ref{fig_syn_unconst_metal} shows the result using \myacro{PDRQ1} and with an inverse norm preconditioner.
\begin{figure}[p]
    \floatpagestyle{empty}
    \centering
    \begin{subfigure}[b]{0.4\textwidth}
        \includegraphics[height=\textwidth]{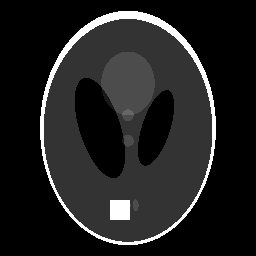} 
        \caption{Phantom $P_\ast$} \label{syn_metal_P}
    \end{subfigure}
    \begin{subfigure}[b]{0.4\textwidth}
        \includegraphics[height=\textwidth]{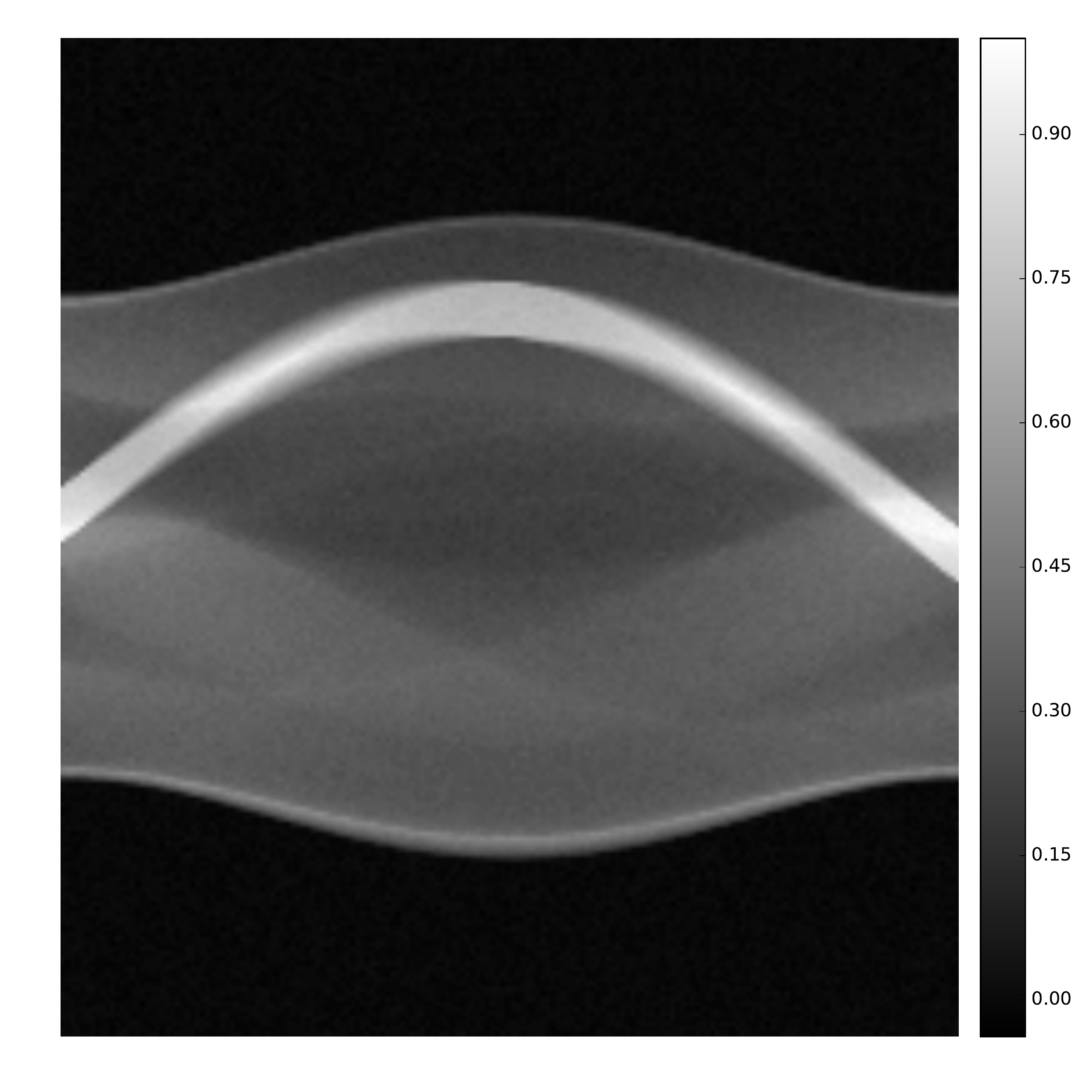}
        \caption{Sinogram $\Rh P_\ast$ with noise} \label{syn_metal_sino}
    \end{subfigure}
    \begin{subfigure}[b]{0.4\textwidth}
        \includegraphics[height=\textwidth]{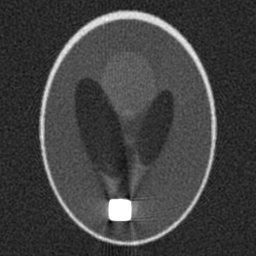}
        \caption{Fbp reconstruction} \label{syn_metal_fbp}
    \end{subfigure}
    \begin{subfigure}[b]{0.4\textwidth}
        \includegraphics[height=\textwidth]{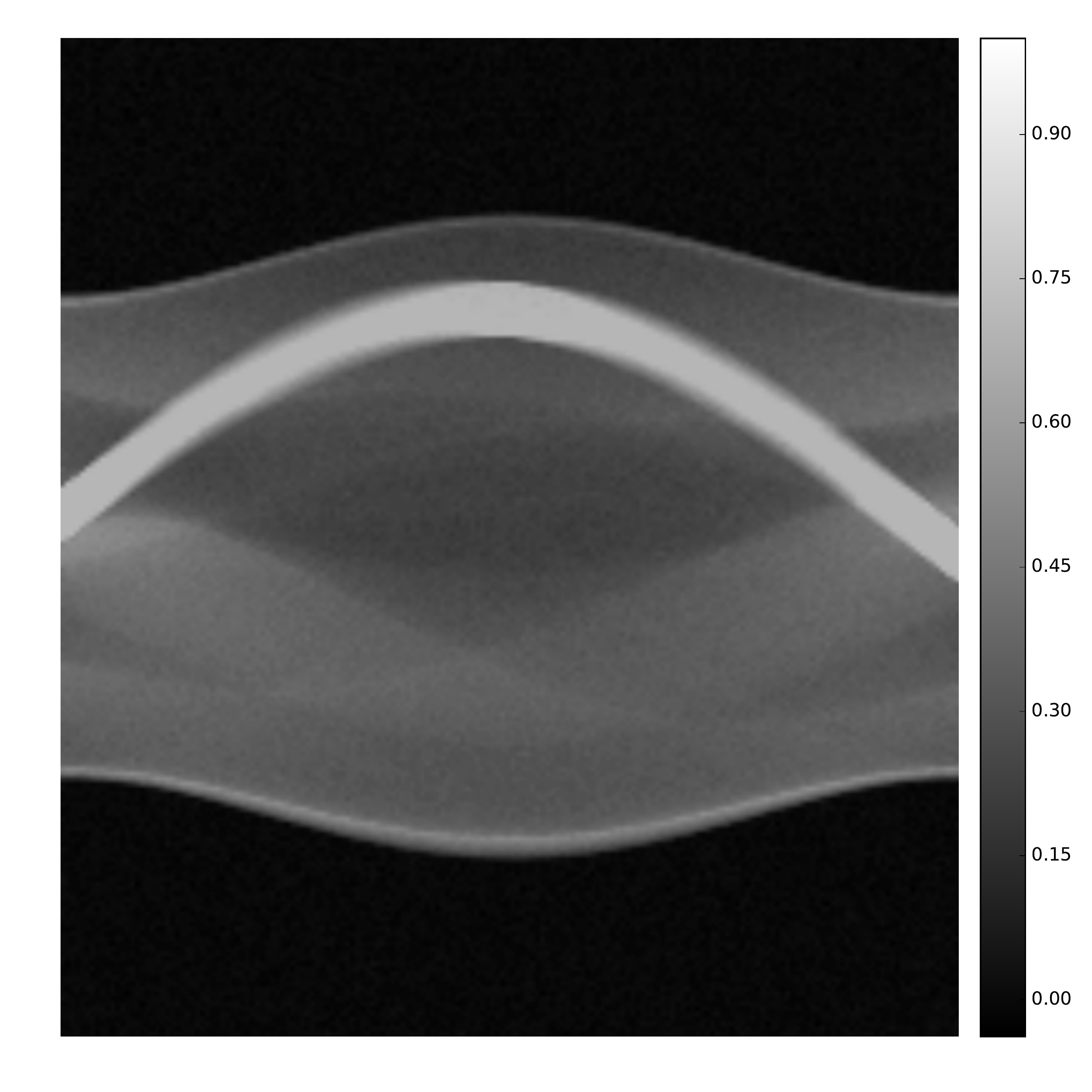}
        \caption{Noisy capped sinogram $\unull$} \label{syn_metal_sinocap}
    \end{subfigure}
    \begin{subfigure}[b]{0.4\textwidth}
        \includegraphics[height=\textwidth]{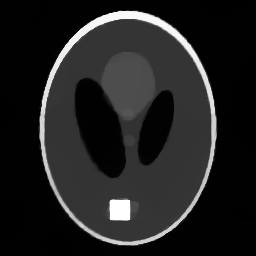}
        \caption{Result} \label{syn_metal_result}
    \end{subfigure}
    \caption[Metal artefact removal from synthetic data]{The discretized Radon transform is applied to the Shepp-Logan phantom \subref*{syn_unconst_p} producing a sinogram \subref*{syn_unconst_sinot}, after adding noise this provides the data $\unull$ in \subref*{syn_unconst_sino}. In \subref*{syn_metal_result} the solution of the unconstrained problem with $\lambda=0.3$ can be seen.} 
\label{fig_syn_unconst_metal}
\end{figure}
\FloatBarrier

\subsection{Real metal artifacts}
CT data from dataset02 also contains images of small metal parts, specifically in Figure~\ref{fig_metal_objects} a screw and a hex key can be seen.
Trying to reconstruct results in typical metal artifacts see Figure~\ref{fig_results_real_metal}.

The constraints on the sinogram $\unull$ in areas deemed to be affected by metal $\Dm_0$  where set depending on the value of the sinogram in that location, \eg $C = 0.8 \unull$ was chosen.
Determining which areas of the sinogram are affected by metal one way is to take all areas of the sinogram which exceed a certain level
Alternatively, what showed most success, is to perform a regular reconstruction then mark areas of the image domain where the estimated radio density exceeds a certain threshold, and then apply back projection to the marked areas to obtain areas in the sinogram which correspond to the presumed metal. 
This is done in the hopes that the process of estimating where metal is located is less sensitive to artefacts than the image reconstruction.
The solution of the so constrained problem can be seen in Figure~\ref{fig_results_real_metal}.
The reconstruction appears to be particularly impeded by parts of the object to be scanned outside of the field of view, this can be alleviated by reconstructing on a larger region, see Figure~\ref{fig_test_metal_ext_fov} for example.
\begin{figure}[h]
    \centering
    \includegraphics[width=\textwidth]{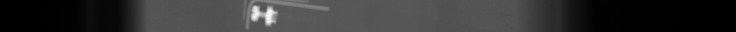}
    \caption[]{Metal objects visible in raw data}
    \label{fig_metal_objects}
\end{figure}

\begin{figure}[h]
    \centering
    \begin{subfigure}[b]{0.4\textwidth}
        \includegraphics[width=\textwidth]{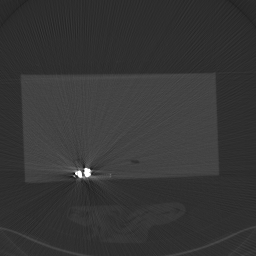}
    \end{subfigure}
    \begin{subfigure}[b]{0.4\textwidth}
        \includegraphics[width=\textwidth]{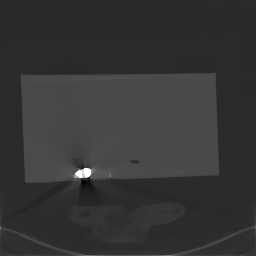}
    \end{subfigure}
    \caption[Metal artefacts in real data]{Left: Metal artefacts that occur when trying to reconstruct with the filtered backprojection. Right: Proposed techniques reduce artefacts somewhat.}
\label{fig_results_real_metal}
\end{figure}

\begin{figure}[h]
        \includegraphics[width=\textwidth]{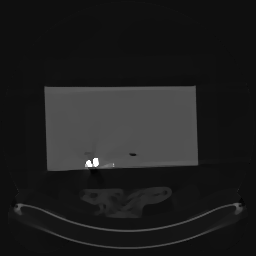}
                \caption[Metal artefact reduction real data]{Better artifact reduction: Extending the field of view and creating a mask from back-projecting a capped(and dilated by 5px)
        unconstrained reconstruction, with $C = 0.8 \unull$ }
        \label{fig_test_metal_ext_fov}
\end{figure}

\section{Conclusion}
The model for artefact reduction by introducing inequality constraints on the sinogram was presented and significantly sped up by preconditioning.
Novel preconditioners were presented as well as existing ones. 
The ability of the proposed technique to reduce metal artefacts in real data was demonstrated in a prototype setting.
Quality of the reconstructions surely could be improved by incorporating more technical knowledge of the raw data and its acquisition into the data processing.
Application to experimental data dedicated to testing for metal artifact reduction such as a real data of a phantom once with and once without metal inclusions under known calibration and with no out of view data could better determine the merit of the proposed method.  
Obvious extensions of the model would be to use more general regularization such as the total generalized variation.
Both sources of a dual energy CT scanner could be used by comparing against data from both in the discrepancy term, although this would not make use of their different energy spectra.
Further modeling of the cause of artefacts might need to be integrated in the model. 
It can be hoped that the results of this thesis can be improved upon to form viable technique for reducing metal artefacts or be integrated with existing models.

\appendix                       
\addpart*{Appendix}             
\chapter{Data extraction} \label{app_data_ext}
This section is dedicated to providing the data $u_0$ as is needed for the described algorithms using the described discretization of the Radon transform  given raw data from a modern CT machine which uses different imaging modalities then the idealized Radon transform.
While the Radon transform as described in Section~\ref{sec_radon_trhm} models a radiation source that emits parallel beams it is much more practical to design a machine with an (approximate) point source.
In this case the beams spread out like a fan.
A method the convert these fan beam data to parallel data is presented in \ref{ssec_fan2para}.
Additionally modern CT-scanners will not simply operate in one plane but rather in space such that each exposure produces data in multiple planes.
A point source and opposing detector -- either flat or a spherical segment -- rotate around the object to be scanned in a plane while moving along it, in what is called the z-direction, thus the source and detector follow a spiral trajectory.
A simple method of how real data from such a device can be preprocessed to be reconstructed into cross sections  with 2D Radon transform is presented in \ref{ssec_simple_data_extr}.
The advanced single slice rebinning(\myacro{ASSRB}), a  more sophisticated method to process the raw data, is presented in \ref{ssec_assrb}.

\section{Fan beam to parallel Beam(Fan2Para)} \label{ssec_fan2para}
In Section~\ref{sec_radon_trhm} the Radon transform was presented as a model for tomography where with each exposure parallel beams are emitted in practice however a fan beam geometry is used where a point source emits beams.
In order to convert data obtained by fan beam geometry to parallel beam a re-sampling operation is performed:
Data at desired coordinates in parallel beam geometry is obtained by converting to equivalent fan beam coordinates and interpolating using the given data in fan beam geometry.
This process is now described in detail, refer to Figure~\ref{fig_beam_geom} for an overview.
Considering the fan beam configuration of a CT machine where a point source $A$ rotates around the origin of which it has a fixed distance $d$.
The position of the source $A$ is measured by the angle $\varphi$ between the line $L$ from  the origin to $A$ and the horizontal x-Axis.
Then each beam emitted from $A$ will be described by the angle $\alpha$ between it and the line $L$.
As an example detector elements linearly spaced on a round detector, in the form of an arc of a circle, would correspond to beams whose angles $\alpha$ are linearly spaced, this is the case in most CT machines.
As in the parallel geometry where the offset $s$ was positive to the left of the center beam, \ie in the direction of $x_0$, and negative to the right here too $\alpha$ has a sign \eg in Figure~\ref{fig_beam_geom} the angle $\alpha$ is counted as positive as it is to the left of the center line $L$.
For a given beam $B$ in fan beam geometry with coordinates $(\alpha, \varphi)$ find the coordinates $(s, \theta)$ that this beam has in a parallel beam configuration.
From Figure \ref{fig_beam_geom} one can see that $\theta = \varphi - (\pi/2 - \alpha)$ as well as $s = d \sin(\alpha)$.
Thus the function that maps a beam from a fan beam configuration to a corresponding beam in a parallel configuration is given by $f \colon (\alpha, \varphi) \mapsto (s, \theta) = (d \sin(\alpha), \varphi - \pi/2 + \alpha)$, this coordinate transform  has the inverse given by $f^{-1} \colon (s, \theta ) \mapsto (\alpha, \varphi)=(\arcsin(\frac{s}{d}), \theta +\pi/2 - \arcsin(\frac{s}{d}))$.
To get an approximate value at a query point $P$ with desired coordinates $(s, \theta)$ in parallel beam space, the coordinates of $P$  are transformed to fan beam coordinates  $f^{-1}(s, \theta)$ then using \eg bilinear interpolation at this coordinates with the values of the collected data at positions $(\alpha, \varphi)$ then gives an approximate value for data at $P$.
This method also known as rebinning is used in the MATLAB function \code{fan2para}, but has also been implemented for the project of this thesis to allow for more flexibility.
\begin{figure}[ht]
        \centering
        \includegraphics[width=\textwidth,trim=0cm 8cm 0cm 9cm, clip]{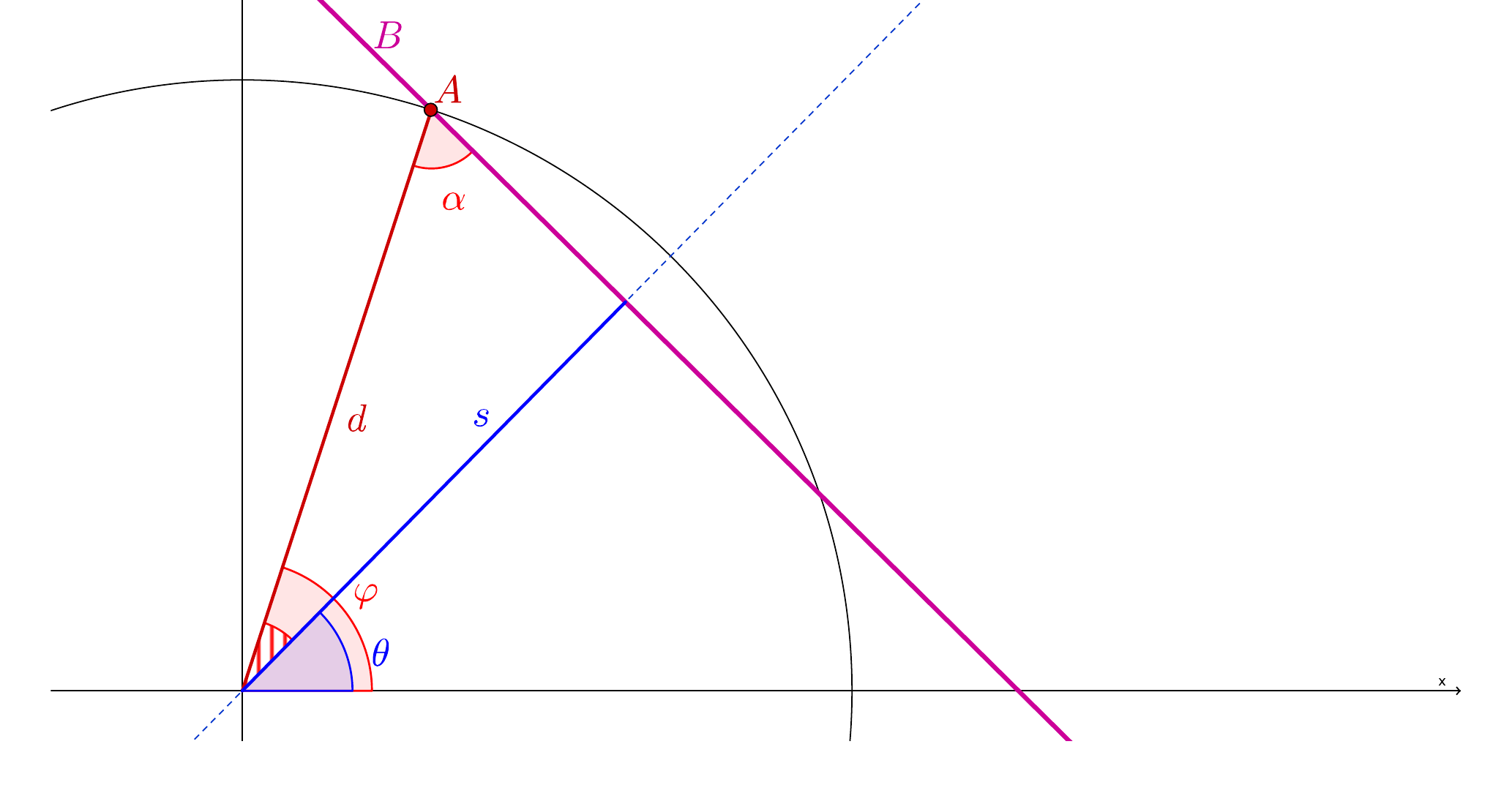}
                \caption[Fanbeam and paralell beam configuration ]{The Beam $B$ has the fan beam coordinates $(\alpha, \varphi)$ and parallel beam coordinates $(s,\theta)$. }
        \label{fig_beam_geom}
\end{figure}

\section{Simple data extraction} \label{ssec_simple_data_extr}
Modern CT machines do not simply work in one plane but collect data for multiple planes simultaneously this is achieved by a cone beam configuration consisting of a point source emitting rays in a cone and an opposing detector whose sensors are arranged in a 2-dimensional array.
One popular configuration, that is used in the data for this project, is that of a cylindrical detector, in which each row of the detector consists of equiangular arranged detectors and the rows are a fixed distance apart.
Thus from the source each row gives a fan beam their distance is usually not measured at the detector but at the center of rotation where this equal distance is called the \emph{slice width} $w$ as every fan beam represents a slice of the object to be scanned.
The data from every exposure is an array of values, each row having the data from one fan beam as described above, let this array be called a \emph{frame}, as interpreting each frame as an image and playing them gives a video as if the source where a camera rotating around the object to be scanned in a spiral path.

A very simple way of extracting the data associated with a z-position is to -- incorrectly -- assume that the data for each frame consists of slices corresponding to \emph{parallel} fan beams that are slice width apart and orthogonal to the z-axis.
This means it is assumed that there is a stack of parallel fan beams rather then the fan beams coming from a common source.
If the detector is in position $z$ then row number $i$ is in z-position $ z + i \cdot w$, to get data associated with a $z$-position $z_0$ all rows whose z-position is within $z_0 \pm 0.5 \cdot w$ are collected, thus giving data associated with z-position $z_0$ as if collected by a single row fan beam detector.
Then the data collected in this fashion can be converted to parallel beam data using the rebinning algorithm described above.
Note that this assumption is not true as the data is from a cone beam.
This very simple approach was also used in the data extraction in~\cite{Boas_2iter} as is explained in the appendix of the paper.

\section{Advanced data extraction: \myacro{ASSRB}} \label{ssec_assrb}
Another method to obtain a parallel beam sinogram from the raw cone beam data is presented in advanced single slice rebinning~\cite{Assrb}, which is based on~\cite{Ssrb}.
This method produces parallel beam data for a reconstruction plane that is tilted with respect to the z-axis to approximate the spiral path the gantry rotates around.
After reconstruction in several tilted planes the images can be interpolated to yield the usual image orthogonal to the z-axis.
\begin{figure}
        \centering
        \includegraphics[width=.25\textwidth]{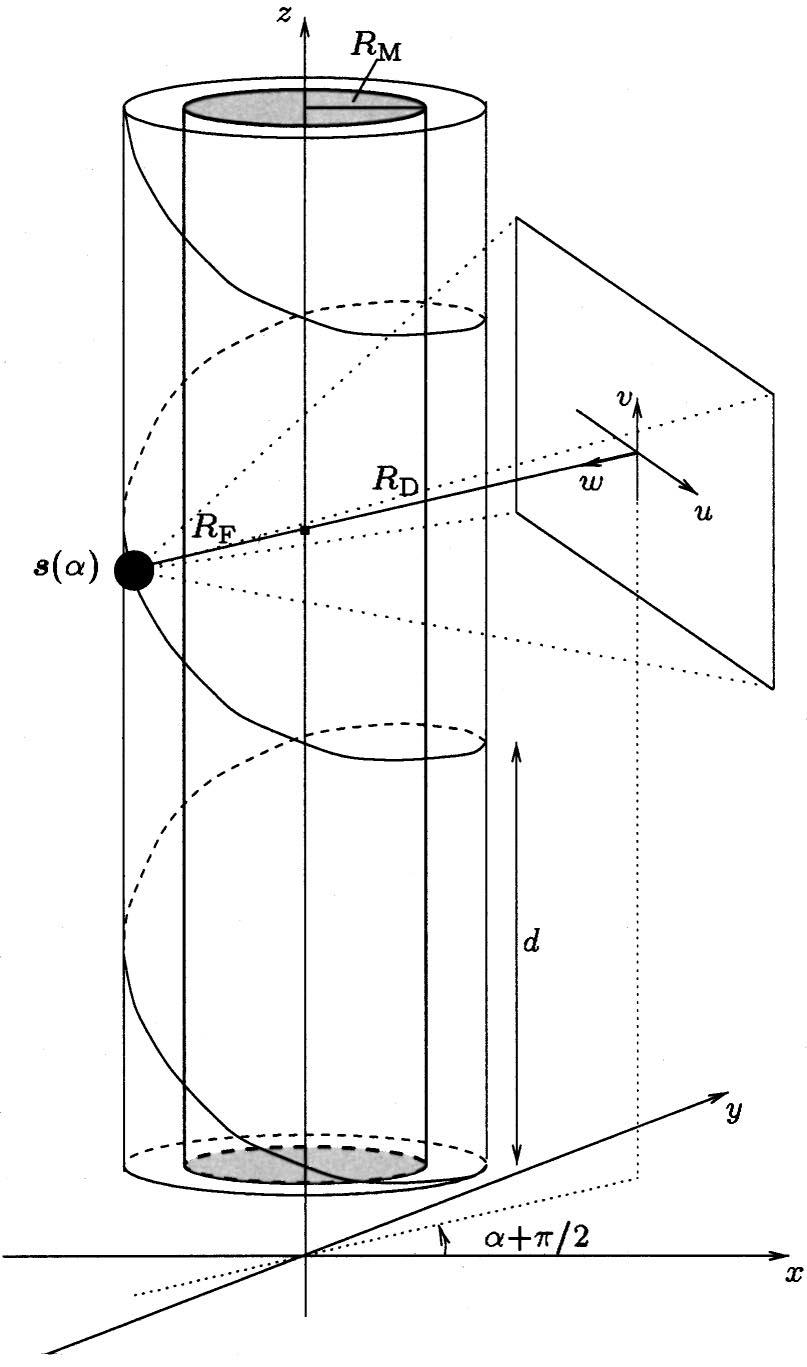}
        \caption{Overview of Geometry used in \myacro{ASSRB}. Graphic from~\cite{Assrb}}
        	\label{fig_assrb_overview}
\end{figure}
It is assumed that the path of the source is given by the spiral trajectory $ s(\alpha) = (d \sin \alpha, d \cos \alpha, \Delta z  \alpha /2 \pi)$ parametrized by the angle $\alpha$, where $d$ is the distance from the source to the center of rotation, called $R_F$ in the original paper.
Opposing the source is a flat detector with a distance to the source of $R_{FD} = d + R_D$, see Figure~\ref{fig_assrb_overview} this will be adjusted to the cylindrical detector geometry later on.
Since the raw data in the example have angles increasing by a fixed increment but the z-positions do not, the angles are fitted to the z-position by linear regression.
This gives $\Delta z = \SI{19.15}{\milli\meter}$ which is the increase in z-position per one rotation.
This also allows us to consider the z-position $z_0$ as input and find the associated angle $\alpha_0$.
The angle $\gamma$, at which the reconstruction plane $R$ is tilted around the line from the source to the center of the detector, is chosen to minimize the area between $R$ and the spiral path when projected onto the detector. It is given by
\begin{equation*}
  \tan \gamma = \Delta z \frac{\alpha^\ast}{2 \pi R_F \sin \alpha^\ast}
  \quad \text{with } \quad
  \alpha^\ast = \arccos \ndel{\Delta z \, \frac{ 1 + \cos (f \pi)}{2}}.
\end{equation*}
For coordinates $\theta \in [-f\pi, f\pi]$ and offset $s \in [-R_M, R_M ]$ in the desired parallel geometry the angle of the cone beam data to be used is unsurprisingly the same as in fan2para case
\begin{equation*}
  \alpha'(\theta, s) = \theta + \arcsin \ndel{\frac{s}{d}}.
\end{equation*}
This gives the coordinates of the desired beam assuming a flat detector as
\begin{equation*}
  u_F(\theta, s, \alpha') = \frac{R_{FD}}{d} \frac{s}{\cos(\alpha' - \theta)}
\end{equation*}
and
\begin{equation*}
  v_F(\theta, s, \alpha') = \frac{R_{FD}}{d}
  \ndel{\frac{s \cos \alpha' \tan \gamma}{\cos(\alpha' - \theta)} - \Delta z \frac{\alpha'}{2 \pi}}.
\end{equation*}
The extracted data is given on a cylindrical detector with coordinates $(\beta, S)$ where $\beta$ is the angle between beams in the $u$-direction and $S$ the distance between fans at the rotation center or nominal slice thickness.
Thus the obtained coordinates for the flat detector have to be transformed to the cylindrical coordinates by
\begin{equation*}
  \beta = - \arctan \ndel{ \frac{u_F}{R_{FD}} } \quad \text{and } \quad  b =  v_F \frac{d}{R_{FD}}.
\end{equation*}
Note that this is inconsistent with the formula for $b$ given in the appendix A in~\cite{Assrb}.
The sinogram is obtained by -- \eg bilinear -- interpolation at the points
\begin{equation}
  \big( \alpha_0 + \alpha'(\theta, s), \, \beta(\theta, s), \, b(\theta, s) \big)
\end{equation}
Finally a correction factor of
\begin{equation*}
  p(\theta, s) = \frac{\cos \gamma}{\sqrt{\sin^2 \theta + \cos^2 \gamma \cos^2 \theta}}
\end{equation*}
needs to be applied to the resulting sinogram.
This method for data preprocessing was implemented and tested but showed no significantly improved image quality.
The \myacro{AMPR}-approach is an extension and generalization of advanced single-slice rebinning.~\cite{ohnesorge2006multi}

\chapter{Data structure} \label{app_file}
Given was raw data from a dual source CT scanner, likely a Siemens Definition Flash, which needed to be reverse engineered due to lack of documentation, it is called dataset02.
For the reconstructions only one of the two sources was used.
Files with the file extension \code{*.ptr} are container files, all entries are big endian.
After the 8 byte magic number \code{'XABC2000'} there is a integer of type \code{uint64} describing the number of entries in the container, in dataset02 there where 44 entries.
This is followed by a description of each entry: first 16 bytes are unknown followed by a \code{uint64} offset indicating the start of the entry and a \code{uint64} giving its length in bytes.
The entries are of different filetypes, the first entry is a text file containing a hardware component list, the second entry contains parameters -- as well as the model type number P47A. 
The third entry contains parameters in XML-format, some interesting values are summarized in Table~\ref{tab_parameters}, for example ChannelsPerLineA denotes the number of detector elements of detector A and can be thought of as the number of beams penetrating the object which are measured.
\begin{table}[ht]
\parbox{.45\linewidth}{
    \centering
    \begin{tabular}{ll}
        Parameter 	        & value \\ \hline
        ChannelsPerLineA  	& 736			\\
        ChannelsPerLineB  	& 480			\\
        FanBeamGrid 		& 0.067864		\\
        RadiusFocusPath		& 595.000000		\\
        MeasurementFieldA 	& 500.000000		\\
        MeasurementFieldB 	& 332.000000 	\\
        FlyingFocalSpot		&SciFFSZ
    \end{tabular}
    \caption{Some parameters} \label{tab_parameters}
}
\hfill
\parbox{.45\linewidth}{
    \centering
    \begin{tabular}{ll}
    length(bytes)  & block content\\ \hline
    640     &metadata           \\
    512     &line information	\\
    32      &image header A     \\
    47104   &data A             \\
    3968    &?                  \\
    32      &image header B     \\
    30720   &data B             \\
    1920    &?               	\\
    \end{tabular}
    \caption{Data structure of entry six for one frame}  \label{tab_blocks_in_part5}
}
\end{table}
The bulk of the data is contained in part six which contains for each of the 19084 frames one image from each of the two detectors accompanied by meta data described in Table~\ref{tab_blocks_in_part5}.
The image data for each frame is given in dimensions \code{n\_slices} $\times$ \code{n\_beams} of data type \code{uint16}, for example one image of data A has dimension $32 \times 736$ which has a length of 47104 bytes.  
The metadata block for each frame consists of $640$ bytes, some of the identified values important for processing the data are given in Table~\ref{tab_meta_data}.
The first few hundred frames are invalid and were likely used for calibration, for example the entries for the angle values of system B are negative for the first 304 frames after which they start to increase by fixed amount and are consistently in the range of angle in degrees.

\begin{table}[h]
    \centering
    \begin{tabular}{lllll}
    parameter   & offset 	& data type      & System  & range \\ \hline
    IdFrame     & 12	        & \code{uint16}  &    &                 \\
    NoBeams     & 32	        & \code{uint16}  & A  &                 \\
    angles      & 36	        & \code{float32} & A  & $0$-$359.6875$  \\
    z-position  & 40	        & \code{int32} 	 &    &$-907040$ - $-1224360$ \\
    adjust factor 	& 48    & \code{float32} & A  &                 \\
    adjust offset 	& 56    & \code{float32} & A  &                 \\
    NoBeams	        & 64    & \code{uint16}  & B  &                 \\
    angles	        & 68    & \code{float32} & B  & $0$-$359.69373$ \\
    adjust factor 	& 80    & \code{float32} & B  &                 \\
    adjust offset 	& 88    & \code{float32} & B  &                 \\
    angles 	        & 100   & \code{uint16}  & A  &                 \\
    z-position/10 	& 184   & \code{int32} 	 & A  &                 \\
    \end{tabular}
    \caption{Meta data positions and data types} \label{tab_meta_data}
\end{table}

\section{Intensity correction factor}
Each frame uses the full range of the data type \code{uint16} in which it is stored, suggesting that these values have been scaled as the dynamic range will vary with z-position.
Considering that the sum over all values of a frame $i$ -- denoted by $S(i)$ -- should represent the integral over the density of the objects in view and this is independent of the angle under which the projection is made, the sums over the frames should vary very little and do so continuously, representing only new material coming into view and leaving the view in z-direction.
A presumed scaling correction factor -- \eg corresponding to exposure time -- $c(i)$ should be such that $S(i) c(i)$ is approximately constant. 
The meta data values stored at offset $48$ are a good candidate but do not quite fit.
Noting the stripe in the uncorrected sinogram \ref{fig_scal_sino_unsc} from column 240 to 340 as well as some smaller column errors correspond to the changes in the values of the metadata \code{adjust\_offset} suggest a second correction parameter.
Thus the correction is suspected to be of the form $\code{data}*\code{adjust\_factor}*2^{-16} + \code{adjust\_offset}$, the result of this correction can be seen in \ref{fig_scal_sino_corect}.
Note remaining artifacts at the edge of the stripe and some overcorrection at the other smaller stripes.
Additionally the units of data remain unknown.

%

\begin{figure}[p]
        \centering
  		\includegraphics[width=.8\textwidth]{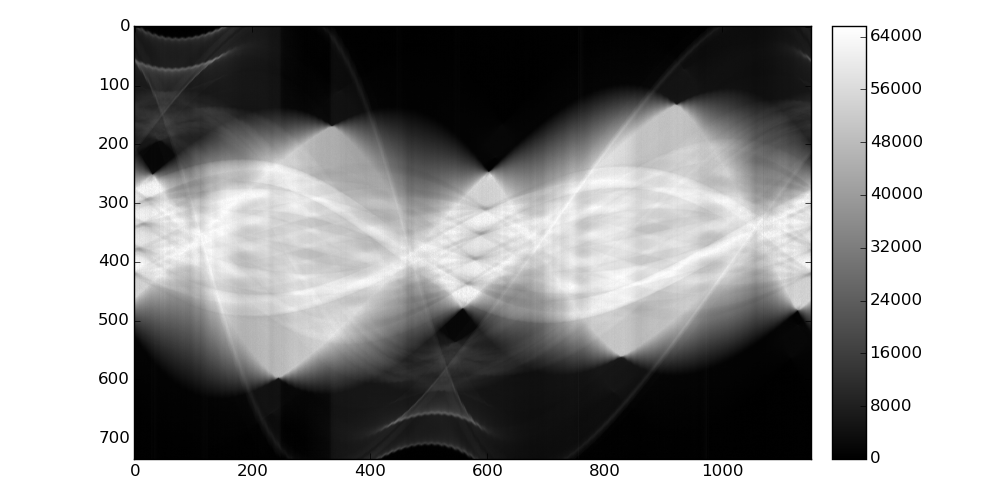}
  		\caption{Uncorrected sinogram\label{fig_scal_sino_unsc}}
  		\includegraphics[width=.8\textwidth]{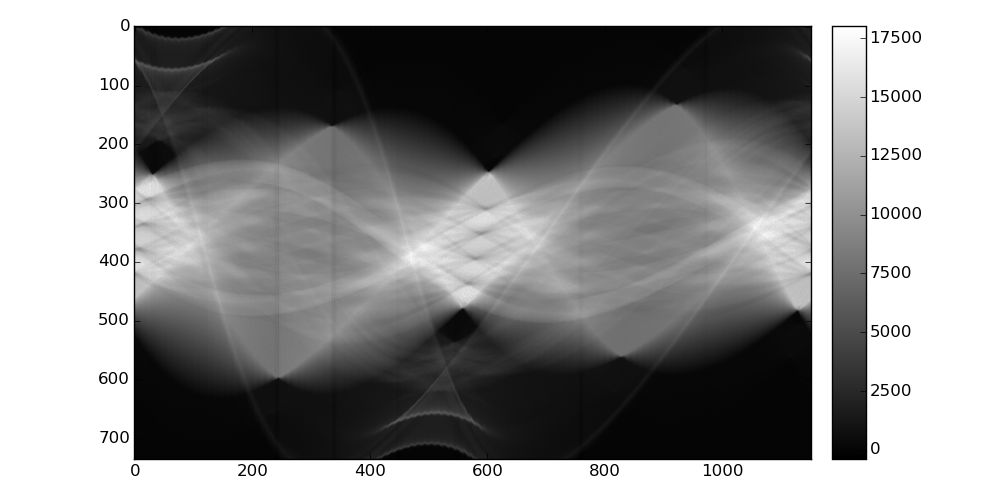}
  		\caption{Sinogram with correction\label{fig_scal_sino_corect}}

\end{figure}

\section{Focal spot correction: z-flying focal spot}
Oscillating behavior of the frames indicates that a technology called `z-sharp'
or `z-flying focal point' ~\cite{kachelriess2006flying} was used.
This is also by indicated by the parameter FlyingFocalSpot in Table~\ref{tab_parameters}.
This technique consist of changing the position of the X-ray source slightly to produce two projections with 32 rows each.
This is done in such a way that when the two projections are interleaved, that is alternating the rows of projection $i$ and $i+1$, the result is as if a single 64 row detector were used.
\begin{figure}
    \centering
    \includegraphics[width=.8\textwidth]{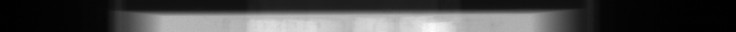}
    \includegraphics[width=.8\textwidth]{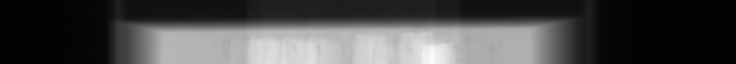}
    \caption{Frame 47000 and interleaved with the next frame \label{fig_single_frame_p_inter}}
\end{figure}

Details from~\cite{kachelriess2006flying} and~\cite{kyriakou2006impact} indicate
that two adjacent frames can indeed be interleaved while still belonging to different angles.
Additionally one has to correct for a change in radial distance
\begin{equation*}
\partial R_F = \frac{1}{4} S \frac{R_F + R_D }{R_D \tan{\phi}}
\end{equation*}
where the anode angle $\phi$ is assumed to be $\ang{7}$,
see figure~\ref{fig_focal_spot_deflection_kachelriess}.
This slight change of $ \Delta R_F = 2 \partial R_F = \SI{5.4}{\milli\metre}$
of the distance from the source to the detector results in visible jumps when interleaving
two adjacent frames
see figure~\ref{fig_frame_with_radial_correction}.
Approximating the true ray $a$ with angle $\beta'$ by a virtual ray
$b$ with angle $\beta$ -- see figure~\ref{fig_radial_correction} -- results in the correction
\begin{equation*}
\beta = \frac{R_{FD} + \Delta R_F}{ R_{FD} } \beta'.
\end{equation*}
This is done by assuming the extracted values are at positions $\beta'$ and interpolating
at corrected positions $\beta$.
\begin{figure}
    \centering
    \includegraphics[width=.5\textwidth]{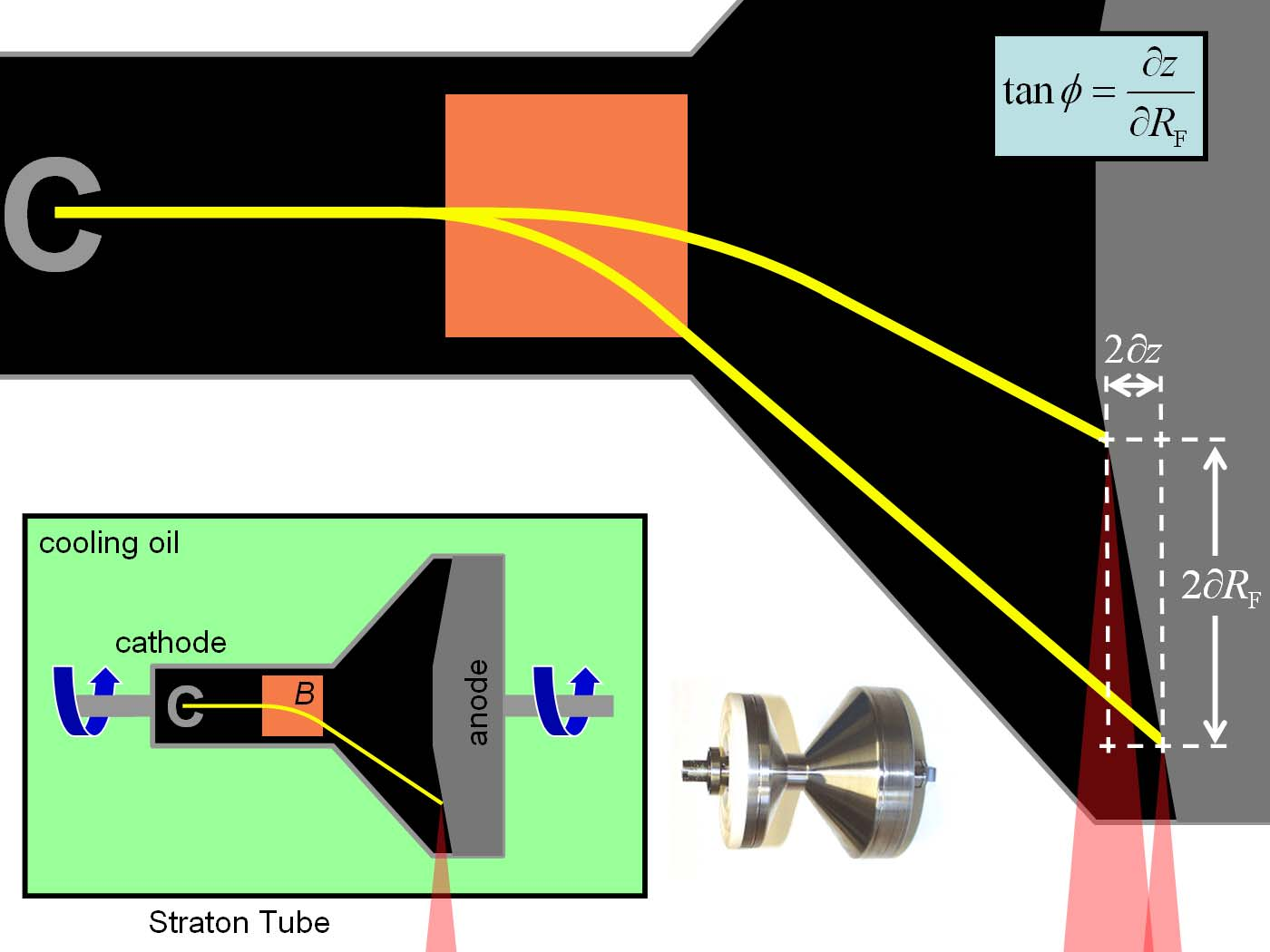}
    \caption[Flying focal spot]{A flying focal spot is generated by deflecting the beam.
        This results in a change in z-position as well a change in radial position
        \ie the distance to the rotation center. Graphic from~\cite{kachelriess2006flying}.
    \label{fig_focal_spot_deflection_kachelriess}}
\end{figure}

\begin{figure}
    \centering
    \includegraphics[width=.5\textwidth]{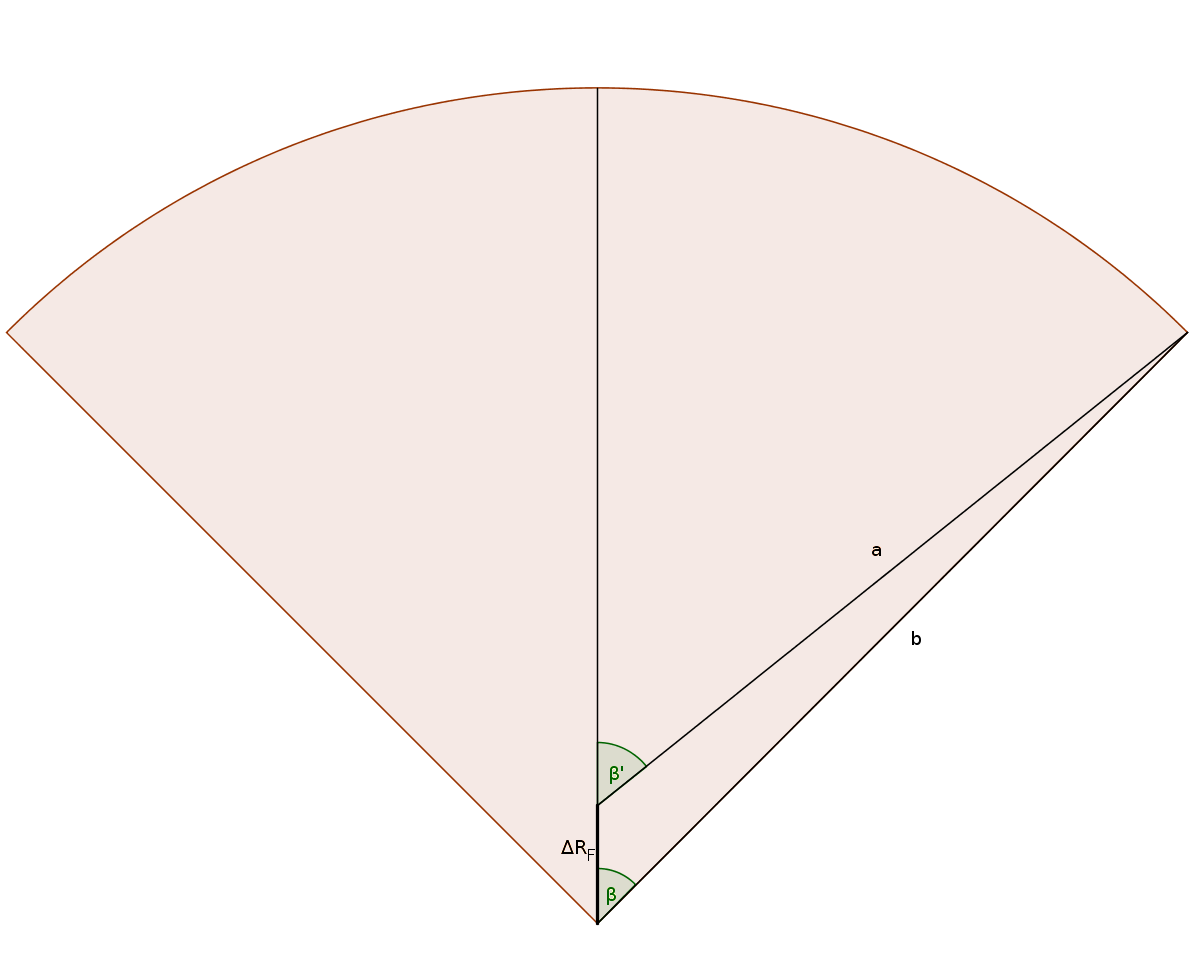}
    \caption[Radial Correction of beams length]{The true ray $a$ with angle $\beta'$ is approximated by a virtual ray $b$ with angle $\beta$.
    \label{fig_radial_correction}}
\end{figure}
\begin{figure}
    \centering
    \includegraphics[width=.4\textwidth]{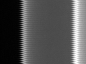}
    \includegraphics[width=.4\textwidth]{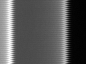}
    \includegraphics[width=.4\textwidth]{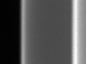}
    \includegraphics[width=.4\textwidth]{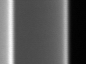}
    \caption[Interleaved frames]{Left a detail of frame $8500$ interleaved with frame $8501$, top without any correction bottom with radial distance correction, right the same for frames $7820- 7821$. The correction works on the right and left side, shifting one projection against the next would not.}
    \label{fig_frame_with_radial_correction}
\end{figure}

\section{GPU Radon implementation} \label{gpu_radon}
For a GPU-parallelization each thread is assigned one pixel of the image and calculates the value of the discrete gradient and divergence after reading in values at neighboring pixels.
In a similar fashion the linear back projection will only perform read operation on the sinogram data and thus has no write conflicts.
For the parallelization of the Radon transform again each thread is assigned one pixel and calculates its projection in the sinogram.
To resolve the resulting write conflict, different thread are assigned different angles to ensure that only one thread writes to any one column of the sinogram at a given time.
Cycling through the angles until every thread performed its operation on every angle - where threads pause if not enough angles are given - completes the operation.
Additionally each thread block is assigned its own sinogram to write to, and in a latter step summing these up to give the final result.

\FloatBarrier

\nocite{*}
\printbibliography              

\end{document}